\documentclass[a4paper,11pt,psamsfonts]{amsart}
\usepackage[dvips]{graphics}
\usepackage{rotating}
\usepackage[in,plain]{fullpage}

\usepackage{amssymb}
\usepackage{amstext}

\usepackage[authoryear]{natbib}
\bibpunct{(}{)}{;}{a}{,}{,}

\newcommand{\calF}{\mathcal{F}}
\newcommand{\fdd}{\stackrel{fdd}{=}}
\newcommand{\al}{\alpha}

\newcommand{\de}{\delta}
\newcommand{\la}{\lambda}
\newcommand{\si}{\sigma}
\newcommand{\ka}{\kappa}
\newcommand{\eps}{\varepsilon}

\newcommand{\Z}{\mathbb{Z}}
\newcommand{\R}{\mathbb{R}}
\newcommand{\Ex}{\mathbb{E}}
\newcommand{\Pb}{\mathbb{P}}

\newcommand{\Var}{\text{Var}}
\newcommand{\convergedist}{\stackrel{d}{\to}}
\newcommand{\equaldist}{\stackrel{d}{=}}
\newcommand{\Zbar}{\bar{Z}}
\newcommand{\Tbar}{\bar{T}}
\newcommand{\Wbar}{\bar{W}}

\newtheorem{thm}{Theorem}
\newtheorem{lem}[thm]{Lemma}
\newtheorem{cor}[thm]{Corollary}
\newtheorem{rem}[thm]{Remark}

\begin{document}
\sloppy

\title{LOOKING FOR CONTINUOUS LOCAL MARTINGALES \\ WITH THE CROSSING TREE (Working Paper)}

\date{November 27, 2009}

\author{Owen D. Jones}
\address{University of Melbourne\\ Parkville, VIC 3010\\ AUSTRALIA}
\email{O.D.Jones@ms.unimelb.edu.au}

\author{David A. Rolls}
\address{University of Melbourne\\ Parkville, VIC 3010\\ AUSTRALIA}
\email{drolls@unimelb.edu.au}


\begin{abstract}
We present statistical tests for the continuous martingale hypothesis. That is, whether an observed process is a continuous local martingale, or equivalently a continuous time-changed Brownian motion. Our technique is based on the concept of the crossing tree. Simulation experiments are used to assess the power of the tests, which is generally higher than recently proposed tests using the estimated quadratic variation (i.e., realised volatility). In particular, the crossing tree shows significantly more power with shorter datasets. We then show results from applying the methodology to high frequency currency exchange rate data. We show that in 2003, for the AUD-USD, GBP-USD, JPY-USD and EUR-USD rates, at small timescales (less than 15 minutes or so) the continuous martingale hypothesis is rejected, but not so at larger timescales. For 2003 EUR-GBP data, the hypothesis is rejected at small timescales and some moderate timescales, but not all.
\end{abstract}

\maketitle

\newpage
\tableofcontents

\newpage
\section{Introduction}
Time-changed Brownian motions have been proposed as models where so-called `volatility clustering' or `intermittency' is observed, in particular in finance but notably also in turbulence and telecommunications.
Models that incorporate time-changed Brownian motion (possibly after taking logs and removing drift) include, for example, stochastic volatility models \citep{HW87}; infinitely divisible cascading motion \citep{CRA03} and fractal activity time geometric Brownian motion  \citep{Hey99}.
In what follows we consider the question of testing whether or not a given process $X$ can be considered to be of the form $B \circ \theta$ where $B$ is Brownian motion and $\theta$ is a {\it continuous} non-decreasing process, possibly dependent on $B$.

From the theory of martingales one can obtain a hierarchy of time-changed Brownian motions.
In what follows we use the terminology {\it subordinator} for a non-decreasing process with stationary independent increments, which is thus a pure-jump L\'evy process, and {\it chronometer} for a general non-decreasing process.
We always take a `time-change' to be with respect to a non-decreasing process, possibly dependent on the past but not on the future.
Let $B$ stand for Brownian motion and $\theta$ for a chronometer, defined on the same filtration $\{ \calF_t \}$, then:
\begin{enumerate}
\item (\citet{monroe_78}) We can write $X \fdd B \circ \theta$ iff $X$ is a semimartingale.

\item  (\cite{dambis65}, \cite{DS65}, \citet{ks91}, \citet[Proposition V.1.5]{rny} We can write $X \fdd B \circ \theta +X_0$ with $\theta$ continuous iff $X$ is a continuous local martingale, in which case $\theta \fdd \langle X \rangle$ (the quadratic variation process).

\item (\citet{ocone_93}) We can write $X \fdd B \circ \theta$ with $\theta$ continuous and independent of $B$ iff $X$ is a continuous local martingale such that for any predictable $H$ with $|H| = 1$, $\int_0^\cdot H(s) dX(s) \fdd X$.

\item (\citet{CS02}) We can write $X \fdd B \circ \theta$ with $\theta$ a subordinator independent of $B$ iff $X$ is a L\'evy process with 0 drift and symmetric L\'evy measure $\nu(A) = \int_A q(x) dx$ where $q(\sqrt{x})$ is completely monotone.
\end{enumerate}
From the results of Dambis, Dubins \& Schwarz, and Revuz \& Yor, testing that $X = B \circ \theta$ for $\theta$ a continuous chronometer is equivalent to testing whether or not $X$ is a continuous local martingale.
Note that a time-changed Brownian motion can be continuous even though the chronometer is discontinuous.
This is the class of continuous semimartingales that are not local martingales, which includes for example Brownian motion with drift; the Ornstein-Uhlenbeck process (the Vasicek model) and Feller's square root process (the Cox, Ingersoll and Ross model).
When we write `continuous time-changed Brownian motion', $X = B \circ \theta$, we mean that $\theta$ (and thus $X$) is continuous.

There is a substantial literature about the problem of testing whether or not a {\em discrete time} process has stationary martingale differences, generally in the context of model verification.
A seminal paper in this area is that of Bierens \citeyearpar{Bie84} and a useful discussion can be found in Dominguez and Lobato \citeyearpar{DL01}.
It should be noted that many authors who claim to test if a discrete time series has stationary martingale differences actually test the less restrictive hypothesis that the differences are stationary and uncorrelated; most tests for zero correlation can be traced back to those of Box and Pierce \citeyearpar{BP70}, Durlauf \citeyearpar{Dur91} or Lo and MacKinlay \citeyearpar{LMacK88}.
The present paper takes a different approach to the discrete time literature, in part because we do not include stationary increments in our null hypothesis, but more fundamentally because there are differences between the behaviour of continuous and discontinuous martingales which clearly do not arise in the discrete time case.

In the econometrics literature the hypothesis that a discrete timeseries has stationary martingale differences is called the {\em martingale hypothesis}.
By analogy, we will call the hypothesis that a continuous process is a continuous local martingale the {\em continuous martingale hypothesis}.
We are careful to distinguish between the continuous martingale hypothesis and the less restrictive hypotheses that a process is a local martingale (possibly discontinuous) or a semimartingale, continuous or otherwise (Brownian motion with a possibly discontinuous time change).

Recently there has been a lot of interest in using high frequency financial data to estimate quadratic variation (realised volatility) as an estimator of `integrated variance'. The goal is to view financial returns (e.g. currency, metal and stock index spot and futures prices) in `financial time' rather than `calendar' time. (See both \citet{maas08} and \citet{mcaleer08} and the papers therein for a review). In the context of foreign exchange rates, \citet{Andersen2000a} showed daily exchange rate returns normalised by daily realised volatilities (formed by summing 30 minute volatilities) look approximately Gaussian. Sometimes formal statistical testing that data is from a time-changed Brownian motion has been included (e.g., Andersen et al., \citeyear{ABDL03}; Park \& Vasudev, \citeyear{PV05}; Peters \& de Vilder, \citeyear{peters06}; Andersen et al., \citeyear{andersen07}). To test the hypothesis that $X = B \circ \theta$, where $\theta$ is a continuous chronometer, one can estimate $\theta$ from the observed quadratic variation and then test whether the increments of $X \circ \hat\theta^{-1}$, look like the increments of Brownian motion. \citet{PV05} (and more recently \citet{andersen07}) also include empirical power estimates for simulated processes, which facilitate direct comparisons with other approaches. Park \& Vasudev (\citeyear{PV05}) is distinguished as the only systematic study of the power of the quadratic variation approach.  Non-quadratic variation tests have also been proposed using moments by An\'{e} \& Geman (\citeyear{ane00}) and using a combination of excursions from zero and local time at zero by Guasoni (\citeyear{Guasoni04}).
 
This quadratic variation/realised volatility approach faces several challenges. First, a good estimate of the quadratic variation is required. It has been shown that market microstructure noise can lead the sample quadratic variation to be a biased estimator, with bias that grows linearly as the sampling frequency increases \citep{mcaleer08}. Second, for statistical testing, one must make an arbitrary choice for the increment length in $X \circ \hat\theta^{-1}$. As we show, this choice can affect the results of the test.

Our test of the continuous martingale hypothesis is based on the recently introduced concept of the crossing tree \citep{JS04} and a characterisation of continuous local martingales in terms of it \citep{jonesrolls08a}. It is closely tied to the idea of first passage times and avoids the need to estimate the quadratic variation, and the problems inherent in that exercise. Using simulation it shows discriminatory power as good, or better, than the quadratic variation method for a range or processes. Unlike alternatives, our approach has the advantage of testing a range of timescales. Thus, our method picks out the scales at which the continuous martingale hypothesis is plausible. We also report results from testing foreign exchange rate tick data from 2003 for five rates (AUD-USD, GBP-USD, JPY-USD, EUR-USD and EUR-GBP). We show that for timescales tested on the order of about fifteen minutes or less the continuous martingale hypothesis is rejected, while it is not rejected at longer timescales.

In Section \ref{quadvar.sec} we describe our implementation of the quadratic variation approach, which provides a basis of comparison for our crossing tree approach. Further, our implementation sidesteps making an arbitrary choice of increment length. Notably, the choice used by Park \& Vasudev (\citeyear{PV05}) seems to cause the power of their test to suffer. In Section \ref{xingtree.sec} we describe the crossing tree and restate from \citet{jonesrolls08a} the characterisation of a continuous local martingale. We use this characterisation in combination with a suite of tests to test the continuous martingale hypothesis. In Section \ref{power.sec} we show results from a simulation study which shows our crossing tree approach is generally more powerful than the quadratic variation approach. We also show the choice of parameter inherent to the quadratic variation-based test can make a large difference to the power of the test. Finally in Section \ref{FX.sec} we give the results from applying the crossing tree test to various foreign exchange rates. We give conclusions in Section \ref{conc.sec}. Matlab code for creating the crossing tree for an observed process is available at \verb=www.ms.unimelb.edu.au/~odj=.

\clearpage
\section{Quadratic Variation-based Test} \label{quadvar.sec}

By Dambis \citeyearpar{dambis65} and Dubins and Schwarz \citeyearpar{DS65} (see also Karatzas and Shreve \citeyearpar[Theorem 4.6]{ks91}) we can write $X -X_0\fdd B \circ \theta$ with $\theta$ continuous iff $X$ is a continuous local martingale, in which case $\theta \fdd \langle X \rangle$ (the quadratic variation process). Thus, with the inverse
\[T_t=\inf \{s \geq0 : \langle X \rangle _s >t \},\]
one tests if $X$ is a continuous local martingale by testing if 
\begin{equation} \label{X_timeinvert.eqn}
Y_{t}=X_{T_t}
\end{equation} 
is standard Brownian motion.  Peters and De Vilder \citeyearpar{peters06} discuss the technique and apply it to S \& P 500 returns. Park and Vasudev \citeyearpar{PV05} go further, providing a number of results for simulated data, futures prices and currency spot prices. In this section we discuss our implementation of the quadratic variation test. There are two parts to the test, estimation of the quadratic variation, and testing for Brownian motion, and both will be discussed.

For ease of comparison our notation closely follows that of Park and Vasudev \citeyearpar{PV05}. Suppose $\{X_1,\ldots,X_n\}$ are $n$ consecutive values from a time interval $(0,n\de]$ with time interval $\de$ between observations.  The quadratic variation is estimated as
\[ \langle X \rangle_t ^\de  = \sum_{k \geq 2:k\de \le t} \left(X_{k\de} - X_{(k-1)\de}  \right)^2, \,\, t \ge 2\de.\]
Define its generalised inverse as $T_t^\de=\inf \{s \geq0 : \langle X\rangle _s^\de  >t \}$ and introduce the analogue of (\ref{X_timeinvert.eqn}) with $Y_t^\de = X_{T_t^\de}$. For intervals of length $\Delta$ form the sequence
\begin{align}
\label{quadvarincr.eqn}
 Z_j^\de =\frac{1}{\sqrt{\Delta}}\left(Y_{(j+1)\Delta} - Y_{j\Delta}  \right), \, j=1,2,\ldots N-1
\end{align}
where $N$ is such that $T^\delta_{N\Delta}=n\delta$. Under the null hypothesis that the observations are consecutive values from a continuous local martingale, $\{Z_j^\de\}$ is approximately an i.i.d sequence of $N(0,1)$ values.  What remains is to choose a value for $\Delta$ and then test that $\{Z_j^\de\}$ is an i.i.d sequence with a common $N(0,1)$ distribution. Note that $\Delta$ selection is not an issue for estimating the quadratic variation, but rather for using it to normalise and recover Brownian motion.

\subsection{$\Delta$ Selection}
There is freedom in the choice of $\Delta$. Clearly, smaller $\Delta$ allows larger $N$ and so more values $Z_j^\de$ to test. However, small $\Delta$ introduces interpolation error as $\Delta \rightarrow 0$. Since the quadratic variation is estimated with some granularity through $\de$, under the null hypothesis
\[\frac{Y_{(j+1)\Delta} - Y_{j\Delta}}{\sqrt{\Delta}}\sim  \frac{B_{(j+1)\Delta}-B_{j\Delta}+O(\sqrt{\de}) }{\sqrt{\Delta}}\]
and small $\Delta$ inflates the error term.
Park and Vasudev \citeyearpar{PV05} provide a lower bound for $\Delta$ as a function of $N$ and $\de$ but it is of limited use since it depends on unknown constants. Effectively they search for $\Delta$ so that the empirical distribution of $\{Z_j^\de - \bar{Z}_N^\de\}$ with $\bar{Z}_N^\de= \sum_{j=1}^{N-1}Z_j^\de /(N-1)$ is closest to the $N(0,1)$ distribution. We found this unreasonable in practice because it significantly reduces the power of the test technique. Rejecting the null hypothesis when the sequence $\{Z_j^\de,\, j=1,2,\ldots N-1\}$ is not sufficiently like $N(0,1)$ is at odds with choosing $\Delta$ to make the sequence most like a Gaussian sequence.

Peters and de Vilder \citeyearpar{peters06} use a slightly different technique for $\Delta$ for testing their S \& P 500 data. Effectively they find successive times $t_i$ at which to evaluate $X$, where the increments $\langle X \rangle ^\de_{t_i}- \langle X \rangle^\de_{t_{i-1}}$ of the quadratic variation are greater than some arbitrarily chosen lower bound, but are minimal. 

Our approach is to find the mean quadratic variation increment 
\[S=\frac{\langle X \rangle_{n\de} ^\de-\langle X \rangle_{2\de} ^\de}{n-2} \]
and choose $\Delta$ by $ \Delta=cS$ for some constant $c$. Instead of trying to decide on the `right' value for $c$, we show test results for a range of values and use the ones most favourable to the quadratic variation approach (e.g. use the highest rejection rates when the null hypothesis is not satisfied.) Simulation results (shown below) illustrate that a range of values for $c$ is generally acceptable. Since our goal is comparison with the crossing tree method, using a range of values is most forgiving to the quadratic variation-based method, but is a high standard for another technique to meet. We leave unanswered the question of how to select one value of $\Delta$ to perform the test but feel this is a drawback to the quadratic variation test.
  
\subsection{Statistical Tests}
Testing if $\{Z_j^\de\}$ is an i.i.d sequence with a common $N(0,1)$ distribution involves two questions. Are the values from an independent sequence? Do the values have a common  $N(0,1)$ distribution? Following Park and Vasudev \citeyearpar{PV05} we ignore the question of independence and test the distribution using the two-sided Kolmogorov-Smirnov (KS) statistic, the Cram\'{e}r-von Mises (CVM) statistic and a standardised mean (SM) statistic
\[T=\frac{1}{\sqrt{N-1}}\sum_{i=1}^{N-1}Z_i^\de =\frac{Y_{N\Delta}-Y_{\Delta}}{\sqrt{\Delta(N-1)}} \]
which is $N(0,1)$ under the null hypothesis. Park and Vasudev provide convergence results that show when these tests are applied to the $\{Z_j^\delta\}$ data, distributions of the test statistics converge to $N(0,1)$ data as $\de \rightarrow 0$, $N \rightarrow \infty$. Results from applying these tests are reported in Section \ref{power.sec} for simulated data and Section \ref{FX.sec} for tick-by-tick foreign exchange data.

Peters and de Vilder \citeyearpar{peters06} address the issue of testing distributions using the Kolmogorov-Smirnov and two other tests. Unlike Park and Vasudev, they also test for independence in the sequence. We haven't tried the additional tests they use, in part because they consider only one timeseries of S \& P 500 data and our goal is also a more systematic comparison with other's results for simulated data. And while the Kolmogorov-Smirnov statistic is known to be sensitive to differences in the mean, the Cram\'{e}r-von Mises test has been the more powerful test in our simulations and we think it provides a good basis for comparison.

\subsection{Type 1 Error}
Under the null hypothesis, using tests with a significance level of 5\% we expect about 5\% of  simulated sample paths to be rejected. To confirm the performance of our implementation we simulated a number of sample paths and counted the number of paths rejected using each of the three statistical tests. Figure \ref{PVBM.fig} (left) shows the results when 1,000,000 independent paths of Brownian motion were simulated. Each path is simulated as $X_k=\sum_{i=1}^k Z_i$, $k=1,\ldots,1250$ and $X_0=0$, where $\{Z_1,\ldots,Z_{1250}\}$ is an i.i.d. sequence of $N(0,1/250)$ values. This corresponds to simulating Brownian motion on $[0,5]$ at the times $k\delta$, $\de=1/250$. For a sense of the confidence intervals, using the Binomial($M$,0.05) distribution, a 95\% confidence interval would use $\pm 0.04$\% for $M=10^6$ paths and $\pm 0.4$\% for $M=10^4$ paths. Figure \ref{PVBM.fig} (right) shows similar results for paths of length $n=5000$. The standardized mean statistic appears extremely good here, rejecting almost exactly 5\%. The two other tests appear comparable. Note that values of $c$ very close to zero seem to give slightly higher rejection rates for Kolmogorov-Smirnov and Cram\'{e}r-von Mises tests. This is noteworthy because for virtually all other processes considered, the SM test will have little or no power and these other two tests will be the main ones to consider. Thus, results from those tests in which $c$ is small must be regarded with some caution. For $c=20$, $\{Z_j^\de\}$ has length 58, on average across all the sample paths. So the statistical tests are using that many values. For $c=140$, the tests are using 7 values on average. Ordinarily one would expect smaller error when testing more data,  but small $c$ is equivalent to small $\Delta$ (in relation to $\de$) in forming the increments (\ref{quadvarincr.eqn}) and granularity in the sample quadratic variation becomes relevant.

\begin{figure}[ht!]
\resizebox{!}{3.0in}{\includegraphics{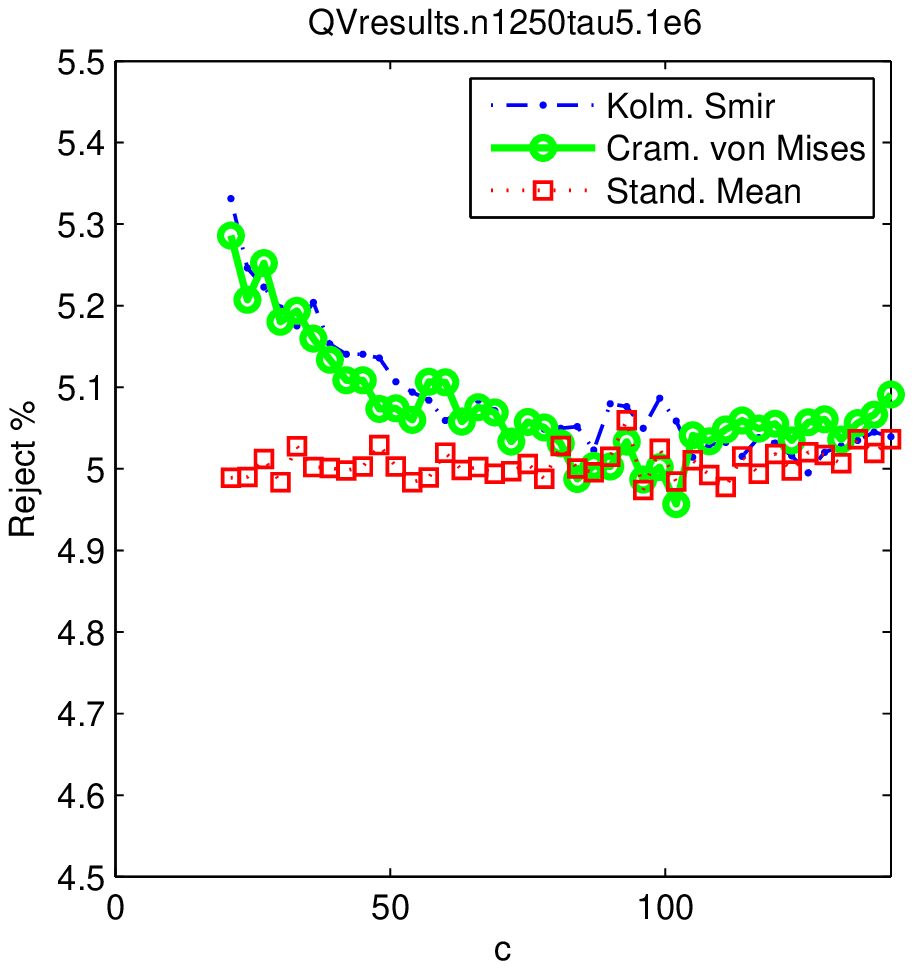}}\resizebox{!}{3.0in}{\includegraphics{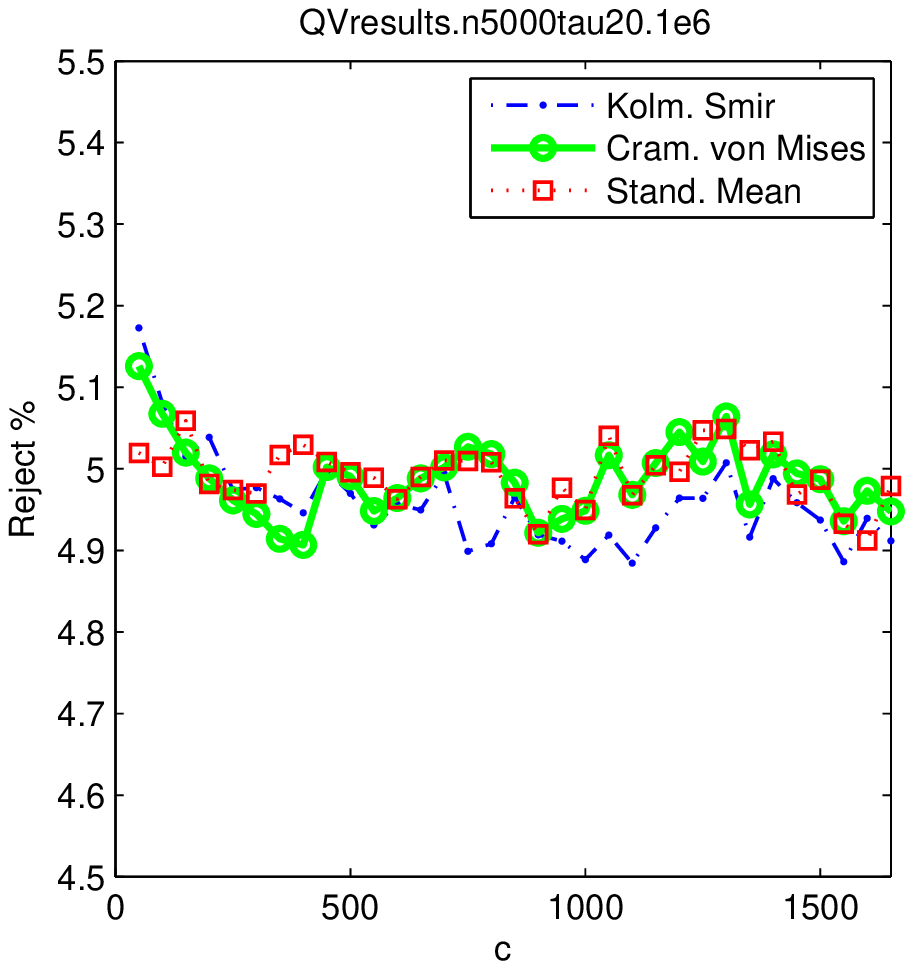}}
\caption{Quadratic variation empirical type 1 error rates from 1,000,000 simulated sample paths of Brownian motion with length 1250 (left) or 5000 (right) using tests at level 5\%.
Here 95\% confidence intervals are of the order of $\pm 0.04$\%.\label{PVBM.fig}}
\end{figure}

\begin{figure}[h!]
\resizebox{3in}{!}{\includegraphics{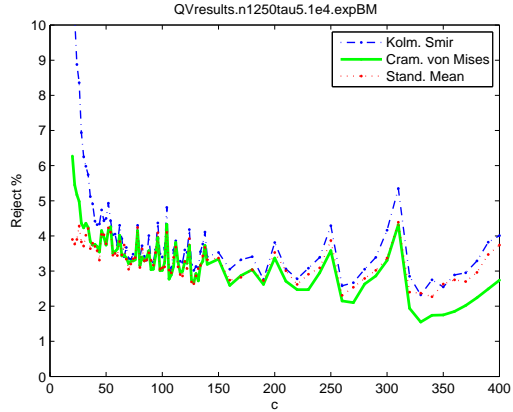}}
\caption{Quadratic variation empirical type 1 error rates from 10,000 simulated sample paths of exponential Brownian motion with length 1250 using tests at level 5\%.
Here 95\% confidence intervals are of the order of $\pm 0.4$\%.}
\label{BMlonghi.fig}
\end{figure}

Figure \ref{BMlonghi.fig} shows results corresponding to an exponential martingale $X_t=\exp(B_t-t/2)$, where $B_t$ is standard Brownian motion, simulated on $[0,1250]$ at equally spaced times with interval length $\de=1/250$. Here, 10,000 independent paths were used. At the left, $c=20$ corresponds to testing 61 values on average, while $c=400$ at the right corresponds to testing two values. The repeating peaks arise from coarseness in the number of values $\{Z_j^\de\}$ tested. The peaks at $c=$ 200, 250, and 310 correspond to the largest $c$ for which $\{Z_j^\de\}$ has length 5, 4 and 3, averaged across all sample paths, respectively. (Aside: there is very little variability in the number tested across the paths for fixed $c$.)  The cumulative sample quadratic variation typically shows a steep upward portion for small $t$ and a flatter portion for larger $t$, connected by a `knee'. This has the effect that larger $\Delta$ makes $N\Delta$ larger for the same $N$. The presence of the knee magnifies the change in $T_{N\Delta}^\de$ for different $\Delta$ to such an extent that the sawtooth behavior disappears if the same tests are applied to $\{Z_j^\de,\,j=1,N-2\}$, thus dropping $Z_{N-1}^\de$. Since a 5\% rejection rate is expected under the null hypothesis, these results further emphasize that while values of $c$ too small should be avoided, values of $c$ too large should also be avoided because of a lack of data. Values of $c$ between 20 and 140 seem acceptable for length 1250. (For length 5000, values of $c$ between 20 and 400 seems acceptable.) Within this range, the rejection rates are between 3.5\% and 5\% for all three tests. Park and Vasudev report rejection rates of 4\% (KS), 5\% (CVM) and 4.4\% (SM) so the results are comparable. 

\section{Crossing Tree-based Test} \label{xingtree.sec}

\subsection{Characterizations of BM and CLM using the Crossing Tree}
The crossing tree was introduced by Jones and Shen \citeyearpar{JS04} in the context of self-similar processes.
In this section we describe the crossing tree then show that it can be used to give a characterization of Brownian motion time-changed using a continuous chronometer.
Using this characterization we give various tests for the continuous martingale hypothesis.

Fix $\de > 0$.
Our definitions depend inherently on $\delta$, but as it remains fixed throughout we will not include it in our notation.
Let $X$ be a continuous process, $X(0) = 0$, then for all $l \in \Z$ we define crossing times (more precisely first passage times) by putting $T^l_0  :=  0$ and
\[
T^l_j := \inf \{ t >  T^l_{j-1} \,:\, |X(t)-X(T_{j-1}^l)|= 2^l \de \}.
\]
By a level $l$ crossing (equivalently size $\de 2^l$ crossing) of the process $X$ we mean a section of the sample path between two successive crossing times $T^l_{j-1}$ and $T^l_{j}$ plus the starting time and position of the crossing, $T^l_{j-1}$ and $X(T^l_{j-1})$.
Let $C^l_j$ be the $j$-th crossing of size $\de 2^l$.
There is a natural tree structure to the crossings, as each crossing of size $\de 2^l$ can be decomposed into a sequence of `subcrossings' of size $\de 2^{l-1}$.
We identify vertices of the tree with crossings and link each level $l$ crossing with its level $l-1$ subcrossings.
This is illustrated in Figure \ref{xtree.fig}.
Define the crossing length $W^l_k := T^l_{k}-T^l_{k-1}$;
orientation $\al^l_k := \text{sgn}(X(T^l_{k+1})-X(T^l_k))$; and the
number of subcrossings $Z^l_k$.

Subcrossing orientations come in pairs, either $+-$, $-+$, $++$ or $--$, corresponding respectively to excursions up and down and direct crossings up and down.
The subcrossings of a crossing can be broken down into some variable number of excursions, followed by a single direct crossing, where the orientation of the direct crossing is the same as the orientation of the crossing.
Let $V^l_j = 0$ if the $j$-th level $l$ excursion is up ($+-$) and $V^l_j = 1$ if it is down ($-+$).

\begin{figure}[htb!]
\begin{center}
\includegraphics[width=6cm]{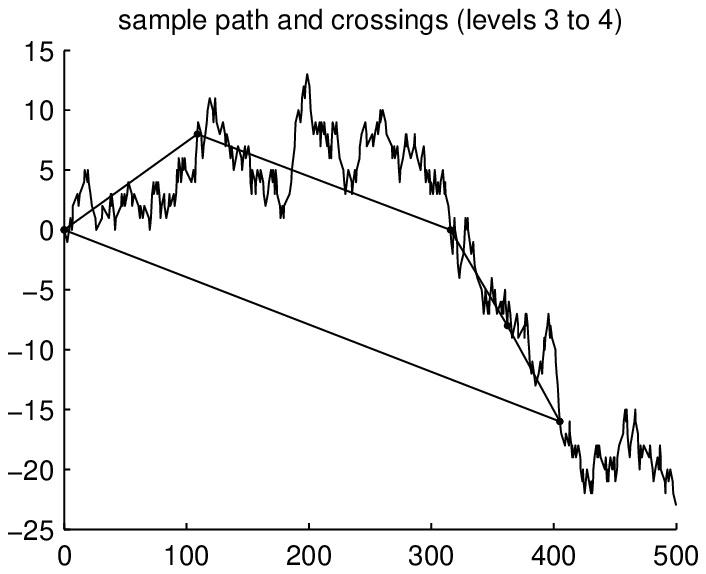}%
\includegraphics[width=6cm]{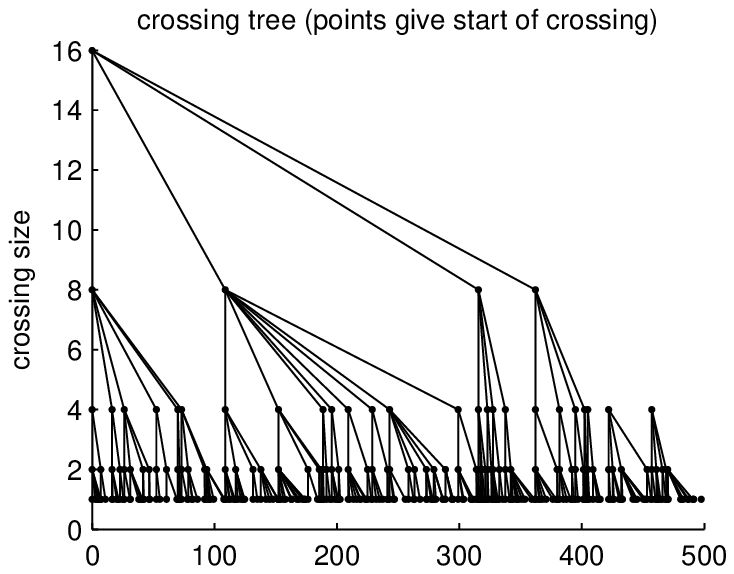}
\caption{The crossing tree associated with a continuous sample path.
Here $\de = 1$. In the left frame, for $l = 3$ and $4$, we have joined the points $(T^l_j, X(T^l_j))$;
we see that the single level 4 crossing can be decomposed into a sequence of four level 3 crossings.
In the right frame we have plotted the points $(T^l_j, \de 2^l)$ for all $l, j \geq 0$, then,
identifying crossing $C^l_{j}$ with its starting time $T^l_{j-1}$,
we joined each point to the points corresponding to its subcrossings.} \label{xtree.fig}
\end{center}
\end{figure}

\begin{thm} \label{bmchar}
Brownian motion is the unique continuous process $B$ for which:
\begin{enumerate}
\item[{BM0}] $B(0) = 0$, $\Var [B(1)] = 1$;
\item[{BM1}] For each $l$ the $W^l_k$, $k = 1, 2, \ldots$, are i.i.d.\ with Laplace transform $\Ex \exp(-\la W^l_k) = 1/\cosh(2^{-l+1/2}\la^{1/2})$;
\item[{BM2}] The $Z^l_k$ are i.i.d.\ for all $l$ and $k$, with $\Pb( Z^l_k = 2 i) = 2^{-i}$, $i = 1, 2, \ldots$;
\item[{BM3}] The $V^l_j$ are i.i.d.\ for all $l$ and $j$, with $\Pb( V^l_j = 0) = \Pb( V^l_j = 1) = 1/2$.
\end{enumerate}
\end{thm}
\begin{proof}
This characterisation of Brownian motion, in terms of its crossings, is based on ideas used in the construction of Brownian motion on a nested fractal.
To see these ideas in their original context, the reader is referred to Barlow \& Perkins \citeyearpar{bp_88} and Barlow \citeyearpar{Bar93}.

Given a Brownian motion $B$, it follows from the strong Markov property that the $Z^l_k$  are all independent.
Moreover since Brownian motion is statistically self-similar, they are identically distributed.
The distribution of $Z^l_k$ is just that of the time taken for a simple random walk on $\Z$ to hit $\pm 2$, starting at 0, whence BM2 follows.

From the strong Markov property we also have that for each $l$, the $W^l_k$ are i.i.d., and self-similarity shows that $W^l_k \equaldist 4^l W^0_1$.
The distribution of $W^0_1$ can be found using martingale techniques (see for example Burq \& Jones \citeyear{BJ08}), which gives BM1.

To see that BM3 holds, consider an up-crossing: the orientations of its subcrossings are the same as the steps taken by a simple random walk on $\Z$, starting at 0 and conditioned to hit 2 before $-2$.

Condition BM0 just specifies the scaling of the process.

Now suppose that we are given a continuous process $B$ satisfying conditions BM0--BM3.
Put $X^l(k) = B(T^l_k)$, and for $l < m$ let $N^{l,m}$ be the first time $X^l$ hits $X^m(1)$ (so $N^{l,l+1} = Z^{l+1}_1$).
Conditions BM2 and BM3 specify the distribution of $\{ X^l(0), \ldots, X^l(N^{l,l+1}) \,|\, X^{l+1}(0), X^{l+1}(1) \}$, and thus by induction the distribution of $\{ X^l(0), \ldots, X^l(N^{l,m}) \,|\, X^m(0), X^{m}(1) \}$, for any $l < m$.
(In the terminology of \citet{bp_88}, the random walks $X^l$, $l \in \Z$, are {\em nested}.)
It is straightforward to show that we get precisely the same laws for the subcrossing numbers and orientations if we take the simple symmetric random walk on $\de 2^l \Z$, started at 0 and run it until it hits $\pm \de 2^m$.
(See \citet{Jon96} for the explicit calculations.)
That is, $\{ X^l(0), \ldots, X^l(N^{l,m}) \,|\, X^m(0), X^{m}(1) \}$ is a simple symmetric random walk on $\de 2^l \Z$, started at 0 and conditioned to hit $X^m(1)$ before $-X^m(1)$.

From BM2 and BM3 we have that for any $m \in \Z$ and $n \in \Z_+$,
\begin{eqnarray*}
\Pb(X^m(1) = \de 2^m \,|\, X^m(0) = 0)
&=& \frac 12 \Pb(X^{m+1}(1) = \de 2^m \,|\, X^{m+1}(0) = 0) + \frac 14 \\
&=& \frac 1{2^n} \Pb(X^{m+n}(1) = \de 2^{m+n} \,|\, X^{m+n}(0) = 0) + \sum_{i=1}^n \frac 1{2^{1+i}}
\;=\; \frac 12.
\end{eqnarray*}
Thus, removing the conditioning on $X^m(1)$, $X^l(k)$ is indistinguishable from a simple symmetric random walk for $k = 0, \ldots, N^{l,m}$.
But $N^{l,m} \geq 2^{m-l}$, so sending $m \to\infty$ we see that $X^l$ is just a simple random walk on $\de 2^l \Z$.

It is well known that as $l \to -\infty$, $4^l X^l$ converges a.s.\ to a Brownian motion, $X$ say.
Moreover \citet{Ham95} shows that, up to a constant scaling, the crossing times of $X$ have the distribution given by BM1, so $X$ is just a scaling of $B$.
Thus from BM0, $B$ is a standard Brownian motion.
\end{proof}

\begin{rem}
\label{char.rem}
Our definition of the crossing tree considers the process when it hits new points on the lattice $X_0 + \de 2^l \Z$, for all levels $l \in \Z$.
We can just as easily consider lattices $a + \de 2^l \Z$, by the simple modification of putting $T^l_0 = \inf \{ t \geq 0 \,:\, X(t) \in a + \de 2^l \Z \}$.
Similarly, our characterisation of Brownian motion can be generalised to allow for lattices centred at an arbitrary point $a$.
The proof is essentially the same, but does require more care with the nested random walks $X^l$, as per \citet{bp_88} Theorem 2.14.
\end{rem}

Clearly the $V^l_j$ and $Z^l_k$ are invariant under a continuous time-change, so a continuous local martingale must satisfy BM2 and BM3.
We show below that these properties characterise a continuous local martingale, up to a shift at time 0.
To do so we need to know a little more about the crossing times of Brownian motion, which the following lemma provides.

\begin{lem} \label{extremeW.thm}
Let $\{P(n)\}_{n=0}^\infty$ be a supercritical Galton-Watson branching process, with $P(0) = 1$.
That is, $P(n)$ is the population size at generation $n$ and $\mu := \Ex P(1) > 1$.
Suppose that $\si^2 := \Var P(1) < \infty$ and let $L^n_k$ be the normed limit of the tree branching from the $k$-th individual in generation $n$.
(So $\Ex L^n_k = 1$, $\Var L^n_k = \si^2/(\mu^2-\mu)$ and $L^0_1 = \lim_{n\to\infty} \mu^{-n} P(n)$ a.s.\ and in mean square.)
Then
\[
\lim_{n\to\infty} \max_{0 \leq k \leq P(n)} \mu^{-n} L^n_k = 0 \mbox{ a.s.}
\]
\end{lem}
\begin{proof}
Our proof uses extreme order statistics for Galton-Watson trees (Pakes \citeyear{Pak98}).
From Chebychev's inequality we have $\Pb( L^n_k > x) \leq x^{-2} \Ex (L^n_k)^2 = x^{-2} (1 + \si^2/(\mu^2-\mu))$.
Thus we can find i.i.d.\ random variables $M^n_k$ such that for all $x$, $\Pb( L^n_k > x) \leq \Pb( M^n_k > x) \sim c_0 x^{-2}$.
The law of $M^n_k$ is in the domain of attraction of the extremal distribution $\exp(-x^{-2})$, that is, writing $M$ for a generic $M^n_k$,
\[
\lim_{z\to\infty} z \Pb( M > x\sqrt{c_0z}) = x^{-2}.
\]
Thus from Pakes \citeyearpar{Pak98} Theorem~4.1 we have, for $H_1(x) = \Ex \exp(-x^{-2}P(1))$,
\[
\lim_{n\to\infty} \Pb\left( \max_{0 \leq k \leq P(n)} \mu^{-n/2} M^n_k \leq x  \sqrt{c} \right) = H_1(x).
\]

Finally, $\max_{0 \leq k \leq P(n)} \mu^{-n} L^n_k$ is a strictly decreasing function of $n$ (since each $L^n_k$ can be written as a sum of terms $\mu^{-1} L^{n+1}_j$ for various $j$), so for any $\eps > 0$
\begin{eqnarray*}
\Pb\left( \lim_{n\to\infty} \max_{0 \leq k \leq P(n)} \mu^{-n} L^n_k \leq \eps \right)
&=& \lim_{n\to\infty} \Pb\left( \max_{0 \leq k \leq P(n)} \mu^{-n} L^n_k \leq \eps \right) \\
&\geq& \lim_{n\to\infty} \Pb\left( \max_{0 \leq k \leq P(n)} \mu^{-n} M^n_k \leq \eps \right) \\
&=& \lim_{n\to\infty} \Pb\left( \max_{0 \leq k \leq P(n)} \mu^{-n/2} M^n_k \leq \mu^{n/2} \eps \right) \\
&=& \lim_{x \to \infty} H_1(x)
\;=\; 1.
\end{eqnarray*}
Sending $\eps \to 0$ gives the result.
\end{proof}

\begin{cor} \label{clmchar}
A continuous process $X$ is a continuous time-change of Brownian motion, equivalently a continuous local martingale, if and only if
\begin{enumerate}
\item[{CLM0}] $X(0) = 0$;
\item[{CLM1}] The $Z^l_k$ are i.i.d.\ for all $l$ and $k$, with $\Pb( Z^l_k = 2 i) = 2^{-i}$, $i = 1, 2, \ldots$;
\item[{CLM2}] The $V^l_k$ are i.i.d.\ for all $l$ and $k$, with $\Pb( V^l_k =0) = \Pb( V^l_k =1)=1/2.$ 
\end{enumerate}
\end{cor}
\begin{proof}
The `if' part is clear, since $Z^l_k$ and $V^l_j$ are unaffected by a continuous time-change.

We show the `only if' part as follows.
Let $W^l_k$ be the crossing times of $X$.
Properties CLM0--CLM2 are enough for us to construct a Brownian motion $B$, with subcrossing family sizes $Z^l_k$, whose crossing times $\Wbar^l_k$ are obtained as normed limits of the Galton-Watson processes defined by the $Z^l_k$ (Hambly \citeyear{Ham95}).
By construction we have $B(\Tbar^l_k) = X(T^l_k)$, where $T^l_k = \sum_{j \leq k} W^l_j$ and $\Tbar^l_k = \sum_{j \leq k} \Wbar^l_j$.
Defining $\theta(T^l_k) = \Tbar^l_k$ we get, for $t = T^l_k$, $B(\theta(t)) = X(t)$.

Since $X$ is continuous, for any $t$ we can find a sequence $\{ k(t,l) \}_{l=-\infty}^{\infty}$ such that for all $l$, $t \in [T^l_{k(t,l)}, T^l_{k(t,l)+1})$.
We use this to extend $\theta$ to $[0, \infty)$: for $t \in [0, \infty)$ put $\theta(t) = \lim_{l\to -\infty} \Tbar^l_{k(t,l)}$.
The result now follows provided that $\theta$ is continuous, that is provided $\lim_{l\to -\infty} \Wbar^l_{k(t,l)} = 0$ for all $t$.
From Lemma \ref{extremeW.thm} we see that this holds for almost all sample paths of $B$.

Finally, by construction we have that $B(t)$ and $\theta(t)$ are $\calF_t$ measurable, where $\{ \calF_t \}$ is the filtration generated by $X$.
\end{proof}
\noindent Moreover, by Remark \ref{char.rem} this characterisation generalises to $X(0) = a$ for constant $a \in \R$. Also notice we do not need to assume the time change is independent of the Brownian motion.

\subsection{Small scale diffusive behaviour}
We show that for regular diffusions, at sufficiently small scales, the diffusion dominates the drift and these processes look like (continuous time-changed) Brownian motions. A key idea for these results is the \emph{scale} function (see Appendix \ref{simtech.sec}) which allows computing the probability $p_{\de}(x)$ that a size $\de$ crossing starting at $x$ is an up-crossing.

\begin{lem}
For a continuous strong Markov process $X$, if $p_\de(x)$ is a constant $\neq 0$ or 1 then the $\{ V^0_k \}$  are i.i.d.\ Bernoulli($1/2$).
\end{lem}
\begin{proof}
Excursions are equiprobable iff for all $x \in \de \Z$
\[
\Pb\left( X(T^0_{k+1}) = x + \de, \, X(T^0_{k+2}) = x \,|\, X(T^0_k) = x, \, X(T^0_{k+2}) = x \right) = \frac{1}{2}.
\]
That is,
\[
\frac{p_\de(x)(1-p_\de(x+\de))}{p_\de(x)(1-p_\de(x+\de))+(1-p_\de(x))p_\de(x-\de)} = \frac{1}{2},
\]
which clearly holds if $p_\de(x)$ is constant and non-degenerate.
If $p_\de(x)$ does not depend on $x$, then from the strong Markov property the crossing orientations $\{ \al^0_k \}_k$ and thus the excursions must be independent.
\end{proof}

An immediate consequence of this result is that CLM2 holds for any continuous time-change of Brownian motion with drift.
The next lemma shows that for a large class of diffusions, CLM1 and CLM2 hold approximately at small scales.
That is, at small scales, these diffusions looks like continuous local martingales.

\begin{lem}
Suppose $X$ is a continuous regular diffusion on some interval, with differentiable scale function $s$ and non-absorbing boundaries, then as $\de \to 0$, $\{ Z^1_k \}_{k=1}^{N(l)}$ and $\{ V^0_k \}_{k=1}^{N_V(l)}$ converge in finite dimensional distribution to i.i.d.\ sequences, with distributions 2$\cdot$Geometric$_1(1/2)$ and Bernoulli$(1/2)$ respectively.
\end{lem}
\begin{proof}
For any $x$ in the interior of the range of $X$ we have from (\ref{hitprob.eq}) that
\[
\lim_{\de \to 0} p_\de(x) = \frac{s'(x)}{2s'(x)} = 1/2.
\]
If $p_\de(x)$ does not depend on $x$, then from the strong Markov property the crossing orientations $\{ \al^0_k \}_k$ must be independent.
Thus we see that as $\de \to 0$, $\{ \al^0_k \}$ converges in finite dimensional distribution to an i.i.d.\ Bernoulli($1/2$) sequence.
That is, $\{ X(T^0_k) \}_k$ converges to a simple symmetric random walk.
The result now follows as per the proof of Theorem \ref{bmchar}.
\end{proof}

\subsection{Test statistics} \label{tests.subsec}
We test the continuous martingale hypothesis by testing properties CLM1 and CLM2 of Corollary \ref{clmchar}.
Formally, property CLM1 characterises the sample-path variation while CLM2 characterises the symmetry of the sample paths. To test CLM1 we need to test the distribution of the $Z^l_k$ and their independence. To test CLM2 we will test that for each level, amongst the up-down and down-up pairs of subcrossings, each has probability $1/2$ and they are independent. The last pair for any subcrossing is up-up or down-down, and is not included in the testing.

We suppose that we have points $\{(T_i,X_i),\,i=0,\ldots,M\}$, with  $X_i=X(T_i)$ as observations of the continuous process $X(\cdot)$ from time interval $[0, T]$. With this we can construct a continuous process on $[0, T]$ by connecting consecutive points by linear interpolation to yield
\[Y(0)=X(0),\, Y(t)=\frac{X_{k}-X_{k-1}}{T_k-T_{k-1}}(t-T_{k-1})+X_{k-1},\, t \in (T_{k-1},T_k]. \]
Then for a given  $\delta>0$ and $\de_0 \in \R$ we can construct at level 0 (i.e., size $\de$) crossing times $T^0_k \in [0, T]$ and corresponding crossing types $\al^0_k$. In practise we use $\delta=\text{median}_{k=1,\ldots,M}\{|X_k-X_{k-1}|\}$ for the level 0 crossing size, but this is arbitrary and not crucial since we test across multiple levels $2^l\de$. Since we also simulate crossing data it is helpful to use a definition that recovers exactly the crossing size in the data when it has that form.

For the choice of $\de_0$, although $\de_0=0$ seems natural, there are advantages to other choices.  Under the null hypothesis, if $\de_0$ is independent of $\{X(T_k^0),\,k>0\}$  the distribution and i.i.d. properties of  $\{V^l_k\}$ and $\{Z^l_k\}$ remain unchanged. For example, one could take $\de_0=X(0)$, since under the null hypothesis $X(t)-X(0)$ is a continuous local martingale and $X(0)$ is independent of the Brownian motion $\{X(\theta^{-1}(t)),\,t>0\}$ \citep[p. 178]{ks91}.

To improve the power of our test against stationary alternatives (e.g., Ornstein-Uhlenbeck processes) we use something slightly more complicated. We imagine our data is $\{(T_k,X_k),\,k=-30,\ldots,M\}$, and use $\{(T_k,X_k),\,k=-30,\ldots,-1\}$ to form the first 30 crossings $\{(T_{k}^0,Y(T_{k}^0),k=-30\ldots,-1\}$. We use those to form the `latticised mean' $\de_0=\sum_{k=-30}^{-1}Y(T^0_k)/30$, and then use $\{(T_k,X_k),\,k=0,\ldots,M\}$ to form $Y(t)$,  $\{V^0_k\}$, and $\{Z^0_k\}$. We estimate the mean using crossings, not the data directly, to ensure $\de_0$ is independent of the tree, since dependence could enter though $\theta(\cdot)$.
All of the simulation results reported below were tested three ways, using $\de_0$ as 0, $X_0$ (the first data point), or the `latticised mean' just described. The latticised mean shows discriminatory power as good or better than the other two possibilities. Substantially larger power was seen for mean reverting alternatives where the mean is non-zero. Since the type 1 error is unaffected we feel using this `latticised mean' is the best for building the crossing tree. We leave for future work determining how much data to dedicate to estimating the mean.

From these definitions we can derive the crossing times $T^l_k$, excursion types $V^l_k$ and subcrossing family sizes $Z^l_k$ for level 1 up to some maximally observed level $m$.
Throughout, let $N(l)$ and $N_V(l)$ be the number of observed crossings and excursions, respectively, at level $l$ for $l = 0, 1, \ldots, m$ . Note that we do not observe $Z^l_k$ for $l = 0$.  Also note that since the tests on $\{Z^l_k\}$ and $\{V^l_k\}$ are invariant to changes in the time co-ordinate of the data, it is enough to have timeseries of the form $\{(i,X_i),\,i=0,\ldots,M\}$.

With the crossing tree for the data thus defined, we can state more precisely the basic question our test addresses. When viewed crossing the lattice $\de_0+2^l\de \Z$, does a linear interpolation of the data process look like a continuous time-changed Brownian motion?

For $\de$ very small, the linear interpolation will provide numerous crossings between data points  and $Z^l_k=2$ will be frequently observed at small levels $l$. The test results will be an artifact of the small $\de$. In these cases the null hypothesis will be rejected simply as a consequence of small $\de$.  Practically, this is a familiar issue when using discrete data where continuous paths are assumed. At resolutions that are too fine, the data doesn't ``look'' continuous, but this isn't the main point of the test which concerns whether the data is from a continuous time-changed Brownian motion. Practically we just test at larger $\delta$ too.

It is common to have observations of a process at regular (e.g., hourly, daily) points in time. For building the crossing tree, this will generally mean some crossings are unseen. This effect is mitigated if the time scale of the observations is small compared to the time scale of the crossings. Alternatively, if we have observations of a process when it changes (e.g., financial tick data) then we don't lose crossings and are just testing continuity as part of the test.

All of our tests using subcrossings are be based on $\{ Z^l_1, Z^l_2, \ldots, Z^l_{N(l)} \}$ and $\{ V^l_1, V^l_2, \ldots, V^l_{N_V(l)} \}$  for each fixed $l  \in \{1, 2, \ldots, m \}$.
This allows us to examine the process $X$ {\em at different scales}.
This is very important from a practical perspective for two reasons.
Firstly, observed processes invariably have a limiting resolution below which they are discrete rather than continuous, so continuity is a modelling assumption that can only apply above a certain scale.
Secondly, one often sees different mechanisms at work at different scales and we should not expect a single model, such as continuous time-changed Brownian motion, to be a good approximation at all scales.
For finance in particular, it is generally accepted that markets exhibit microstructure at small scales that is not seen at larger scales.

To relate the physical scale of crossings to a temporal scale we use average crossing times.
An observation of $Z^l_k$ is based on level $l-1$ crossings (of size $\de 2^{l-1})$.
Accordingly we use $\Ex W^{l-1}_k$ for the temporal scale corresponding to the physical scale $\de 2^{l-1}$, which we test using the $Z^l_k$. We remark that in general, for fixed $l$, the $W^l_k$ are not stationary, so care needs to be taken when estimating and interpreting $\Ex W^l_k$.
In particular high-frequency financial data typically exhibits daily and weekly periodicity in volatility, which is expressed as daily and weekly periodicity of the $W^l_k$ \citep{AJ06}.

\subsection{Distribution Tests}

We consider several tests of the null hypothesis that, for each level $l$, the subcrossing sizes $\{ Z^l_1, Z^l_2, \ldots, Z^l_{N(l)} \}$ form a random sample of values from a 2$\cdot$Geometric$_1$(1/2) distribution. 

\subsubsection{Twos Test}
Under the null hypothesis, $\{ Z^l_1, Z^l_2, \ldots, Z^l_{N(l)} \}$ is a random sample of values from 2$\cdot$Geometric$_1$(1/2), and so `2' is expected to be the most frequent number. Let $T$ be the number of twos in $\{ Z^l_1, Z^l_2, \ldots, Z^n_{N(l)} \}$. Then, conditional on $N(l)=M$, $T$ has a Binomial($M$,$1/2$) distribution under the null hypothesis. This is exact even for short datasets where tests based on asymptotic methods may have problems. We reject the null hypothesis when the value of $T$ is too close to 0 or $M$.

\subsubsection{$\chi^2$-Distribution test}
We use a $\chi^2$-test to compare the empirical distribution of the $\{ Z^l_k \}_{k=1}^{N(l)}$ against the distribution given in CLM1, that is $2\cdot\mbox{Geometric}_1(1/2)$.

For $i \geq 1$ put $O_i = \#\{ k \,:\, Z^l_k = 2i \}$ and $E_i = N(l) 2^{-i}$ (observed and expected frequencies).
Let $O_{i+} = \sum_{j \geq i} O_j$ and $E_{i+} = \sum_{j \geq i} E_j$, then our test statistic is
\begin{equation} \label{testD}
D(l,d) = \sum_{i=1}^{d-1} \frac {(O_i - E_i)^2}{E_i} + \frac{(O_{d+} - E_{d+})^2}{E_{d+}}.
\end{equation}
Under the continuous martingale hypothesis, for fixed $l$ and $d$, $D(l,d) \convergedist \chi^2_{d-1}$ as $N(l) \to \infty$.
We reject the hypothesis (at level $l$) if $D(l,d)$ is large.

It is usually recommended that the approximate $\chi^2$ distribution is only used if the expected frequencies are at least 5.
In our case the smallest expected frequencies are $E_{d-1} = E_{d+} = N(l) 2^{-(d-1)}$. Since the smallest sensible value for $d$ is 3 (2 degrees of freedom), this would suggest that:  1) don't apply the test for  $N(l) < 20$, 2) use $d=3$ for $ 20 \le N(l) <40$, 3) use $d=4$ for $ 40 \le N(l) <80$ and $d_{N(l)}=\lfloor \log_2(N(l)/5)+1 \rfloor$, $N(l) \ge 20$ in general. 

Since under the null hypothesis we know we have a $2 \cdot$ Geometric$_1$(1/2) distribution, we have been able to improve on the rule of thumb. For $N(l) \le 13$ we still do not apply the test. For each of the 26 values for $N(l)$ in  $\{14, 15, \ldots,39\}$ we used Monte Carlo simulations to obtain empirical critical values. For each $M$ in  $\{14, 15, \ldots,39\}$,  10,000,000 datasets were used to create empirical critical values $C_{M,\,0.95}$ corresponding to the 0.95 quantile of the observed test statistic with $d=3$. We then reject the null hypothesis if $D(l,3) \ge C_{N(l),\,0.95}$. 

Following numerous simulations, for $N(l) \ge 40$ we have found using two extra bins, so $d_{N(l)}+2$, is preferable to $d_{N(l)}$. Simulations for various values in $ 40 \le N(l) \le 300$ show that for a test at 5\% significance level, the type 1 error is usually not much larger (e.g. 5.2\% vs. 4.8\%), and is far outweighed by the substantial power the test gains.

We note here that stationarity of the $Z^l_k$ can also be examined using a $\chi^2$-test.
If we split the sample into $m$ parts, calculate $D(l,d)$ over each part then sum, the resulting statistic will be asymptotically $\chi^2_{m(d-1)}$.

\subsubsection{G Test}
Besides the $\chi^2$ tests described above we implemented several other tests. The so-called `G-test' based on the log-likelihood ratio forms the test statistic
\[G_d = 2 \left(\sum_{i=1}^{d-1} O_i\ln(O_i/E_i) + O_{d_+}\ln\left(O_{d_+}/E_{d_+}\right)\right)\]
and has the same limiting $\chi^2_{d-1}$ distribution as $N(l) \to \infty$. We used the basic rule $d_{N(l)}=\lfloor \log_2(N(l)/5)+1 \rfloor$.

\subsubsection{Discrete Kolmogorov-Smirnov}
For continuous data, the Kolmogorov-Smirnov goodness-of-fit test is thought to be more powerful than a ``binned'' test like the $\chi^2$ because the former makes use of the natural order to the data which is lost through binning \citep{horn77}. As such, a Kolmogorov-Smirnov test for discrete data should be preferable. Unlike with copntinuous data, the distribution of the test statistic is not distribution free. \citet{conover72} provides a way to approximate critical values for short datasets. \citet{wood78} suggest estimating the critical value of the test statistic using a Monte Carlo simulation to generate multivariate Normal vectors with calculable covariance matrix, with dimension one less than the data length. Our approach is more straightforward. Using standard definitions for datasets of length $M$, $F_M(x)= \# \{X_i \le x\}/M$ and $H(x)$ are the empirical c.d.f. and the c.d.f. under the null hypothesis. The two-sided test statistic is 
\[D_M=\max_J M^{1/2}|H(x)-F_M(x)| \]
where the maximum is over the set of jump points $J$. For each data length $M$ from 2 to 1000, 1,000,000 datasets were simulated under the null hypothesis to form as many test statistics. The 0.95 quantile of those provided the critical values $C_{M,\,0.95}$. If $N(l) \le 1000$ we reject the null hypothesis if $D_{N(l)} > C_{N(l),\,0.95}$. If $N(l) > 1000$ we reject if $D_{N(l)} > C_{1000,\,0.95}$, using $C_{1000,\,0.95}$ as our estimate of the asymptotic critical value, which the data suggests is reasonable.

\subsubsection{KLP 1998}
Finally, a test based on the first two moments is described by \citet{klp98}. In particular, we use their results to test if the data is from the particular negative Binomial distribution that is the geometric distribution. Their test statistic $T_{NB}$ is built around a standardized version of
\[ \hat{c}= \frac{\frac{1}{n}\sum_{i=1}^n X_i(X_{i}-1)}{\left(\frac{1}{n}\sum_{i=1}^n X_i\right)^2}, \]
and is asymptotically $N(0,1)$. Thus we reject if $|T_{NB}|>C_{0.975}$ where $C_{0.975}$ is the 0.975 quantile of the $N(0,1)$ distribution, which asymptotically would provide a test at 5\% significance.

\subsection{Independence tests}
Testing for independence in a sequence of (stationary) random variables is not straight-forward. In the absence of a sufficient omnibus test we have tried a variety of tests, and report on the use of several key ones here.

Linear dependence can be measured using the autocorrelation, and we include such a test. One can consider the autocorrelation at many lags, however we restrict ourselves to lag 1 autocorrelation because in practice a short term dependence in the increments of $X$, and thus a lag 1 dependence in $\{ Z^l_k \}_{k=1}^{N(l)}$, is the most plausible departure from independence.

The most fundamental approach to testing the independence of a sequence $A_1, A_2, \ldots$ is, for $r = 2, 3, \ldots$, to estimate the joint distribution of $A_{k}, A_{k+1}, \ldots, A_{k+r-1}$ and compare this with the $r$-th power of the estimated marginal distribution. For continuous distributions, estimation of joint distributions can be done in a distribution free manner, and there are general approaches to testing independence using estimates of joint distributions. See for example \citet{BKR61} and \citet{Rob91}. For estimating joint distributions of discrete random variables, ad-hoc approaches are required.

Other existing tests are built around runs and clustering in the data. Fix $l$ and let $I_k = 1$ if $Z^l_k = 2$ and $I_k = 0$ otherwise. Let $U_i$ and  $V_i$ be the length of the $i$-th run formed by consecutive zeros and ones, respectively. Let $n_0$ and $n_1$ be the number of zeros and ones in $\{I_1,\ldots,I_{N(l)}\}$.  Under the continuous martingale hypothesis $I_1, \ldots, I_{N(n)}$ are an i.i.d.\ Bernoulli($1/2$) sequence, and $\{U_i,\,i=1,\ldots,N_0\}$ and $\{V_i,\,i=1,\ldots,N_1\}$ are two i.i.d sequences, where $N_0$ and $N_1$ are the number of runs of zeros and ones respectively.  Tests of independence can be built around the binary sequence $\{I_1,\ldots,I_{N(l)}\}$ and its related quantities.

There are a variety of other tests that could be used, but it is impossible to implement and report on all of them. One such test, designed to detect non-linear dependence using the so-called correlation dimension, is that of \citet{BDSLeB96}. We feel the tests we have included are a fair cross-section of the known tests for independence for finite distributions.

\subsubsection{Autocorrelation test}
There are many approaches to testing for linear dependence in a sequence of random variables, based on the autocovariance.
In particular those of \citet{DW50}, \citet{BP70}, \citet{LMacK88} (variance ratio) and \citep{Dur91} (spectral approach).
However for simplicity we restrict ourselves to the lag 1 autocorrelation.  

Let $\Zbar^l$ be the average of $Z^n_1, \ldots, Z^l_{N(l)}$, then our test statistic is
\begin{equation} \label{testA}
I_1(l) = \frac{ \sum_{k=1}^{N(l)-1} (Z^l_{k+1} - 4)(Z^l_k - 4)} {\sum_{k=1}^{N(l)} (Z^l_k - \Zbar^l)^2}.
\end{equation}
Note that we use the known mean of $Z^l_k$ in the numerator but not the denominator, as this helps inflate $I_1(l)$ when the null hypothesis is false. Using {\it sample means} in the numerator, we know from Kendall, Stuart and Ord \citeyearpar{KSO83} that $I_1(l) \approx N(-1/N(l), 1/N(l))$ for $N(l)$ large. However, using the known means under the null hypothesis, we have the following result.

\begin{lem}
\[ \sqrt{N(l)}I_1(l) \convergedist N(0,1) \text{ as } N(l) \rightarrow \infty. \]
\end{lem} 

\begin{proof}
For fixed $l$, under the null hypothesis, $\{ Z^l_k \}_{k=1}^{N(l)}$ is an i.i.d. sequence with common distribution $2\cdot \text{Geometric}_1(1/2)$. So, for all $k$ we have $\Ex[Z^l_k]=4$ and Var$[Z^l_k]=8$. Now, let $U^l_k=(Z^l_{k+1} - 4)(Z^l_k - 4)$ and 
\[T_l=\frac{\sum_{k=1}^{N(l)-1} U^l_k}{8\sqrt{N(l)-1}}.\] We have $\Ex[U^l_k]=0$, $\Var[U^l_k]=64$, $\text{Cov}(U^l_i,U^l_j)= 0$ for $i \ne j$ and $\Var[T_l]=1$. Since  $\{U^l_k,\, k=1,\ldots,N(l)-1\}$ is a stationary $m$-dependent sequence with $m=1$, by Ibragimov and Linnik \cite[Theorem 19.21]{IbragLin}, $T_l \convergedist N(0,1)$ as $N(l) \rightarrow \infty$. By consistency of the sample variance estimator,
\[\frac{8(N(l)-1)}{\sum_{k=1}^{N(l)}(Z^l_k - \Zbar^l)^2} \convergedist 1 \text{ as } N(l) \rightarrow \infty, \text{ and } \]
\[\sqrt{N(l)}I_1(l)=\frac{\sqrt{N(l)}}{\sqrt{N(l)-1}} \frac{8(N(l)-1)}{\sum_{k=1}^{N(l)}(Z^l_k - \Zbar^l)^2}T_l \convergedist N(0,1) \text{ as } N(l) \rightarrow \infty.\]

\end{proof}

Thus $I_1(l) \approx N(0, 1/N(l))$ for $N(l)$ large and we reject the continuous martingale hypothesis (at level $l$) if $|I_1(l)|$ is large. In practice we found this asymptotic result has acceptable error under the null hypothesis for $N(l) >100$.  For $N(l) \le 100$ we used empirical critical values again. In particular, for $M$ in $\{5,6,\ldots 100\}$ we simulated 10,000,000 datasets of length $M$ to create estimated critical values $C_{M,\,0.025}$  and $C_{M,\,0.975}$ using the 0.025 and 0.975 quantiles of the generated test statistics. For $5 \le N(l) \le 100$ we reject the continuous martingale hypothesis (at level $l$) if $I_1(l) \le C_{N(l),0.025}$ or $I_1(l) \ge C_{N(l),0.975}$. We found using both critical values more robust than using one and assuming $|I_1(l)|$ is symmetric. For $N(l)<5$ there just isn't enough data to reject the null hypothesis.

\subsubsection{Joint Distribution test}

Because estimating joint distributions requires a lot of data, we restrict ourselves to the bivariate case. In our case the joint distribution of $Z^l_k$ and $Z^l_{k+1}$ is known under the null hypothesis, so we can compare the empirical joint distribution against the known distribution using a $\chi^2$-test. Fix $l$ and split the observations $Z^l_1, \ldots, Z^l_{N(l)}$ into pairs, $(Z^l_{2k-1}, Z^l_{2k})$ for $k = 1, \ldots, \lfloor N(l)/2 \rfloor$.
For $i, j \geq 1$ put $O_{i,j} = \#\{ k \,:\, (Z^l_{2k-1}, Z^l_{2k}) = (2i, 2j) \}$ and $E_{i,j} = \lfloor N(l)/2 \rfloor 2^{-(i + j)}$ (observed and expected frequencies), make the definitions
\[O_{i+,j} = \sum_{h \geq i} O_{h,j},\, O_{i,j+} = \sum_{k \geq j} O_{i,k},\, O_{i+,j+} = \sum_{h \geq i, k \geq j} O_{h,k}\]
and similarly for $E_{i+,j}$, $E_{i,j+}$ and $E_{i+,j+}$.
Our test statistic is
\begin{align} \label{testJ}
\begin{split}I_2(l,d) =& \sum_{i,j=1}^{d-1} \frac {(O_{i,j} - E_{i,j})^2}{E_{i,j}}
+ \sum_{i=1}^{d-1} \frac {(O_{i,d+} - E_{i,d+})^2}{E_{i,d+}}\\
+&\sum_{j=1}^{d-1} \frac {(O_{d+,j} - E_{d+,j})^2}{E_{d+,j}}
+ \frac{(O_{d+,d+} - E_{d+,d+})^2}{E_{d+,d+}}.
\end{split}
\end{align}
Under the continuous martingale hypothesis, for fixed $l$ and $d$, $I_2(l,d) \convergedist \chi^2_{d^2-1}$ as $N(l) \to \infty$.
We reject the hypothesis (at level $l$) if $I_2(l,d)$ is large.

It is usually recommended that the approximate $\chi^2$ distribution is only used if the expected frequencies are at least 5.
In our case the smallest expected frequency is $E_{d+,d+} = \lfloor N(l)/2 \rfloor 2^{-2(d-1)}$. Since the smallest sensible value for $d$ is 3 this would suggest we need $N(l) \geq 160$. Following extensive simulations we have found we can use $d=3$ for $N(l) \ge 10$, significantly smaller than 160. Figure \ref{jointdist_perf.fig} shows the rejection rate and 95\% confidence intervals from simulations of 100,000 datasets of length $M$ between 10 and 100 corresponding to the null hypothesis. These show a rejection rate of at most 5.7\% from a test of significance level 5\%, which we believe is acceptable considering the benefit from applying the test to much shorter datasets. We do not apply the test to datsets with length below 10.

This test could reject either from bi-variate dependence or a departure from the hypothesized marginal geometric distribution. Labelling this test as an independence test is very deliberate on our part. Notice that if the sequence of subcrossings $\{Z_k^l\}$ is randomly permuted, the same test could also be applied, but the cause of any departure from the null hypothesis is confined to a departure from the geometric distribution. In fact we have done this with our simulation data and found, consistently, that a substantially smaller percentage of paths is rejected after permuting the subcrossings. This tells us that dependence, not distribution, is the main reason for rejection with the method.

\begin{figure}[ht]
\resizebox{3in}{!}{\includegraphics{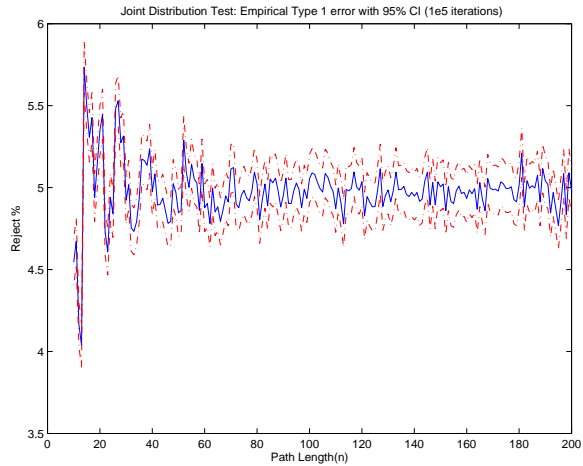}}
\caption{Empirical type 1 error with 95\% confidence interval for our joint distribution test for datasets of length $n\geq10$ using $10^5$ independent iterations for each $n$. \label{jointdist_perf.fig}}
\end{figure}

\subsubsection{Wald and Wolfowitz}
The test of Wald and Wolfowitz \citeyearpar{wald40} is based on the \textit{number} of runs in $\{I_1,\ldots,I_{N(l)}\}$. Too few indicates a tendency to cluster, too many indicates a tendency for values to alternate. Asymptotically the test statistic is $N(0,1)$. It is implemented in Matlab as `runstest' which allows computing a $p$-value using the exact distribution for all lengths instead of the asymptotic $N(0,1)$ limiting distribution.

\subsubsection{O'Brien 1976}
The test of O'Brien \citeyearpar{obrien76} applies the results of Dixon \citeyearpar{dixon40} for two independent samples to the situation of the binary sequence, and is sensitive to multiple clustering. Without loss assume $n_1>n_0$, so ones are more numerous. This test uses the biased version of the sample variance $s^2$ of the sequence $\{V_1,\ldots,V_{N_1}\}$, conditional on $n_0$ and $n_1$. Too much or too little variability indicates a tendency to cluster. O'Brien suggests that for $c s^2$ where $c$ is known the distribution is asymptotically $\chi^2_{\,\nu}$ where 
$\nu$ is a function of $n_0$ and $n_1$, and a table of critical values can be used when $n_0$ and $n_1$ are less than or equal to 10. Unfortunately $\nu$ can be a non-integer, which leads to some approximation error rounding $\nu$  up or down. Using an idea alluded to in Dixon, instead of a $\chi^2$ distribution we use a Gamma distribution for the test statistic, which allows the shape parameter to be fractional and for less approximation.  For $n_0$ and $n_1$ less than or equal to 10 we use empirically generated critical values for a test at 5\% significance.

\subsubsection{O'Brien  \& Dyck 1985}
O'Brien and Dyck \citeyearpar{obrien85} extend the sample variance-based method to using a weighted combination of the sample variances for the zeros and ones: $T=c_0 s_0^2+c_1 s_1^2$ which is approximately $\chi^2_{\nu_0+\nu_1}$ where the degrees of freedom can be calculated. The weights $c_1$ and $c_2$ are chosen so that $c_0 s_0^2$ and $c_1 s_1^2$, which are two asymptotically independent $\chi^2$ statistics on their own, have equal scale parameter. This approach again requires rounding up or down to create an integer degrees of freedom, or using a Gamma distribution. Potthoff \citeyearpar{potthoff} raises several criticisms of the choice of the weightings, and discusses alternatives, but goes no better in establishing which combination might be better. He also suggests that since the contributions of $s_0^2$ and $s_1^2$ are asymptotically independent Normal, the sum is asymptotically Normal, with known mean and variance. We have tried four test statistics with related asymptotic distribution, three using the weighted sum as either a  $\chi^2$ distribution with degrees of freedom rounded up or down, or a gamma distribution. Our fourth test statistic is Potthoff's sum which, when standardized, is asymptotically $N(0,1)$ and does not involve any weights. We found the Normal approximation lacked power, and rounding the degrees of freedom down consistently rejected too much. As such, we will only report on results using the gamma distribution approximation of the weighted sum.

\subsubsection{Larsen}
The test of Larsen et al. \citeyearpar{larsen73} is a  test for unimodal clustering. For $n_1$ successes let the \textit{locations} of those successes be $1 \le R_1 < R_2 < \cdots <R_{n_1} \le N(n)$. Let 
\[K_1=\sum_{i=1}^{n_1} |R_i- \tilde{R}|\] 
where $\tilde{R}$ is the sample median of $\{R_1,\ldots, R_{n_1}\}$. The test statistic is the standardized version of $K_1$ which is asymptotically Normal. For datasets of length $M$ in $\{3,4,\ldots 80\}$ we used empirical critical values again. In particular, we simulated 10,000,000 datasets of length $M$ to create estimated critical values $C_{M,\,0.025}$  and $C_{M,\,0.975}$ using the 0.025 and 0.975 quantiles of the generated test statistics. For $M>80$ we use the asymptotic $N(0,1)$ approximation to the test statistic.

\subsection{Empirical Type 1 Error}

First, we demonstrate the performance of the crossing tree technique when the null hypothesis is satisfied. The first 1250 crossings of standard Brownian motion were simulated using the technique described in Appendix \ref{simtech.sec}.  To obtain 1250 crossings corresponding to standard Brownian motion on $[0,5]$, we used $\de=1/(5\sqrt{10})$ as the basic crossing size. Figure \ref{StdBM3.fig} (left) shows the results for the distribution tests applied to the sequence of subcrossing lengths $\{Z_k^l,\, k=1,\ldots,N(l)\}$. At the smallest levels all five tests reject close to 5\%, with the Twos tests rejecting the fewest paths, at about 4.5\% at level 1. The $\chi^2$ test with two extra degrees of freedom and the G test reject close the level of the test.  These are both ``binned'' tests and similar performance is reasonable. As the levels, $l$, increase the length of subcrossing data $\{Z_k^l,\, k=1,\ldots,N(l)\}$ decreases. For Brownian motion, each increase in level decreases the length $N(l)$ by a factor of four on average: at level one $N(l)$ is about 305 on average, at level 4 it is about 5 on average. Thus the percentage of 10,000 paths rejected falls, since $N(l)$ becomes too small to allow a test. The smallest length we allow varies with the test. Table \ref{StdBM3.table} shows the data from the graph in boldface, and also the percentage of paths rejected out of the number tested, and the number tested. At level 4, all of the paths were shorter than 14, which is our cutoff to apply the $\chi^2$ test with two extra degrees of freedom, so the test was never applied. On the other hand, the Twos test is an exact Binomial test which we can apply for all values of $N(l)$. Our implementation of the discrete Kolmogorov-Smirnov test uses empirically generated critical values, so it can also be applied for all lengths, and performs well with shorter datasets.

Figure \ref{StdBM3.fig} (right) shows the results from applying tests of independence to the subcrossing data $\{Z_k^l,\, k=1,\ldots,N(l)\}$.  The joint distribution-based test shows rejection at about 5\% across levels one to three, with little reduction as level increases, until a large drop at level 4. The autocorrelation test starts  below 5\% but increases to a value closer to the 5\% value expected. The four ``classical'' tests of independence we have chosen to include appear qualitatively similar, although the Runs test rejects fewer paths than the other three. The Matlab implemenation of the Runs test is exact and can be applied to datasets of all lengths, so its performance is disappointing here.

Figure \ref{StdBM3.fig} (right) also shows results from applying our four ``classical'' tests of independence to the binary sequences $\{V_k^l \}_{k=1} ^{N_V(l)}$ representing the the excursions at level $l$. (These tests are indicated with the suffix `UD' for Up-Down in the legend.) The results mirror those from the same tests applied to the subcrossing data, with the Runs test rejecting a noticeably smaller percentage of paths than the other three. From the totality of the tests, saying 5\% rejection is reasonable (perhaps somewhere between 4.7\% and 5.7\% on closer inspection). This compares favourably with the quadratic variation method. Park and Vasudev \citeyearpar{PV05} found 6.4\% rejection using a two-sided KS test, 5.6 \% rejection using a CVM test, and  4.4\% rejection using their SM test. Our implemention shows 5\%--5.5\% rejection using the same three tests, as was shown in Figure \ref{PVBM.fig} (left).

\begin{figure}[h]
\resizebox{3in}{!}{\includegraphics{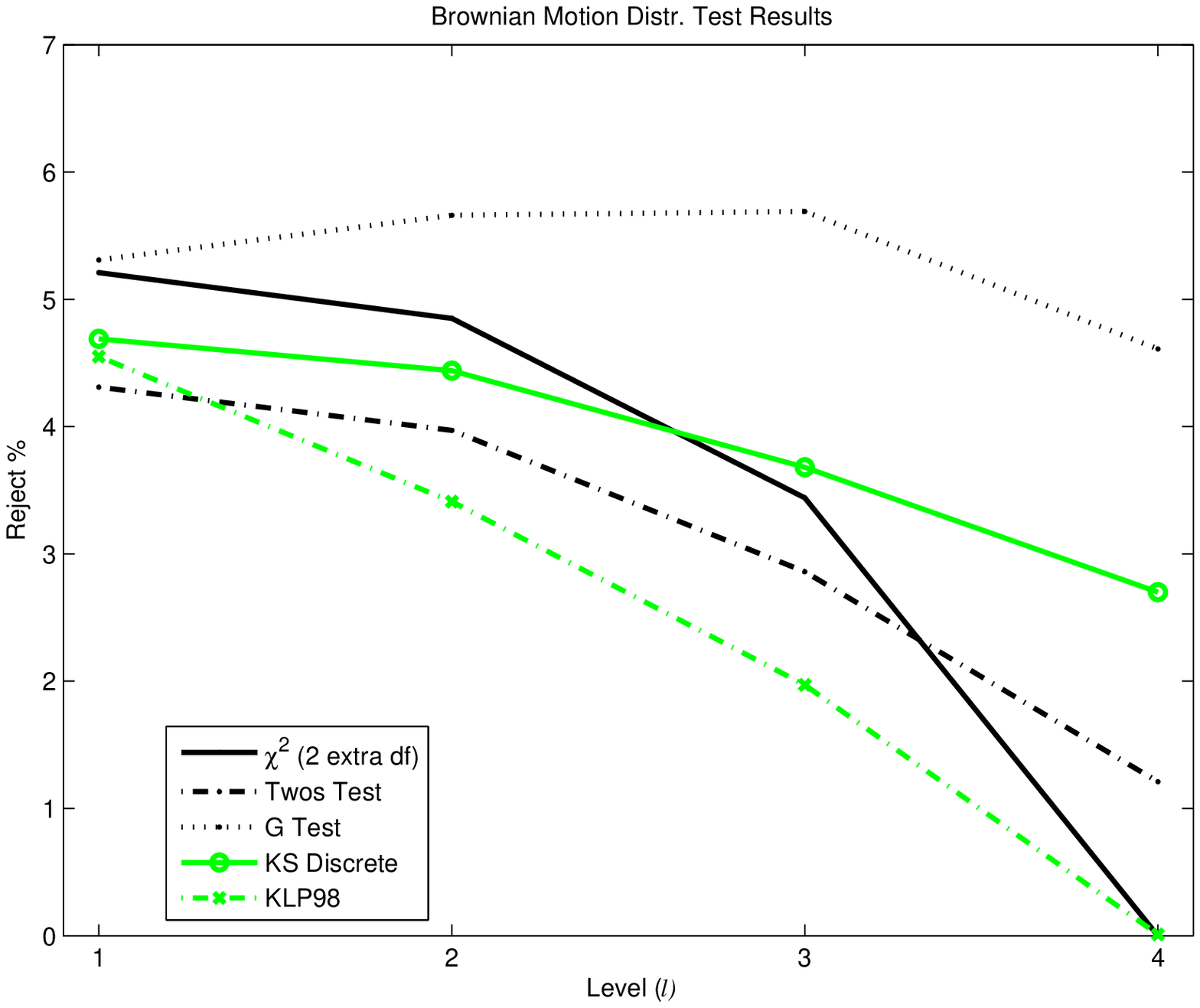}}\resizebox{3in}{!}{\includegraphics{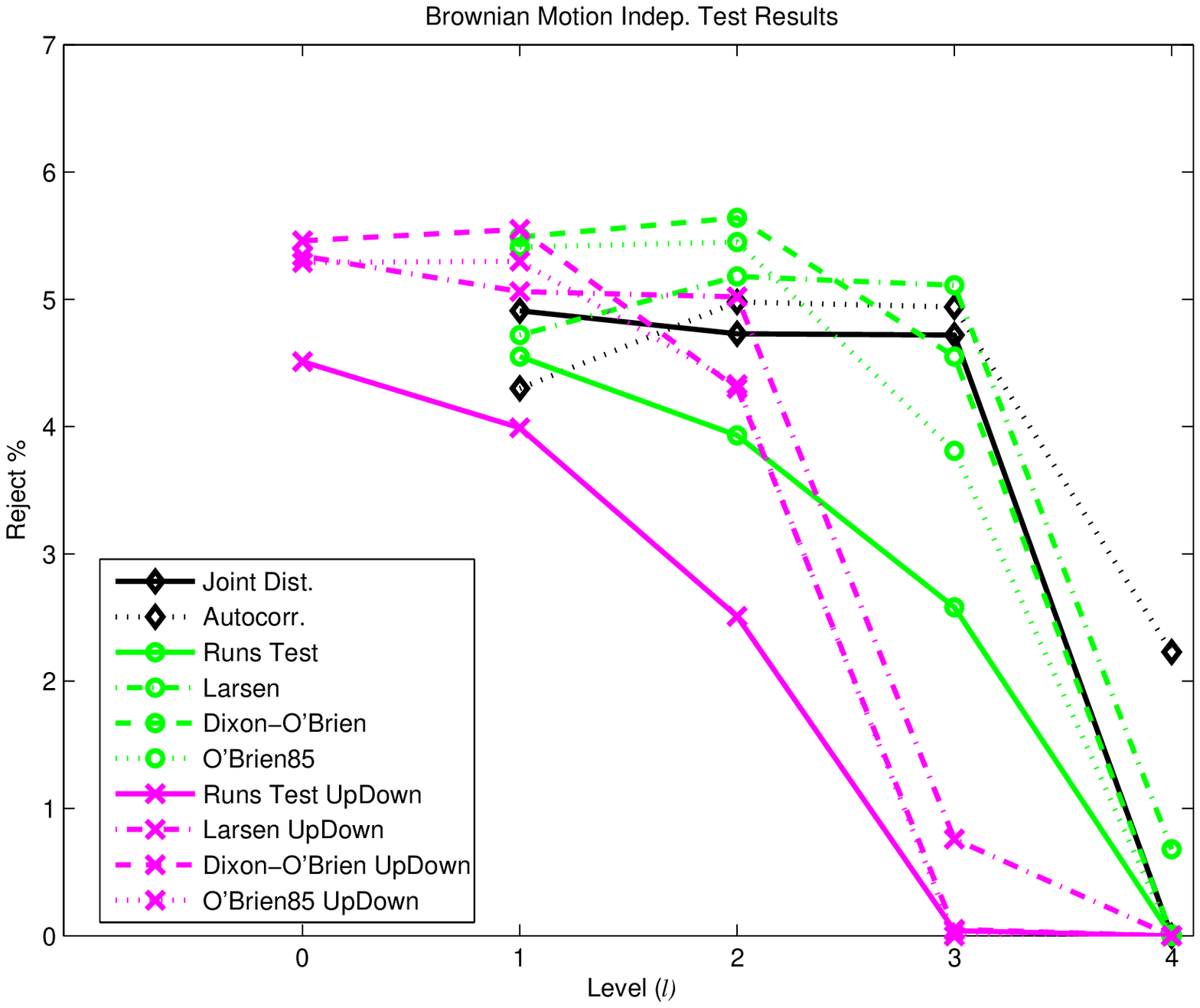}}
\caption{Crossing tree results for standard Brownian Motion. Percentage of 10,000 sample paths rejected, with an average of 1250 crossings by time 5. Rejection using distribution tests (left) and independence tests (right).  Since the null hypothesis is satisfied, 5\% rejection is expected for all tests. Here, 95\% confidence intervals are of the order $\pm0.4\%$.}
\label{StdBM3.fig}
\end{figure}

\begin{table}[th!]
 {\small 
 \begin{center} 
 \begin{tabular}{ |c || *{5}{c|}} 
\hline \multicolumn{6}{|c|}{$\alpha=0$, $n = 1250$, $\de = 1/(5 \sqrt{10})$ }  \\ 
 \hline \hline 
& \multicolumn{5}{c|}{\raisebox{0ex}[12pt]{} {\bf \% of all} (\% of tested; \# tested)} \\ \hline {\bf levels }& 0 & 1 & 2 & 3 & 4 \\ \cline{2-6} 
 \hline \hline\raisebox{0ex}[12pt]{{\bf $\chi^2$ (+2 df)}}&  & {\bf   5.2} & {\bf   4.9} & {\bf   3.4} (  3.7;   9395) & {\bf   0.0} (  NaN;      0)\\ \hline 
{\bf Twos Test} &  & {\bf   4.3} & {\bf   4.0} & {\bf   2.9} & {\bf   1.2} (  1.2;   9969)\\ \hline 
{\bf G Test} &  & {\bf   5.3} & {\bf   5.7} & {\bf   5.7} & {\bf   4.6} (  4.6;   9969)\\ \hline 
{\bf KS Discrete} &  & {\bf   4.7} & {\bf   4.4} & {\bf   3.7} & {\bf   2.7}\\ \hline 
{\bf KLP98 Test} &  & {\bf   4.5} & {\bf   3.4} & {\bf   2.0} & {\bf   0.0}\\ \hline 
\hline 
{\bf Joint Dist.}&  & {\bf   4.9} & {\bf   4.7} & {\bf   4.7} (  4.7;   9963) & {\bf   0.0} (  0.0;     36)\\ \hline 
{\bf Autocorr.} &  & {\bf   4.3} & {\bf   5.0} & {\bf   4.9} & {\bf   2.2} (  4.8;   4637)\\ \hline 
{\bf Runs Test} &  & {\bf   4.5} & {\bf   3.9} & {\bf   2.6} & {\bf   0.0}\\ \hline 
{\bf Larsen Test} &  & {\bf   4.7} & {\bf   5.2} & {\bf   5.1} (  5.1;   9995) & {\bf   0.7} (  0.8;   8782)\\ \hline 
{\bf  Dix.-OBri.} &  & {\bf   5.5} & {\bf   5.6} & {\bf   4.5} & {\bf   0.0} (  0.0;   9969)\\ \hline 
{\bf OBri85} &  & {\bf   5.4} & {\bf   5.5} & {\bf   3.8} (  4.8;   7970) & {\bf   0.0} (  0.0;     10)\\ \hline \hline 
{\bf Runs UD} & {\bf   4.5}  & {\bf   4.0} & {\bf   2.5} & {\bf   0.0} & {\bf   0.0}\\ \hline 
{\bf Larsen UD} & {\bf   5.3}  & {\bf   5.1} & {\bf   5.0} (  5.0;   9999) & {\bf   0.8} (  0.8;   9185) & {\bf   0.0} (  0.0;   4208)\\ \hline 
{\bf  Dix.-OBri. UD} & {\bf   5.5} & {\bf   5.5} & {\bf   4.3} & {\bf   0.1} (  0.1;   9963) & {\bf   0.0} (  0.0;   7021)\\ \hline 
{\bf OBri85 UD} & {\bf   5.3}  & {\bf   5.3} & {\bf   4.3} (  5.5;   7894) & {\bf   0.0} (  0.0;     77) & {\bf   0.0} (  NaN;      0)\\ \hline 
\end{tabular} 
 \end{center}
 } 
\caption{Crossing tree results for standard Brownian Motion. Percentage of 10,000 sample paths rejected, with an average of 1250 crossings by time 5. Since the null hypothesis is satisfied, 5\% rejection is expected for all tests. At higher levels, insufficient data length means some datasets are not tested.} \label{StdBM3.table}
\end{table}

\section{Power estimates}\label{power.sec}
To measure the performance of various distribution and independence tests, we applied them to some well-known processes, including diffusion processes defined by stochastic differential equations of the form:
\begin{equation} \label{diffusion}
dX_t = A(X_t) dt + B(X_t) dW_t.
\end{equation}
In particular we consider
\begin{enumerate}
\item Brownian motion with drift, $A(X_t) = \alpha$, $B(X_t) = 1$;
\item Ornstein-Uhlenbeck process, $A(X_t) = -\al X_t$, $B(X_t) = 1$;
\item Feller's square-root diffusion process, $A(X_t) = \kappa(\mu - X_t)$, $B(X_t) = \sqrt{X_t}$;
\item Fractional Brownian motion with Hurst parameter $H \in (0,1)$.
\end{enumerate}
In each case we simulated 10,000 sample paths of the process in question, applied the tests to either the subcrossing data or up/down symmetry data, and report on our results in several ways. Graphically, we show the percentage out of 10,000 that was significant at the 5\% level with our tests of distribution and our tests of independence. We also report  results in tables. At higher levels, the length of the data drops below the length we feel it appropriate to apply a test. In such cases we report the percentage that was significant, both out of 10,000 and out of the number tested, along with the number of datasets tested. We do this to distinguish between, say, 0\% rejected out of 13 tested and 0\% rejected out of 9727 tested. The latter seems deserving of more weight in our opinion. Park and Vasudev used simulation to test their quadratic variation method on processes 1--3 above, using 10,000 sample paths each. So we can do direct comparisons with their results, and with our implementation of the quadratic variation method. \citet{peters06} do not provide power results.

Since the quadratic variation method uses points regularly spaced in time, and the crossing tree uses crossings, there is a question of what is a `direct' comparison. For example, Park and Vasudev \citeyearpar{PV05} simulate processes with regularly spaced times of interval length 1/250 up to time 5, so 1250 points, which they argue represents daily information for 5 years.  We are free to set the scale $\de$ of the crossing tree, and we determined a value so that, \textit{on average}, the number of crossings in a time interval is the same as the number of points in the quadratic variation method in a time interval of the same length. So, we would compare 1250 points using quadratic variation with 1250 crossings using the crossing tree, with all parameters of the process kept constant. How we determine $\de$ to do this for various processes is described in Appendix \ref{detdelta.sec}.

\subsection{Brownian motion with drift} 
Brownian Motion with drift is an example of a process that does not satisfy property CLM1 of Corollary \ref{clmchar}. For drift 1 and 1.5, 1250 crossings were simulated 10,000 times each. Figure \ref{BMDrift1.fig} shows the results for drift 1, for which $\de=0.06328774784$ was used.

Of the distribution tests (left panel), only the G test rejects at a noteworthy level relative to the 5\% level of the tests. The test of Kyriakoussis et al. shows quickly dimininshing power as level increases. This is likely because the test statistic is based on an asymptotic Gaussian distribution. This is an increasingly poor approximation as the level increases, and so the number of subcrossings decreases.

Amongst the non-up-down independence tests (right panel), rejection between 4\% and 6\% is most common with the autocorrelation test showing rejection around 7.3\% at levels 3 and 4. So the autocorrelation test demonstrates the most power of these non-up-down independence tests.  For the up-down symmetry tests, we expect up-down as often as down-up, and each pair to be independent of all others because of the independent increments of the process. Thus 5\% rejection, as with standard Brownian motion, would be expected.  At higher levels the drift dominates the variability in the path.

Figure \ref{BMdrift.fig} (left) shows results from our implementation of the quadratic variation method, applied to Brownian motion with drift 1. Clearly the rejection rates, in the range 30-60\% are larger for this method, but an exact number is difficult to identify. The results for this type of process are also unique among all those we've seen in that the standardized mean test shows higher power than the Cram\'{e}r-von Mises test. This is not surprising since the increments being tested are being tested as mean 0, although they are not. Park and Vasudev report rejection rates of 47\% (KS), 56\% (CV), and 60\% (SM), which are similar to ours, and have the same ordering from smallest to largest. Clearly, the quadratic variation method shows higher power than the crossing tree at this length.

Figure \ref{BMDrift15.fig} shows results for Brownian motion with drift $\al=1.5$. With larger drift, the smaller value $\de=0.06334057822$ gives an average of 1250 crossings by time 5. The larger drift term provides some clear differences. Rejection rates from the distribution tests are generally higher. In particular, the discrete Kolmogorov-Smirnov and Twos test show rejection around 15.5\% at level 4, about 10\% higher than for drift 1. With the Twos test in particular, the result is easily explained. With a stong drift component there will be many crossings with only two subcrossings and no  excursions to give additional crossings. Thus, 2 subcrossings will occur more than 50\% of the time (the expected rate under the null hypothesis). As with drift 1, the power of the test by Kyriakoussis et al. isn't really comparable at higher levels. 

Amongst the independence tests (right panel), the autocorrelation test is striking that it rejects 20\% at level 4. But, as with drift 1, these tests still show low power compared to the quadratic variation method. With our implementation, 89-91\% of paths are rejected. Park and Vasudev reported identical rejection rates:  82\% (KS), 89\% (CV) and 91\% (SM). So, the quadratic variation method again shows higher power to reject Brownian motion with drift.

Of all the processes we consider here, the higher power of the quadratic variation method over the crossing tree is particular to Brownian motion with drift. We did additional simulations (not shown) to explore the crossing tree further for this process. Again using drift 1, and the same corresponding value of $\de$ for the crossing tree, datasets with more crossings were simulated. Especially for these two explorations, only 1000 datasets were used, providing lengths of either 10,000 crossings or 20,000 crossings. So these are 8 and 16 times longer, respectively, than for the results previously described. The additional length provides more data at lower levels, and additional levels at which to apply tests. For length 10,000, the autocorrelation test rejected 83\% at level 5 testing 998 paths. The Twos test, the $\chi^2$ test and the KS discrete test rejected about 67\%. With 20,000 crossings those tests rejected over 95\% at level 5. While higher levels were not tested, it is likely that there was sufficient length to test level 6, with higher rejection rates again.

\begin{figure}[htb!]
\resizebox{3in}{!}{\includegraphics{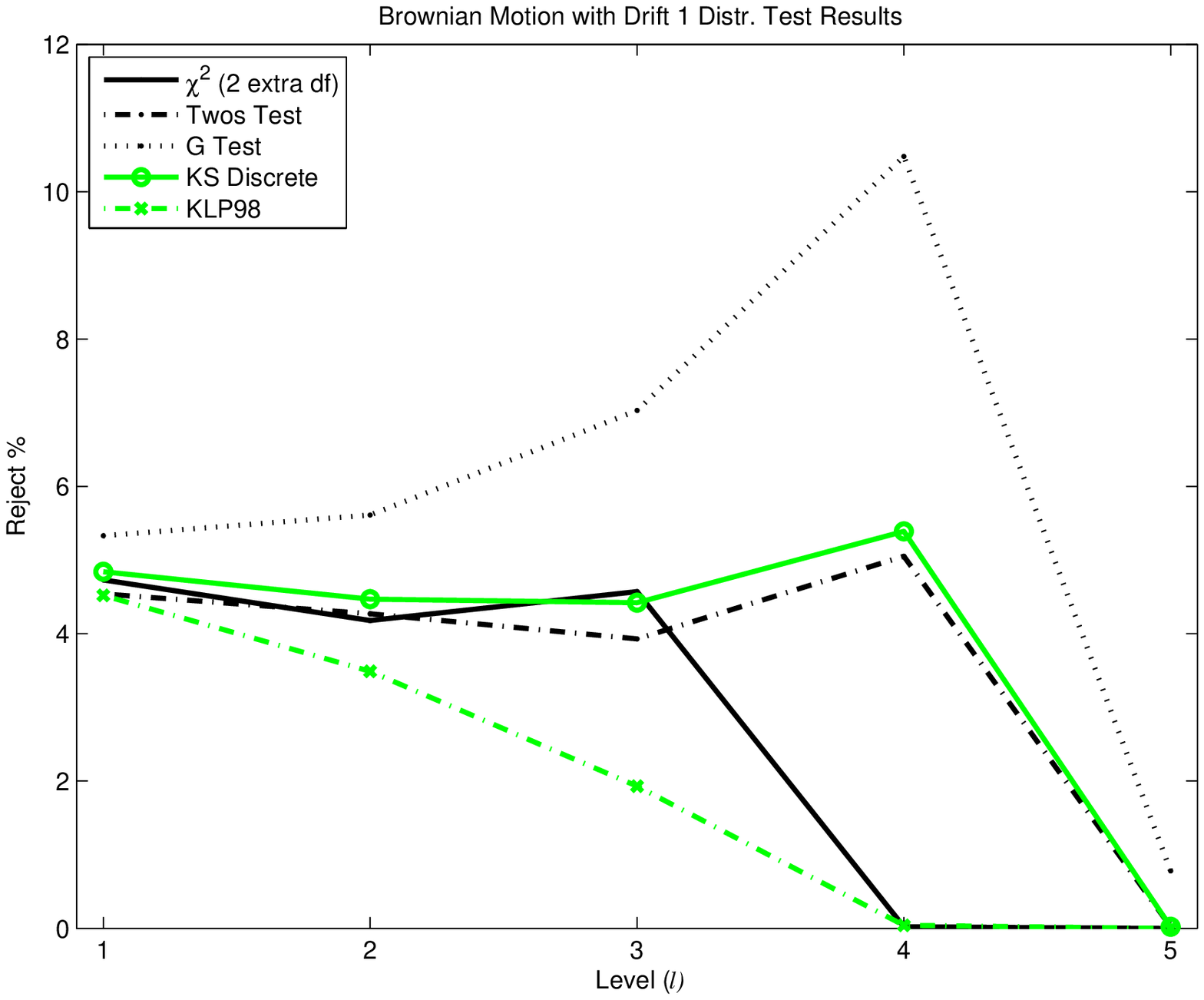}}\resizebox{3in}{!}{\includegraphics{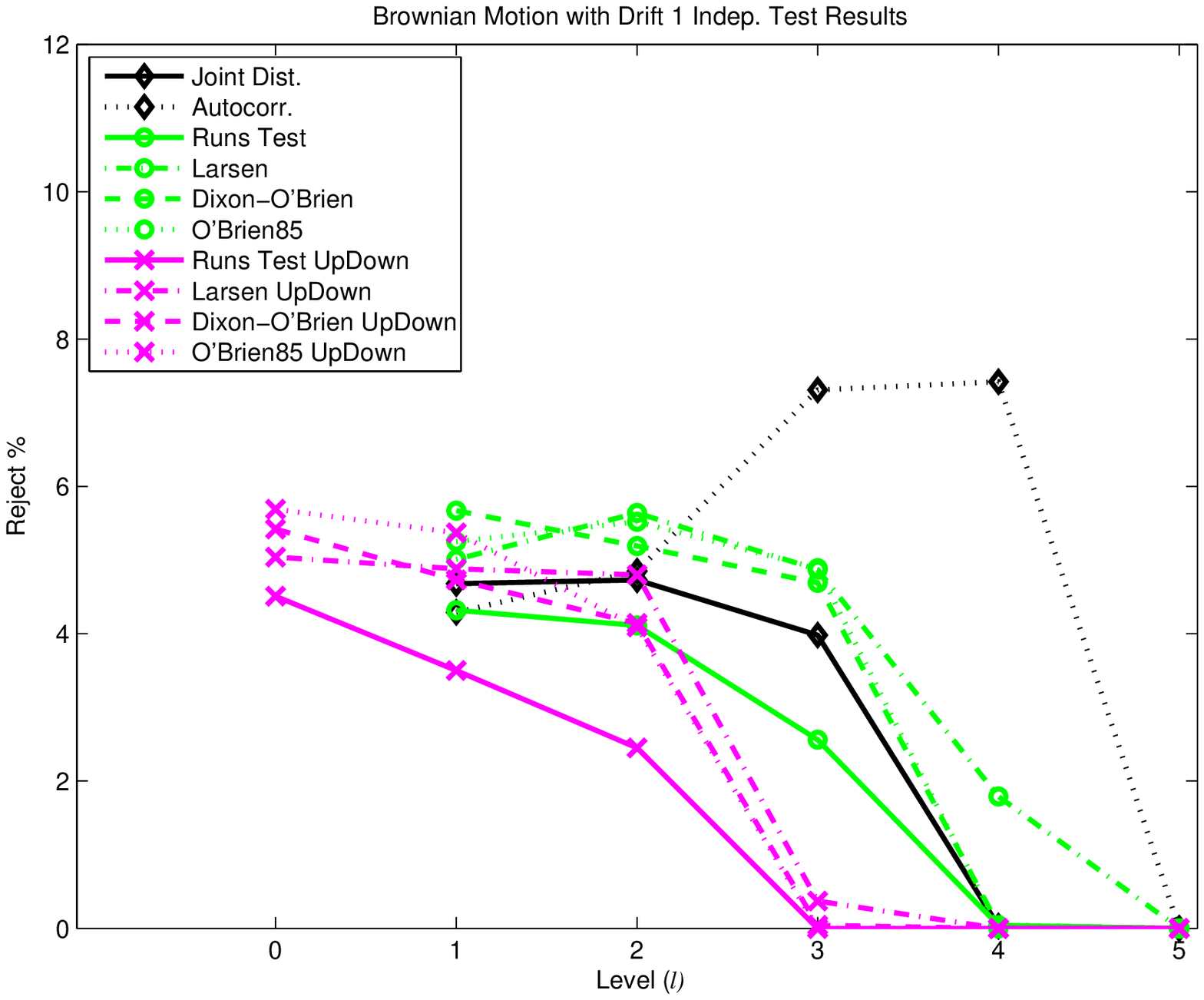}}
\caption{Crossing tree results for Brownian Motion with drift $\al=1$. Percentage of 10,000 sample paths rejected, with an average of 1250 crossings by time 5. Rejection using distribution tests (left) and independence tests (right).  Here, 95\% confidence intervals are 1\% or smaller.\label{BMDrift1.fig}}
\end{figure}

\begin{table}[thb!]

 {\footnotesize
 \begin{center} 
 \begin{tabular}{ |c || *{6}{c|}} 
\hline \multicolumn{7}{|c|}{$\alpha=1$, $n = 1250$, $\de = 0.06328774784$ }  \\ 
 \hline \hline 
& \multicolumn{6}{c|}{\raisebox{0ex}[12pt]{} {\bf \% of all} (\% of tested; \# tested)} \\ \hline {\bf levels }& 0 & 1 & 2 & 3 & 4 & 5 \\ \cline{2-7} 
 \hline \hline\raisebox{0ex}[12pt]{{\bf $\chi^2$ (+2 df)}}&  & {\bf   4.7} & {\bf   4.2} & {\bf   4.6} (  4.7;   9721) & {\bf   0.0} (100.0;      2) & {\bf   0.0} (  NaN;      0)\\ \hline 
{\bf Twos Test} &  & {\bf   4.5} & {\bf   4.3} & {\bf   3.9} & {\bf   5.1} (  5.1;   9998) & {\bf   0.0} (  0.0;   9397)\\ \hline 
{\bf G Test} &  & {\bf   5.3} & {\bf   5.6} & {\bf   7.0} & {\bf  10.5} ( 10.5;   9998) & {\bf   0.8} (  0.8;   9397)\\ \hline 
{\bf KS Discrete} &  & {\bf   4.8} & {\bf   4.5} & {\bf   4.4} & {\bf   5.4} & {\bf   0.0}\\ \hline 
{\bf KLP98 Test} &  & {\bf   4.5} & {\bf   3.5} & {\bf   1.9} & {\bf   0.0} & {\bf   0.0}\\ \hline 
\hline 
{\bf Joint Dist.}&  & {\bf   4.7} & {\bf   4.7} & {\bf   4.0} (  4.0;   9998) & {\bf   0.0} (  1.0;    289) & {\bf   0.0} (  NaN;      0)\\ \hline 
{\bf Autocorr.} &  & {\bf   4.3} & {\bf   4.9} & {\bf   7.3} & {\bf   7.4} ( 10.0;   7384) & {\bf   0.0} (  0.0;      7)\\ \hline 
{\bf Runs Test} &  & {\bf   4.3} & {\bf   4.1} & {\bf   2.6} & {\bf   0.0} & {\bf   0.0}\\ \hline 
{\bf Larsen Test} &  & {\bf   5.0} & {\bf   5.6} & {\bf   4.9} (  4.9;   9999) & {\bf   1.8} (  1.9;   9660) & {\bf   0.0} (  0.0;   8507)\\ \hline 
{\bf  Dix.-OBri.} &  & {\bf   5.7} & {\bf   5.2} & {\bf   4.7} & {\bf   0.0} (  0.0;   9998) & {\bf   0.0} (  0.0;   9397)\\ \hline 
{\bf OBri85} &  & {\bf   5.3} & {\bf   5.5} & {\bf   4.9} (  5.9;   8358) & {\bf   0.0} (  0.0;     41) & {\bf   0.0} (  NaN;      0)\\ \hline \hline 
{\bf Runs UD} & {\bf   4.5}  & {\bf   3.5} & {\bf   2.5} & {\bf   0.0} & {\bf   0.0} & {\bf   0.0}\\ \hline 
{\bf Larsen UD} & {\bf   5.0}  & {\bf   4.9} & {\bf   4.8} & {\bf   0.4} (  0.4;   8683) & {\bf   0.0} (  0.0;   2410) & {\bf   0.0} (  0.0;    135)\\ \hline 
{\bf  Dix.-OBri. UD} & {\bf   5.4} & {\bf   4.7} & {\bf   4.1} & {\bf   0.0} (  0.0;   9903) & {\bf   0.0} (  0.0;   4416) & {\bf   0.0} (  0.0;    256)\\ \hline 
{\bf OBri85 UD} & {\bf   5.7}  & {\bf   5.4} & {\bf   4.1} (  5.3;   7675) & {\bf   0.0} (  0.0;     19) & {\bf   0.0} (  NaN;      0) & {\bf   0.0} (  NaN;      0)\\ \hline 
\end{tabular} 
 \end{center}
 } 
\caption{Crossing tree results for Brownian Motion with drift $\al=1$. Percentage of 10,000 sample paths rejected, with an average of 1250 crossings by time 5.  At higher levels, insufficient data length means some datasets are not tested.} \label{BMDrift1.table}
\end{table}

\begin{figure}[htb!]
\resizebox{3in}{!}{\includegraphics{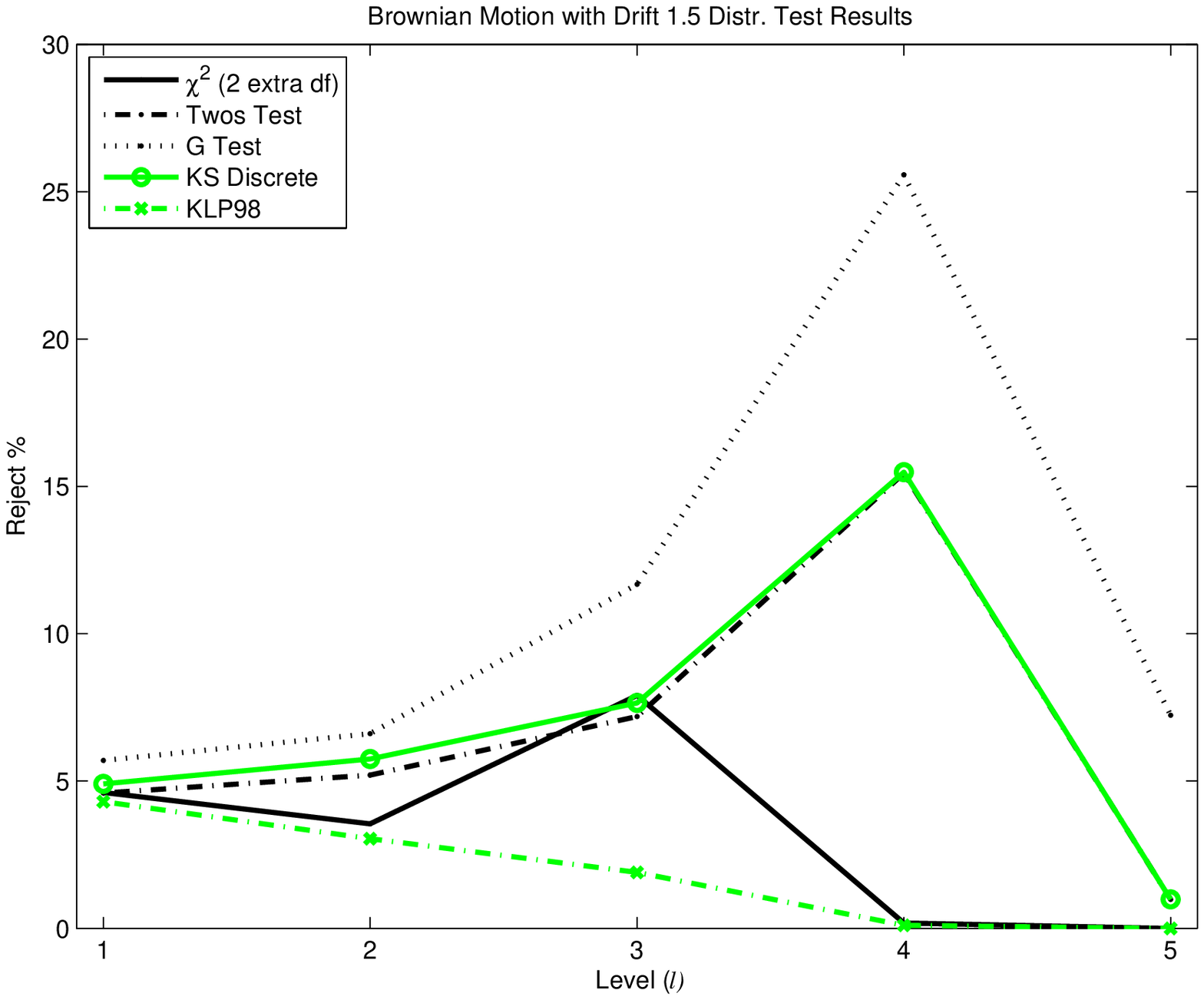}}\resizebox{3in}{!}{\includegraphics{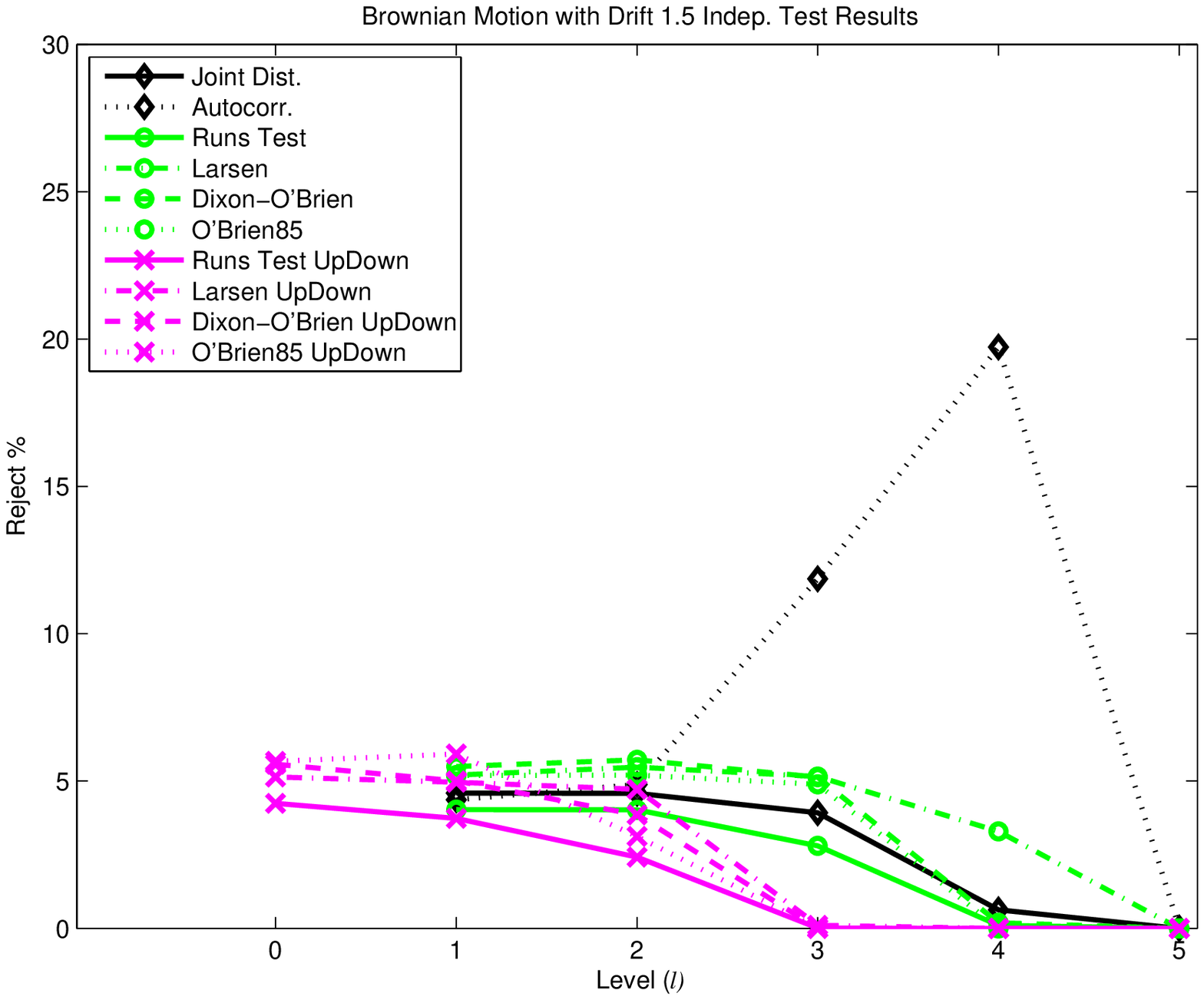}}
\caption{Crossing tree results for Brownian Motion with drift $\al=1.5$. Percentage of 10,000 sample paths rejected, with an average of 1250 crossings by time 5. Rejection using distribution tests (left) and independence tests (right).  Here, 95\% confidence intervals are 1\% or smaller.\label{BMDrift15.fig}}
\end{figure}

\begin{table}[thb!]

 {\footnotesize 
 \begin{center} 
 \begin{tabular}{ |c || *{6}{c|}} 
\hline \multicolumn{7}{|c|}{$\alpha=1.5$, $n = 1250$, $\de =0.06334057822$ }  \\ 
 \hline \hline 
& \multicolumn{6}{c|}{\raisebox{0ex}[12pt]{} {\bf \% of all} (\% of tested; \# tested)} \\ \hline {\bf levels }& 0 & 1 & 2 & 3 & 4 & 5 \\ \cline{2-7} 
 \hline \hline\raisebox{0ex}[12pt]{{\bf $\chi^2$ (+2 df)}}&  & {\bf   4.6} & {\bf   3.5} & {\bf   7.9} (  8.0;   9924) & {\bf   0.2} (100.0;     18) & {\bf   0.0} (  NaN;      0)\\ \hline 
{\bf Twos Test} &  & {\bf   4.6} & {\bf   5.2} & {\bf   7.2} & {\bf  15.4} & {\bf   1.0} (  1.0;   9923)\\ \hline 
{\bf G Test} &  & {\bf   5.7} & {\bf   6.6} & {\bf  11.7} & {\bf  25.6} & {\bf   7.2} (  7.3;   9923)\\ \hline 
{\bf KS Discrete} &  & {\bf   4.9} & {\bf   5.8} & {\bf   7.6} & {\bf  15.5} & {\bf   1.0}\\ \hline 
{\bf KLP98 Test} &  & {\bf   4.3} & {\bf   3.0} & {\bf   1.9} & {\bf   0.1} & {\bf   0.0}\\ \hline 
\hline 
{\bf Joint Dist.}&  & {\bf   4.6} & {\bf   4.6} & {\bf   3.9} (  3.9;   9999) & {\bf   0.6} (  4.3;   1464) & {\bf   0.0} (  NaN;      0)\\ \hline 
{\bf Autocorr.} &  & {\bf   4.4} & {\bf   4.9} & {\bf  11.9} & {\bf  19.7} ( 22.3;   8853) & {\bf   0.0} (  1.7;     58)\\ \hline 
{\bf Runs Test} &  & {\bf   4.0} & {\bf   4.0} & {\bf   2.8} & {\bf   0.1} & {\bf   0.0}\\ \hline 
{\bf Larsen Test} &  & {\bf   5.2} & {\bf   5.5} & {\bf   5.1} & {\bf   3.3} (  3.3;   9921) & {\bf   0.0} (  0.0;   9708)\\ \hline 
{\bf  Dix.-OBri.} &  & {\bf   5.5} & {\bf   5.7} & {\bf   5.1} & {\bf   0.2} & {\bf   0.0} (  0.0;   9923)\\ \hline 
{\bf OBri85} &  & {\bf   5.2} & {\bf   5.2} & {\bf   4.9} (  5.7;   8550) & {\bf   0.0} (  0.0;     69) & {\bf   0.0} (  NaN;      0)\\ \hline \hline 
{\bf Runs UD} & {\bf   4.2}  & {\bf   3.7} & {\bf   2.4} & {\bf   0.0} & {\bf   0.0} & {\bf   0.0}\\ \hline 
{\bf Larsen UD} & {\bf   5.1}  & {\bf   5.0} & {\bf   4.7} (  4.7;   9998) & {\bf   0.1} (  0.1;   8179) & {\bf   0.0} (  0.0;   1416) & {\bf   0.0} (  0.0;     37)\\ \hline 
{\bf  Dix.-OBri. UD} & {\bf   5.6} & {\bf   5.0} & {\bf   3.9} & {\bf   0.0} (  0.0;   9722) & {\bf   0.0} (  0.0;   2666) & {\bf   0.0} (  0.0;     58)\\ \hline 
{\bf OBri85 UD} & {\bf   5.7}  & {\bf   5.9} & {\bf   3.1} (  4.3;   7210) & {\bf   0.0} (  0.0;      5) & {\bf   0.0} (  NaN;      0) & {\bf   0.0} (  NaN;      0)\\ \hline 
\end{tabular} 
 \end{center}
 } 
\caption{Crossing tree results for Brownian Motion with drift $\al=1.5$. Percentage of 10,000 sample paths rejected, with an average of 1250 crossings by time 5. At higher levels, insufficient data length means some datasets are not tested. \label{BMDrift15.table}}
\end{table}

\begin{figure}[htb!]
\resizebox{3in}{!}{\includegraphics{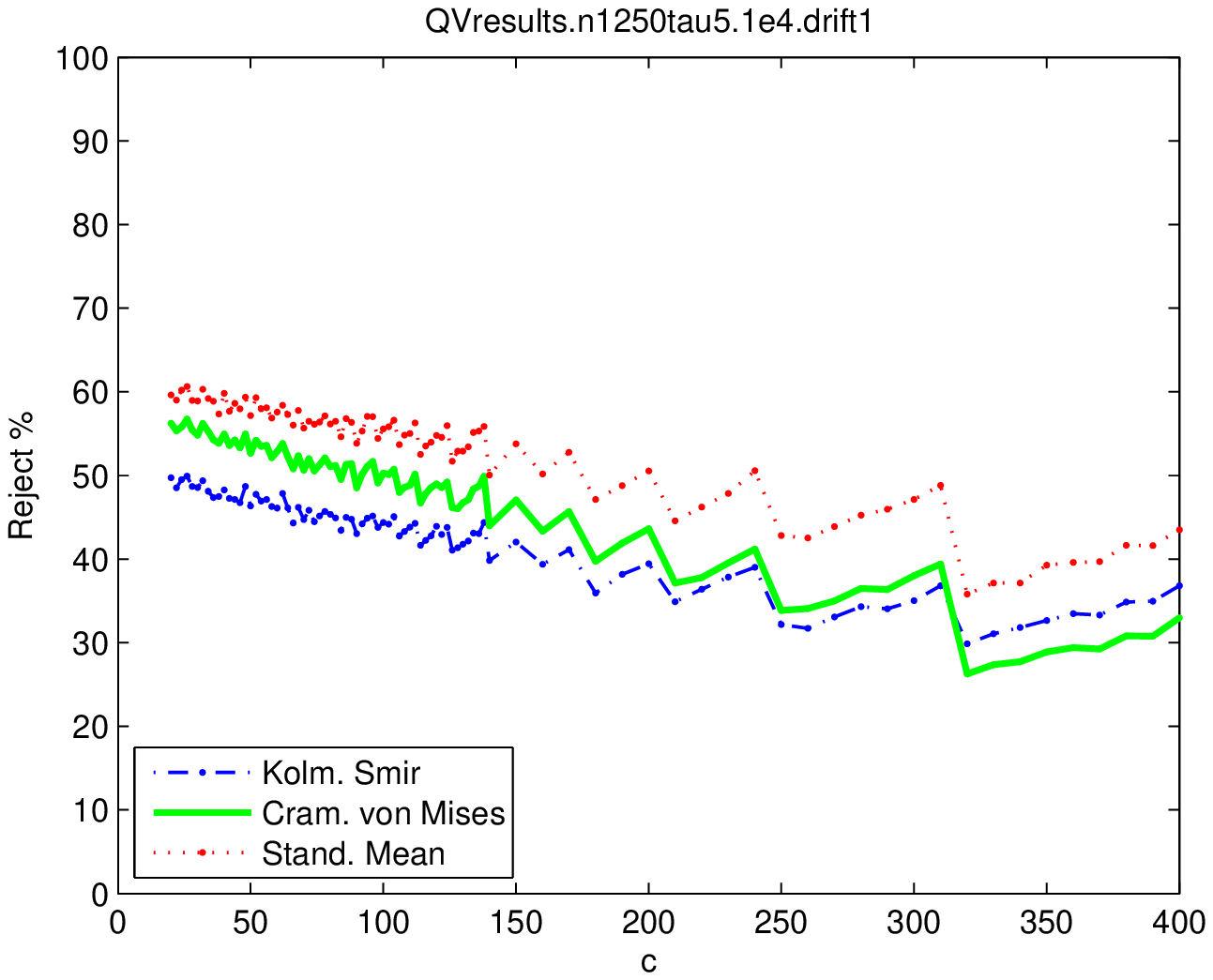}}\resizebox{3in}{!}{\includegraphics{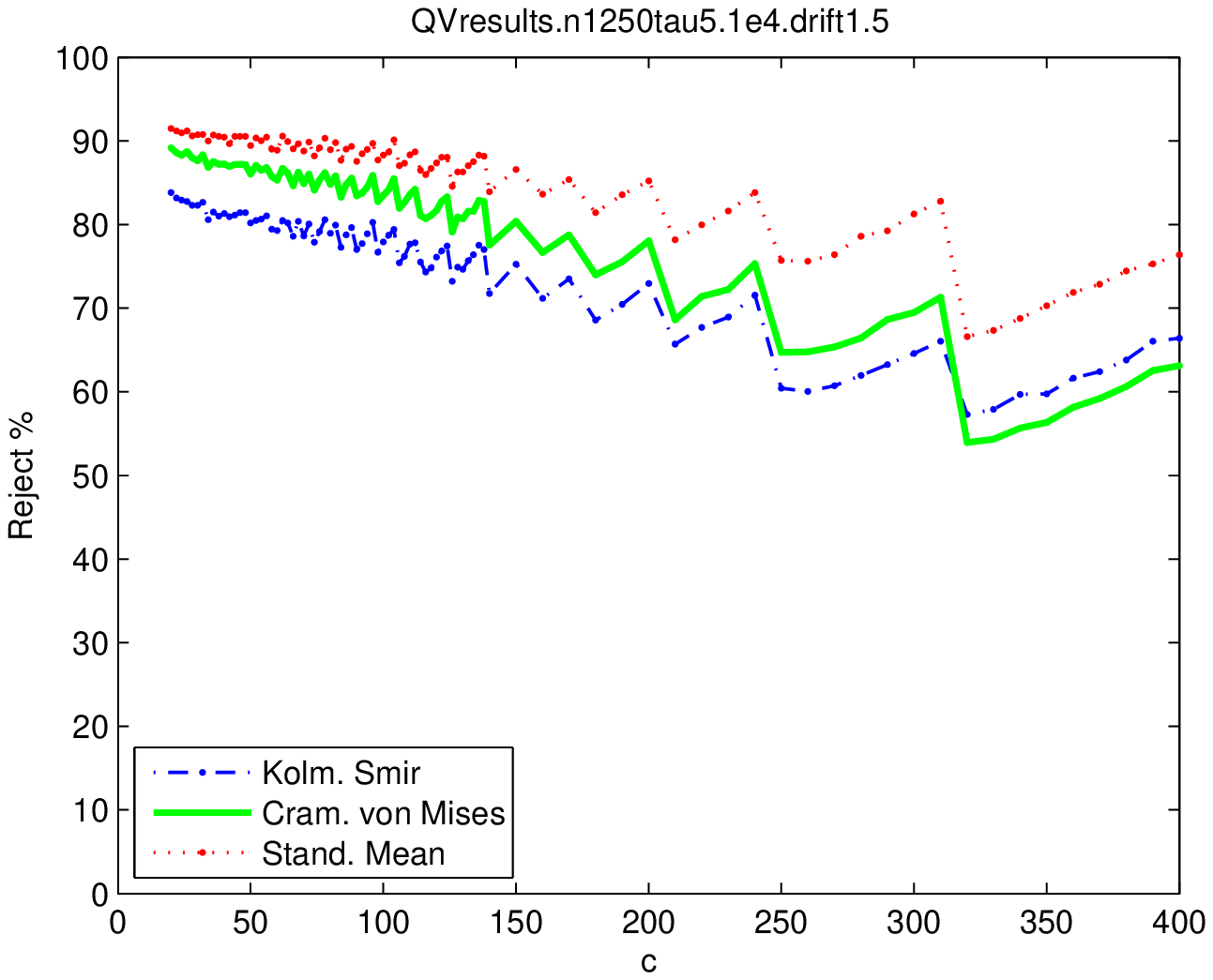}}
\caption{Rejection rates using sample quadratic variation for 10,000 sample paths of length 1250 for standard Brownian motion with  drift 1 (left) and 1.5 (right), at times spaced every 1/250. Here, 95\% confidence intervals are 1\% or smaller. \label{BMdrift.fig}}
\end{figure}

\subsection{Ornstein-Uhlenbeck process} 
For the Ornstein-Uhlenbeck process 
\[dX(t) = -\al X(t) dt + \si dW(t)\] 
with $\al, \si > 0$ we considered $\si=1$ and both $\al=8$ and $\al=10$. Park and Vasudev \citeyearpar{PV05} argue that these values of $\al$ (relative to $\si$) are reasonable for short term interest rate data. This process has a stationary $N(0,\sigma^2/(2\alpha))$ distribution.  In all cases it was the stationary process we worked with.

First we simulated 1250 crossings of the process with $\al=8$, $\si=1$. The value $\de=0.063015$ was determined to give an average of 1250 crossings in time 5, as desired. Figure \ref{OU1.fig} (left) shows the results from the distribution tests. The $\chi^2$ test with two extra degrees of freedom is most powerful, rejecting over 11\% at level 2. The lack of power at level 3 appears to be an artifact of data length relative to our cutoffs. We do not apply the $\chi^2$ test to datasets shorter than length 14, even with our use of empirical critical values. The KLP98 test does much better with this process, rejecting almost 9\% by level 3. Interestingly,  the discrete KS test rejects at about 8\%, but at level 2, suggesting it is truly sensitive to different features than KLP98.

Figure \ref{OU1.fig} (right) shows the results from the independence tests. Here the joint distribution test dominates, rejecting over 14\% at level 3 while the other tests exceed their significance level of 5\% barely at all. The Runs test is also noteworthy since it's rejection rate doesn't fall quickly, as it does with Brownian motion and Brownian motion with drift. In fact, it is superior to the other three ``classical'' tests at level 3.

An additional comment on the joint distribution test is in order. That test makes an assumption about independence and an assumption about the marginal distribution of the number of subcrossings. Considering the large rejection rates from the distribution tests, it could be a violation of the distribution assumption that causes the joint distribution test to reject so many paths. However, when we permute the subcrossing data $\{Z_k^l\}$ we find the joint distribution test then rejects about 4\% of all paths at level 3. So 10\% of the paths are rejected by the joint distribution test due to bivariate dependence in the data. Phrased another way, of the paths rejected by the joint distribution test, almost 70\% are rejected due to bivariate dependence.  At levels one and two, the results after permuting are virtually indistiguishable from those without permuting, and reject about 5\%, which is the level of the test. These results were consistent for all the OU processes we studied.

These results justify rejecting 11--15\% of the sample paths. By comparison, results from our implementation of the quadratic variation method are shown in Figure \ref{OUshort.fig}. Rejection rates are around 0.7\% (KS) and 1\% (CVM). The standardized mean (SM) test consistently has no power. So the crossing tree method appears to reject 10--15 times as many sample paths. Park and Vasudev \citeyearpar{PV05} reported no results for length 1250, except to say the test had ``low power''. This is not surprising since their implementation compromised the power by searching for the time interval length that makes the transformed data most like $N(0,1)$.

Figure \ref{OU2.fig} shows results for the same Ornstein-Uhlenbeck process, but with 5000 crossings. The additional length provides longer datasets for our tests, and the result is most noticeable at level three. As with the shorter datasets, the $\chi^2$ test rejects the most paths, followed by the KLP98 test. But with the longer datasets, considerably more paths are rejected. With our $\chi^2$ test, 73\% are rejected, compared to just 11\% with 1250 crossings. From the independence tests, the joint distribution test again dominates, even more so now, rejecting almost 87\%. The Runs test again shows the most power of the classical ``independence'' tests, although at 23\% all of the distribution tests except for the Twos test are more powerful.

The results for 5000 crossings suggest the crossing tree rejects around 73-87\% of the sample paths. Again, this is better than the quadratic variation method. Figure \ref{OUlong.fig} (left) shows  rejection rates of about  47\% (KS) and 67\% (CVM) with our implemention. Park and Vasudev \citeyearpar{PV05} reported even lower results: 13\% (KS) and 24\% (CVM).

For comparison we also simulated crossings from the Ornstein-Uhlenbeck process with $\al=10$, hence a stronger mean reversion effect. The scale parameter for the crossing tree was found to be 0.062945. Figure \ref{OU3.fig} shows the results are generally comparable, with slightly higher rejection rates, to those for $\al=8$. Again, amongst the distribution tests, our $\chi^2$ test rejects the most while the KLP98 test rejects the secondmost, again at levels two and three respectively. Amongst the independence tests the joint distribution test again rejects considerably more than the others. One noteworthy difference from the shorter dataset is that now the Runs test distinguishes itself with noticeably more power than the other ``classical'' tests at level 3. Primarily from the $\chi^2$ and joint distribution tests, the crossing tree method rejects about 15\% of the sample paths. Figure \ref{OUshort.fig} (right) shows with our implemenation of the quadratic variation method rejection rates of only 1.2\% (KS) and 1.8\%(CVM). Again, Park and Vasudev reported no results for this length.

Figure \ref{OU4.fig} shows results for 5000 crossings and $\al=10$ are comparable to those for 1250 crossings. From the distribution tests, our $\chi^2$ test rejects the most, and the KLP98 test rejects the second most. Amongst the independence tests, the joint distribution test rejects the most while the Runs test rejects the second most. Again, with additional crossings the rejection rates are much higher, at around 97\% for the joint distribution test and 78\% for our $\chi^2$ test. Again these rates are higher (but generally comparable) to those from our implementation of the quadratic variation method. Figure \ref{OUlong.fig} (right) shows rejection rates of 69\% (KS) and 87\% (CVM). Again, our quadratic variation rates are higher than those of Park and Vasudev, who reported rates of 31\% (KS) and 52\% (CVM).

Finally, we simulated the process with $\al=1$, $\si=1$, $n=5000$ to investigate a weaker mean reversion effect. Results are shown in Figure \ref{OU5.fig} and Table \ref{OU5.table}. Because the mean reversion is weaker, it is much harder to distinguish this alternative from Brownian motionn and the rejection rates are much closer to their signficance level of 5\%.. At level 3 our distribution and joint distribution tests reject about 7\%.  Figure \ref{OUk1.fig} shows results for our implementation of the quadratic variation-based method. Neither test performs well, with rejections rates 0.17\% (KS) and 0.12\%(CVM) well below 1\%. Interestingly, for these parameters  rejection rates are higher with shorter ($n=1250$) datasets, unlike our other results using -sample variation.

In summary, for the Ornstein-Uhlenbeck processes considered here, the crossing tree method shows higher power. With strong mean reversion ($\al=8$ or 10, $\si=1$), for length 1250 the crossing tree method rejects between 8 and 15 times more paths. For length 5000 the difference is less dramatic, but an additional 10\% are rejected. With stronger mean reversion ($\al=1$, $\si=1$) the difference is more dramatic, with the crossing tree method rejection around 40 times as many paths.
  
\begin{figure}[p]
\resizebox{3in}{!}{\includegraphics{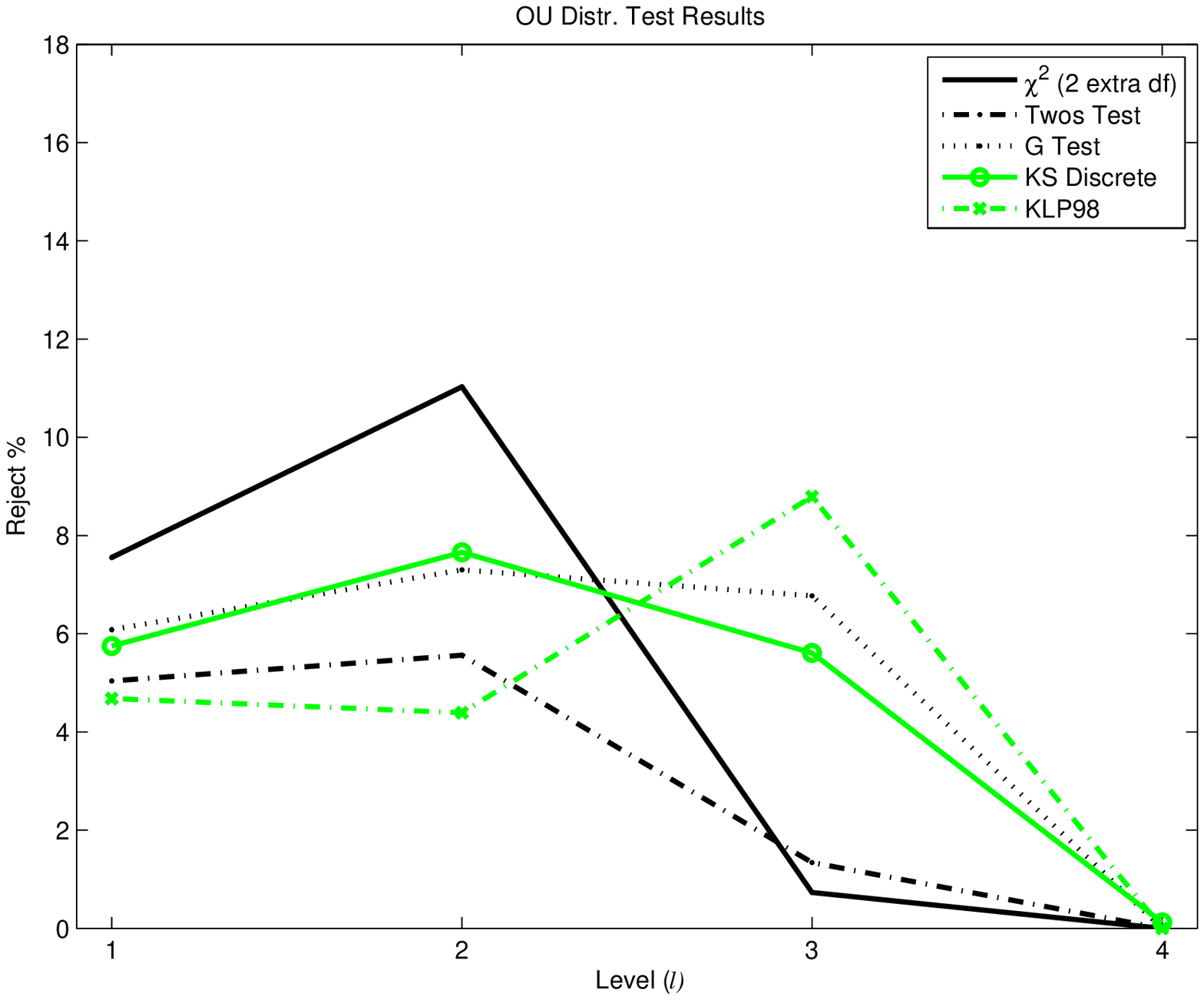}}\resizebox{3in}{!}{\includegraphics{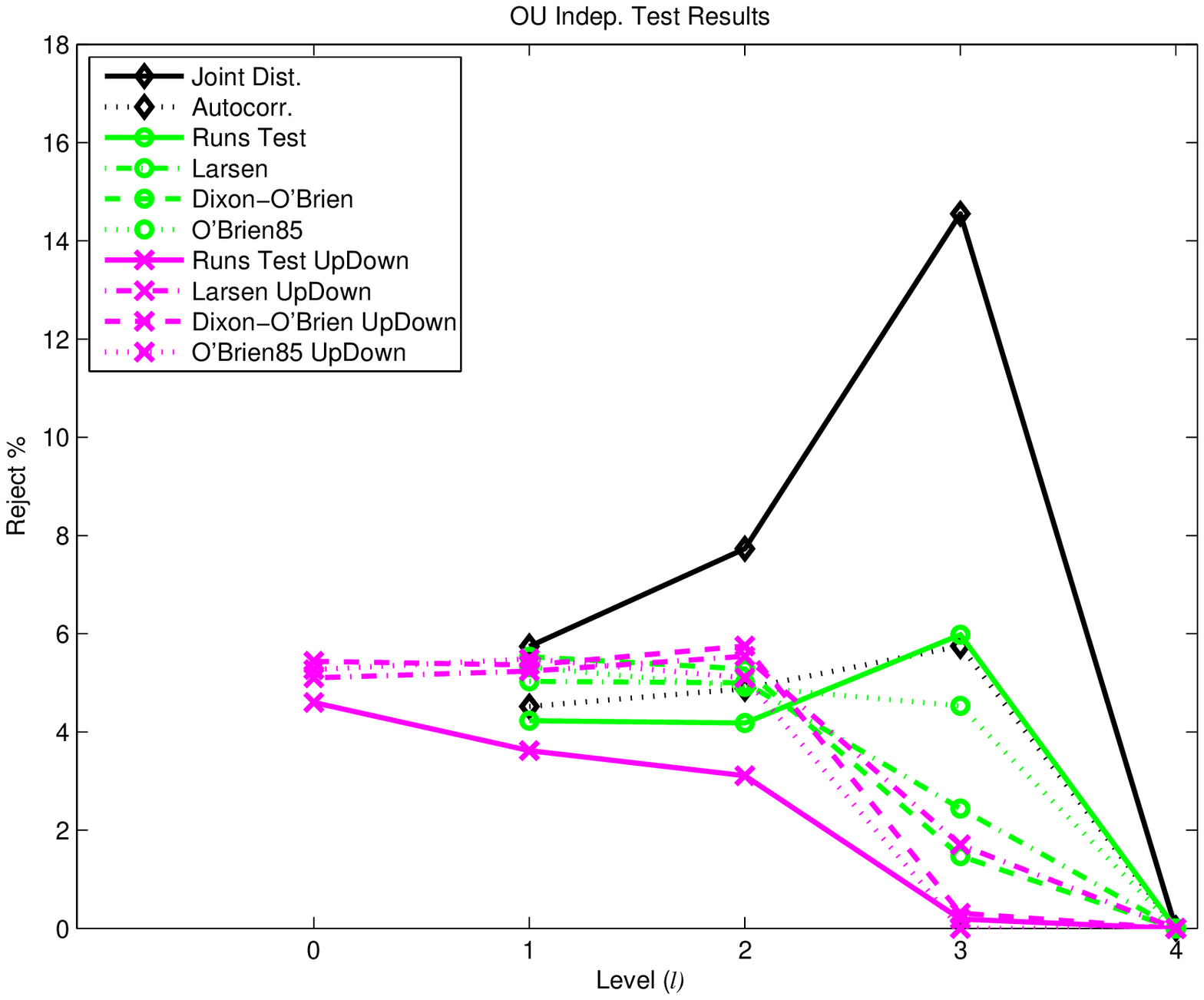}}
\caption{Crossing tree results for the Ornstein-Uhlenbeck process with $\al=8$, $\sigma=1$, $n = 1250$, $\de = 0.063015$. Percentage of 10,000 sample paths rejected using distribution tests (left) and independence tests (right). Here, 95\% confidence intervals are 1\% or smaller. \label{OU1.fig}}
\end{figure}

\begin{table}[p]

 {\small 
 \begin{center} 
 \begin{tabular}{ |c || *{5}{c|}} 
\hline \multicolumn{6}{|c|}{$\al=8$, $\sigma=1$, $n = 1250$, $\de = 0.063015$ }  \\ 
 \hline \hline 
& \multicolumn{5}{c|}{\raisebox{0ex}[12pt]{} {\bf \% of all} (\% of tested; \# tested)} \\ \hline {\bf levels }& 0 & 1 & 2 & 3 & 4 \\ \cline{2-6} 
 \hline \hline\raisebox{0ex}[12pt]{{\bf $\chi^2$ (+2 df)}}&  & {\bf   7.5} & {\bf  11.0} & {\bf   0.7} (  1.8;   4110) & {\bf   0.0} (  NaN;      0)\\ \hline 
{\bf Twos Test} &  & {\bf   5.0} & {\bf   5.6} & {\bf   1.3} & {\bf   0.0} (  0.0;    325)\\ \hline 
{\bf G Test} &  & {\bf   6.1} & {\bf   7.3} & {\bf   6.8} & {\bf   0.0} (  0.0;    325)\\ \hline 
{\bf KS Discrete} &  & {\bf   5.8} & {\bf   7.7} & {\bf   5.6} & {\bf   0.1}\\ \hline 
{\bf KLP98 Test} &  & {\bf   4.7} & {\bf   4.4} & {\bf   8.8} & {\bf   0.0}\\ \hline 
\hline 
{\bf Joint Dist.}&  & {\bf   5.7} & {\bf   7.7} & {\bf  14.5} ( 17.7;   8210) & {\bf   0.0} (  NaN;      0)\\ \hline 
{\bf Autocorr.} &  & {\bf   4.5} & {\bf   4.9} & {\bf   5.8} (  5.8;   9908) & {\bf   0.0} (  NaN;      0)\\ \hline 
{\bf Runs Test} &  & {\bf   4.2} & {\bf   4.2} & {\bf   6.0} & {\bf   0.0}\\ \hline 
{\bf Larsen Test} &  & {\bf   5.0} & {\bf   5.0} & {\bf   2.4} (  2.4;   9965) & {\bf   0.0} (  0.0;    308)\\ \hline 
{\bf  Dix.-OBri.} &  & {\bf   5.5} & {\bf   5.3} & {\bf   1.5} & {\bf   0.0} (  0.0;    325)\\ \hline 
{\bf OBri85} &  & {\bf   5.3} & {\bf   4.9} (  4.9;   9999) & {\bf   4.5} ( 11.9;   3820) & {\bf   0.0} (  NaN;      0)\\ \hline \hline 
{\bf Runs UD} & {\bf   4.6}  & {\bf   3.6} & {\bf   3.1} & {\bf   0.2} & {\bf   0.0}\\ \hline 
{\bf Larsen UD} & {\bf   5.1}  & {\bf   5.2} & {\bf   5.5} & {\bf   1.7} (  1.9;   9144) & {\bf   0.0} (  0.0;    142)\\ \hline 
{\bf  Dix.-OBri. UD} & {\bf   5.4} & {\bf   5.4} & {\bf   5.8} & {\bf   0.3} (  0.3;   9941) & {\bf   0.0} (  0.0;    290)\\ \hline 
{\bf OBri85 UD} & {\bf   5.3}  & {\bf   5.5} & {\bf   5.1} (  6.0;   8538) & {\bf   0.0} (  0.0;    194) & {\bf   0.0} (  NaN;      0)\\ \hline 
\end{tabular} 
 \end{center}
 } 
\caption{Crossing tree results for the Ornstein-Uhlenbeck process with $\al=8$, $\sigma=1$, $n = 1250$, $\de = 0.063015$. Percentage of 10,000 sample paths rejected, with an average of 1250 crossings by time 5. At higher levels, insufficient data length means some datasets are not tested.} \label{OU1.table}
\end{table}

\begin{figure}[htb!]
\resizebox{3in}{!}{\includegraphics{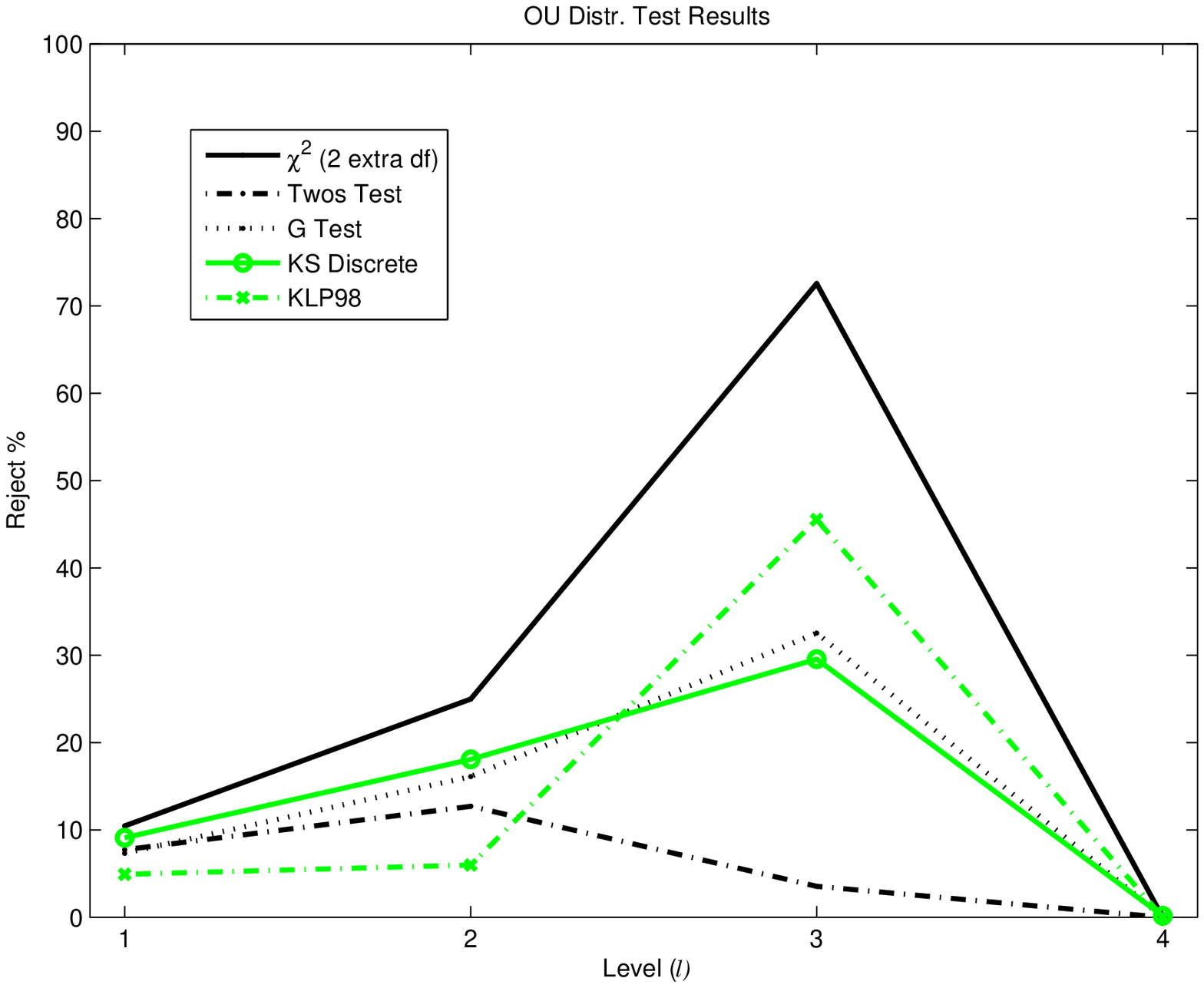}}\resizebox{3in}{!}{\includegraphics{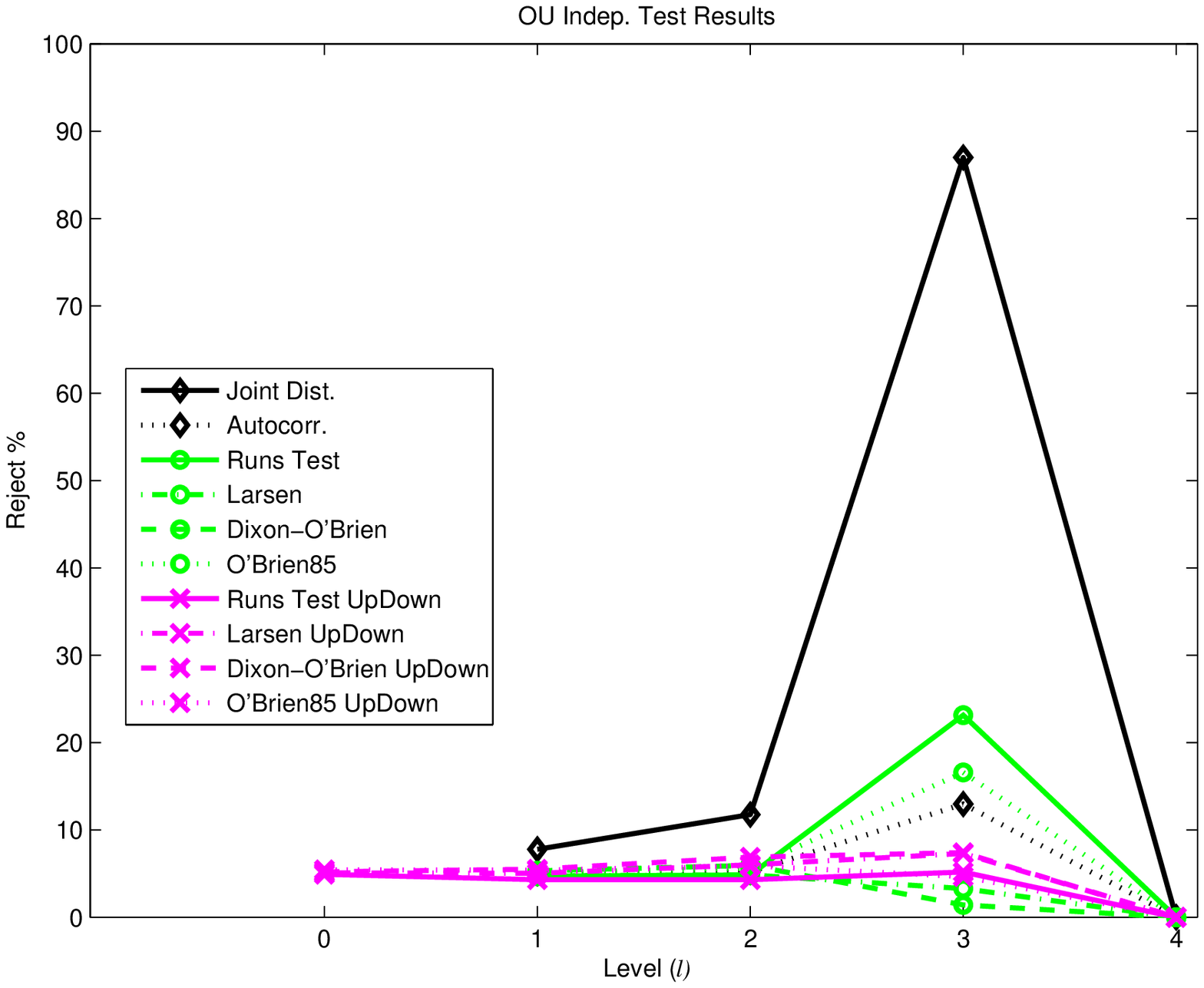}}
\caption{Crossing tree results for the Ornstein-Uhlenbeck process with $\al=8$, $\sigma=1$, $n = 5000$, $\de =  0.063015$. Percentage of 10,000 sample paths rejected using distribution tests (left) and independence tests (right). Here, 95\% confidence intervals are 1\% or smaller. \label{OU3.fig}}
\end{figure}

\begin{table}[thb!]

 {\small 
 \begin{center} 
 \begin{tabular}{ |c || *{5}{c|}} 
\hline \multicolumn{6}{|c|}{$\al=8$, $\sigma=1$, $n = 5000$, $\de = 0.063015$ }  \\ 
 \hline \hline 
& \multicolumn{5}{c|}{\raisebox{0ex}[12pt]{} {\bf \% of all} (\% of tested; \# tested)} \\ \hline {\bf levels }& 0 & 1 & 2 & 3 & 4 \\ \cline{2-6} 
 \hline \hline\raisebox{0ex}[12pt]{{\bf $\chi^2$ (+2 df)}}&  & {\bf  10.5} & {\bf  25.0} & {\bf  72.6} & {\bf   0.0} (  NaN;      0)\\ \hline 
{\bf Twos Test} &  & {\bf   7.7} & {\bf  12.7} & {\bf   3.5} & {\bf   0.0} (  0.0;   1340)\\ \hline 
{\bf G Test} &  & {\bf   7.3} & {\bf  16.1} & {\bf  32.5} & {\bf   0.0} (  0.0;   1340)\\ \hline 
{\bf KS Discrete} &  & {\bf   9.1} & {\bf  18.1} & {\bf  29.5} & {\bf   0.2}\\ \hline 
{\bf KLP98 Test} &  & {\bf   4.9} & {\bf   6.0} & {\bf  45.5} & {\bf   0.0}\\ \hline 
\hline 
{\bf Joint Dist.}&  & {\bf   7.8} & {\bf  11.8} & {\bf  87.0} & {\bf   0.0} (  NaN;      0)\\ \hline 
{\bf Autocorr.} &  & {\bf   5.1} & {\bf   5.2} & {\bf  13.0} & {\bf   0.0} (  0.0;      2)\\ \hline 
{\bf Runs Test} &  & {\bf   4.7} & {\bf   4.8} & {\bf  23.1} & {\bf   0.0}\\ \hline 
{\bf Larsen Test} &  & {\bf   4.7} & {\bf   4.9} & {\bf   3.3} & {\bf   0.0} (  0.0;   1314)\\ \hline 
{\bf  Dix.-OBri.} &  & {\bf   5.3} & {\bf   5.9} & {\bf   1.4} & {\bf   0.0} (  0.0;   1340)\\ \hline 
{\bf OBri85} &  & {\bf   5.1} & {\bf   5.6} & {\bf  16.6} ( 16.8;   9851) & {\bf   0.0} (  NaN;      0)\\ \hline \hline 
{\bf Runs UD} & {\bf   4.9}  & {\bf   4.3} & {\bf   4.3} & {\bf   5.2} & {\bf   0.0}\\ \hline 
{\bf Larsen UD} & {\bf   5.1}  & {\bf   5.0} & {\bf   6.0} & {\bf   7.3} (  7.3;   9988) & {\bf   0.0} (  0.0;    662)\\ \hline 
{\bf  Dix.-OBri. UD} & {\bf   5.2} & {\bf   5.5} & {\bf   6.9} & {\bf   7.4} & {\bf   0.0} (  0.0;   1303)\\ \hline 
{\bf OBri85 UD} & {\bf   5.4}  & {\bf   5.4} & {\bf   5.9} & {\bf   4.7} (  5.6;   8298) & {\bf   0.0} (  NaN;      0)\\ \hline 
\end{tabular} 
 \end{center}
 } 
\caption{Crossing tree results for the Ornstein-Uhlenbeck process with $\al=8$, $\sigma=1$, $n = 5000$, $\de =  0.063015$. Percentage of 10,000 sample paths rejected, with an average of 5000 crossings by time 20. At higher levels, insufficient data length means some datasets are not tested. } \label{OU3.table}
\end{table}

\begin{figure}[htb!]
\resizebox{3in}{!}{\includegraphics{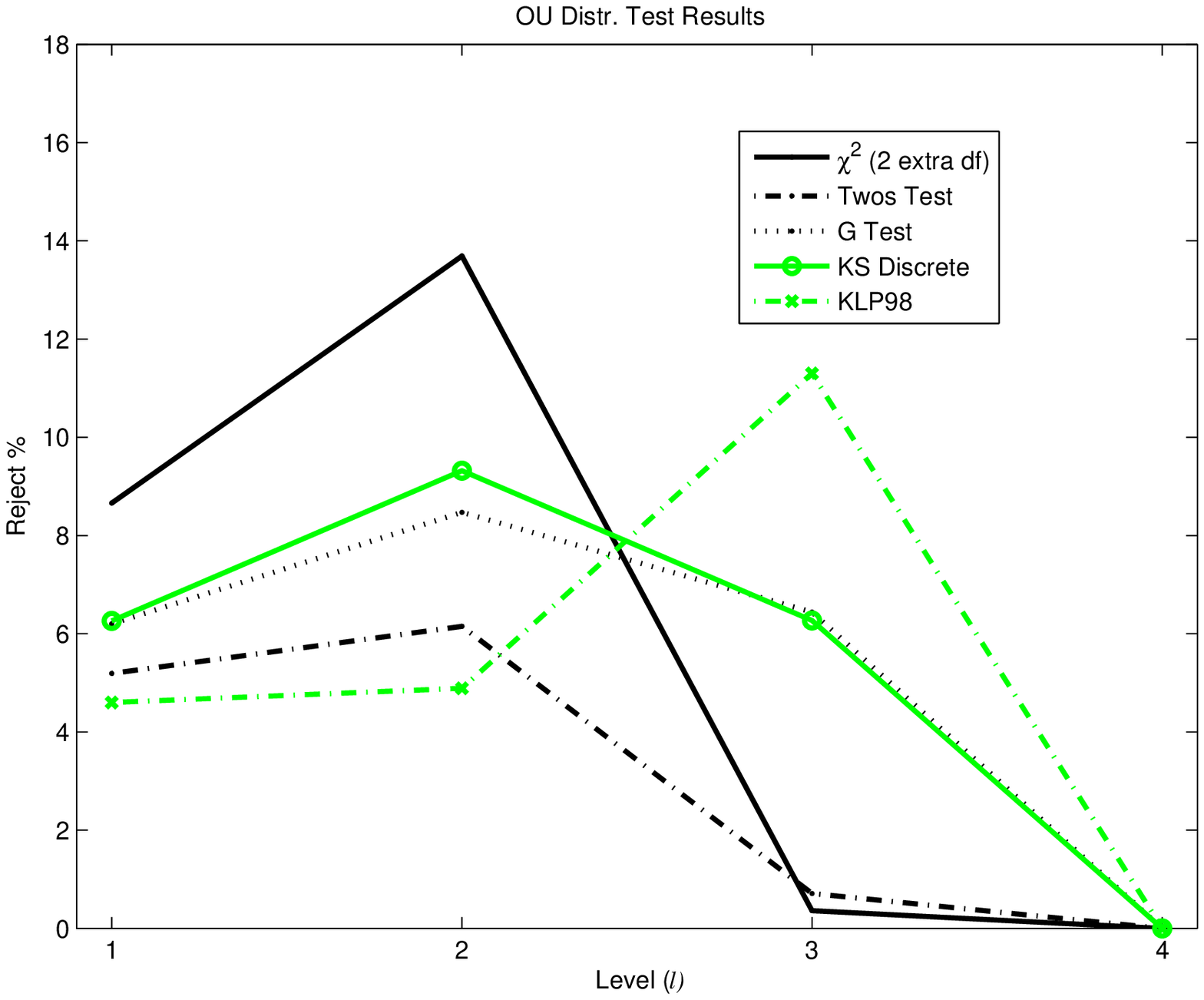}}\resizebox{3in}{!}{\includegraphics{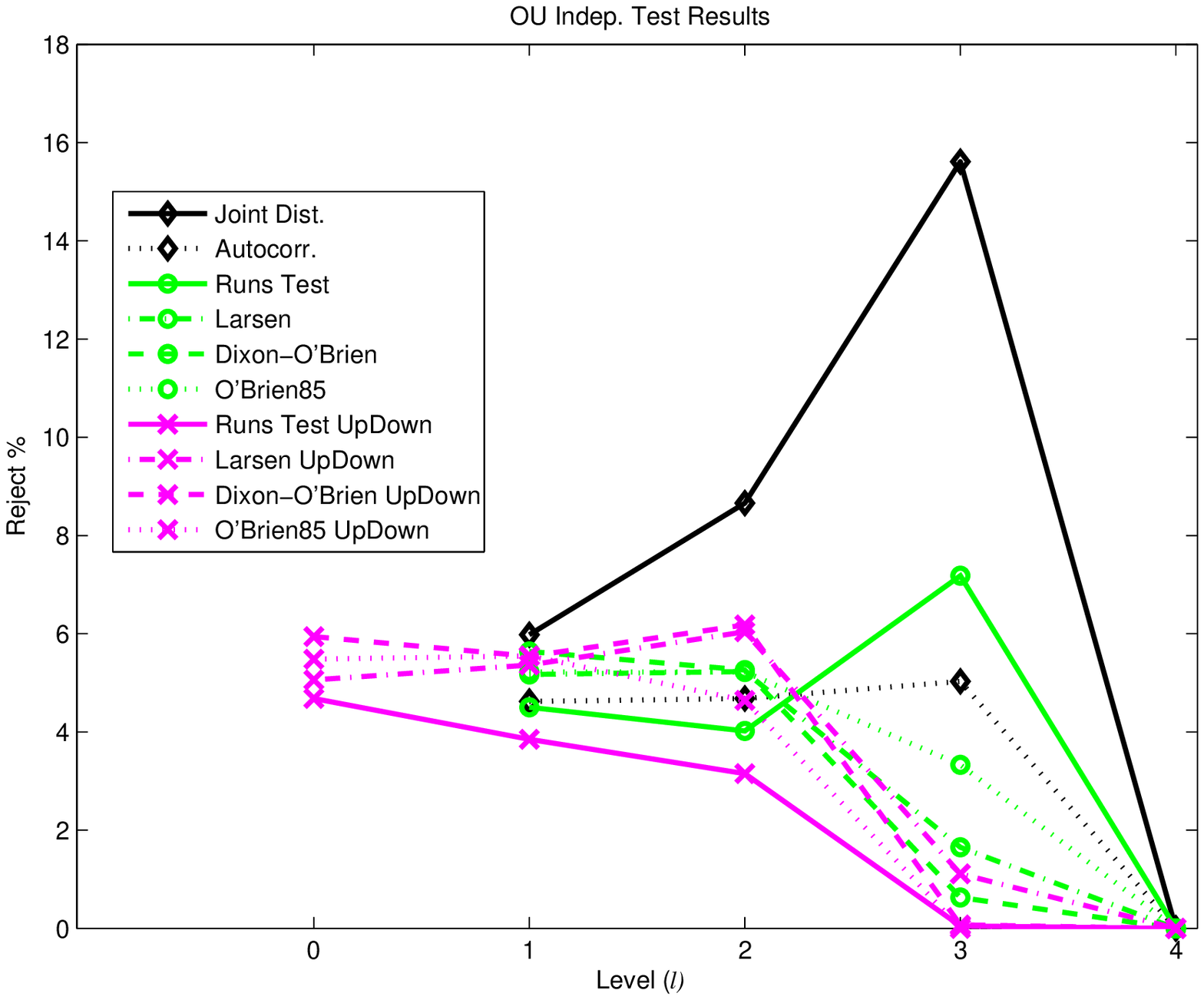}}
\caption{Crossing tree results for the Ornstein-Uhlenbeck process with $\al=10$, $\sigma=1$, $n = 1250$, $\de = 0.062945$. Percentage of 10,000 sample paths rejected using distribution tests (left) and independence tests (right). Here, 95\% confidence intervals are 1\% or smaller. \label{OU2.fig}}
\end{figure}

\begin{table}[thb!]

 {\small 
 \begin{center} 
 \begin{tabular}{ |c || *{5}{c|}} 
\hline \multicolumn{6}{|c|}{$\al=10$, $\sigma=1$, $n = 1250$, $\de = 0.062945$ }  \\ 
 \hline \hline 
& \multicolumn{5}{c|}{\raisebox{0ex}[12pt]{} {\bf \% of all} (\% of tested; \# tested)} \\ \hline {\bf levels }& 0 & 1 & 2 & 3 & 4 \\ \cline{2-6} 
 \hline \hline\raisebox{0ex}[12pt]{{\bf $\chi^2$ (+2 df)}}&  & {\bf   8.7} & {\bf  13.7} & {\bf   0.4} (  1.6;   2275) & {\bf   0.0} (  NaN;      0)\\ \hline 
{\bf Twos Test} &  & {\bf   5.2} & {\bf   6.2} & {\bf   0.7} (  0.7;   9988) & {\bf   0.0} (  0.0;     58)\\ \hline 
{\bf G Test} &  & {\bf   6.2} & {\bf   8.5} & {\bf   6.4} (  6.4;   9988) & {\bf   0.0} (  0.0;     58)\\ \hline 
{\bf KS Discrete} &  & {\bf   6.3} & {\bf   9.3} & {\bf   6.3} & {\bf   0.0}\\ \hline 
{\bf KLP98 Test} &  & {\bf   4.6} & {\bf   4.9} & {\bf  11.3} & {\bf   0.0}\\ \hline 
\hline 
{\bf Joint Dist.}&  & {\bf   6.0} & {\bf   8.7} & {\bf  15.6} ( 24.4;   6409) & {\bf   0.0} (  NaN;      0)\\ \hline 
{\bf Autocorr.} &  & {\bf   4.6} & {\bf   4.7} & {\bf   5.0} (  5.3;   9577) & {\bf   0.0} (  NaN;      0)\\ \hline 
{\bf Runs Test} &  & {\bf   4.5} & {\bf   4.0} & {\bf   7.2} & {\bf   0.0}\\ \hline 
{\bf Larsen Test} &  & {\bf   5.2} & {\bf   5.2} & {\bf   1.7} (  1.7;   9923) & {\bf   0.0} (  0.0;     56)\\ \hline 
{\bf  Dix.-OBri.} &  & {\bf   5.6} & {\bf   5.3} & {\bf   0.6} (  0.6;   9988) & {\bf   0.0} (  0.0;     58)\\ \hline 
{\bf OBri85} &  & {\bf   5.2} & {\bf   5.2} (  5.2;   9999) & {\bf   3.3} ( 14.6;   2274) & {\bf   0.0} (  NaN;      0)\\ \hline \hline 
{\bf Runs UD} & {\bf   4.7}  & {\bf   3.9} & {\bf   3.1} & {\bf   0.1} & {\bf   0.0}\\ \hline 
{\bf Larsen UD} & {\bf   5.1}  & {\bf   5.4} & {\bf   6.0} & {\bf   1.1} (  1.2;   9004) & {\bf   0.0} (  0.0;     27)\\ \hline 
{\bf  Dix.-OBri. UD} & {\bf   5.9} & {\bf   5.5} & {\bf   6.2} & {\bf   0.1} (  0.1;   9906) & {\bf   0.0} (  0.0;     54)\\ \hline 
{\bf OBri85 UD} & {\bf   5.5}  & {\bf   5.5} & {\bf   4.6} (  5.4;   8617) & {\bf   0.0} (  0.0;     80) & {\bf   0.0} (  NaN;      0)\\ \hline 
\end{tabular} 
 \end{center}
 } 
\caption{Crossing tree results for the Ornstein-Uhlenbeck process with $\al=10$, $\sigma=1$, $n = 1250$, $\de = 0.062945$. Percentage of 10,000 sample paths rejected, with an average of 1250 crossings by time 5. At higher levels, insufficient data length means some datasets are not tested. } \label{OU2.table}
\end{table}

\begin{figure}[htb!]
\resizebox{3in}{!}{\includegraphics{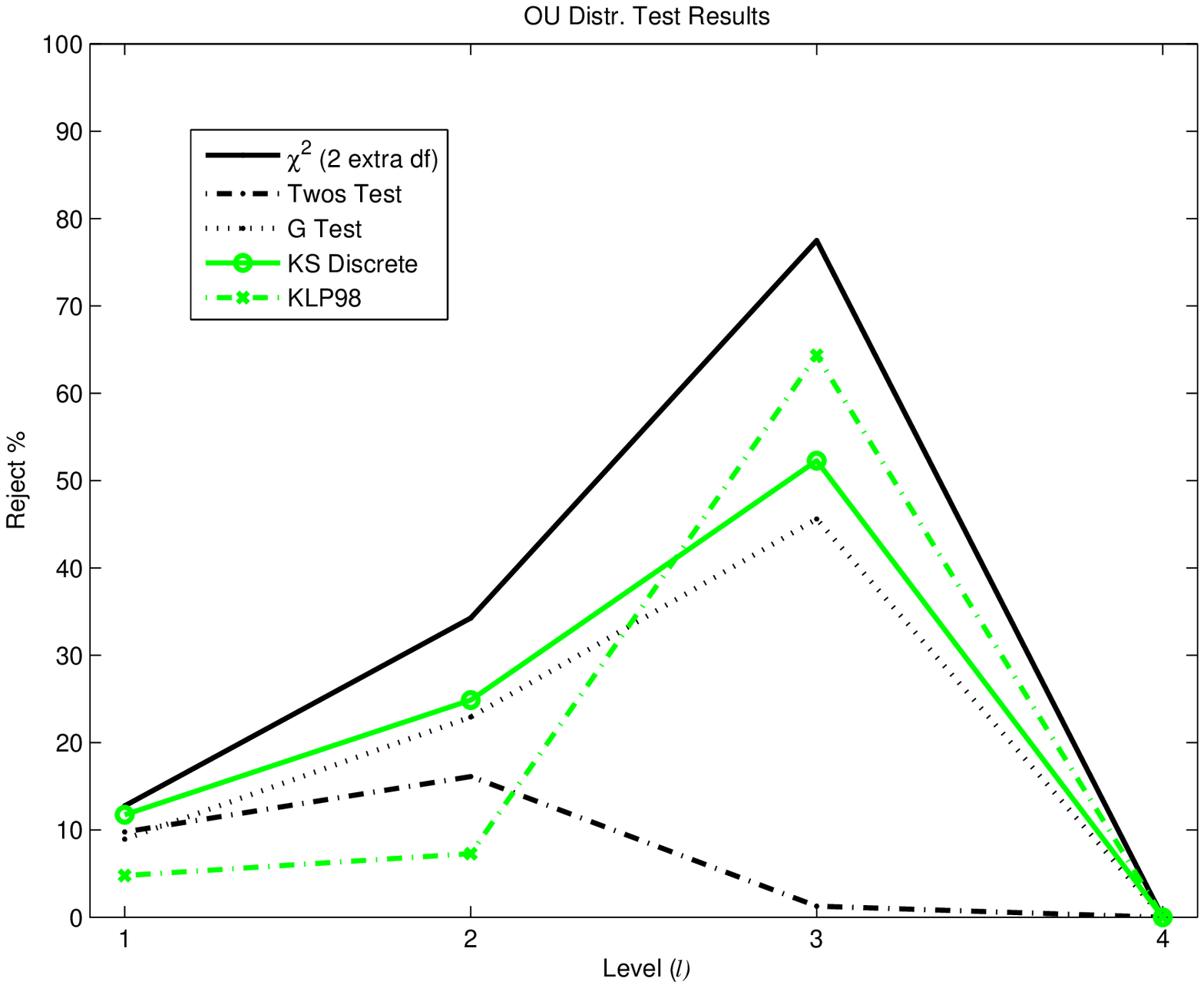}}\resizebox{3in}{!}{\includegraphics{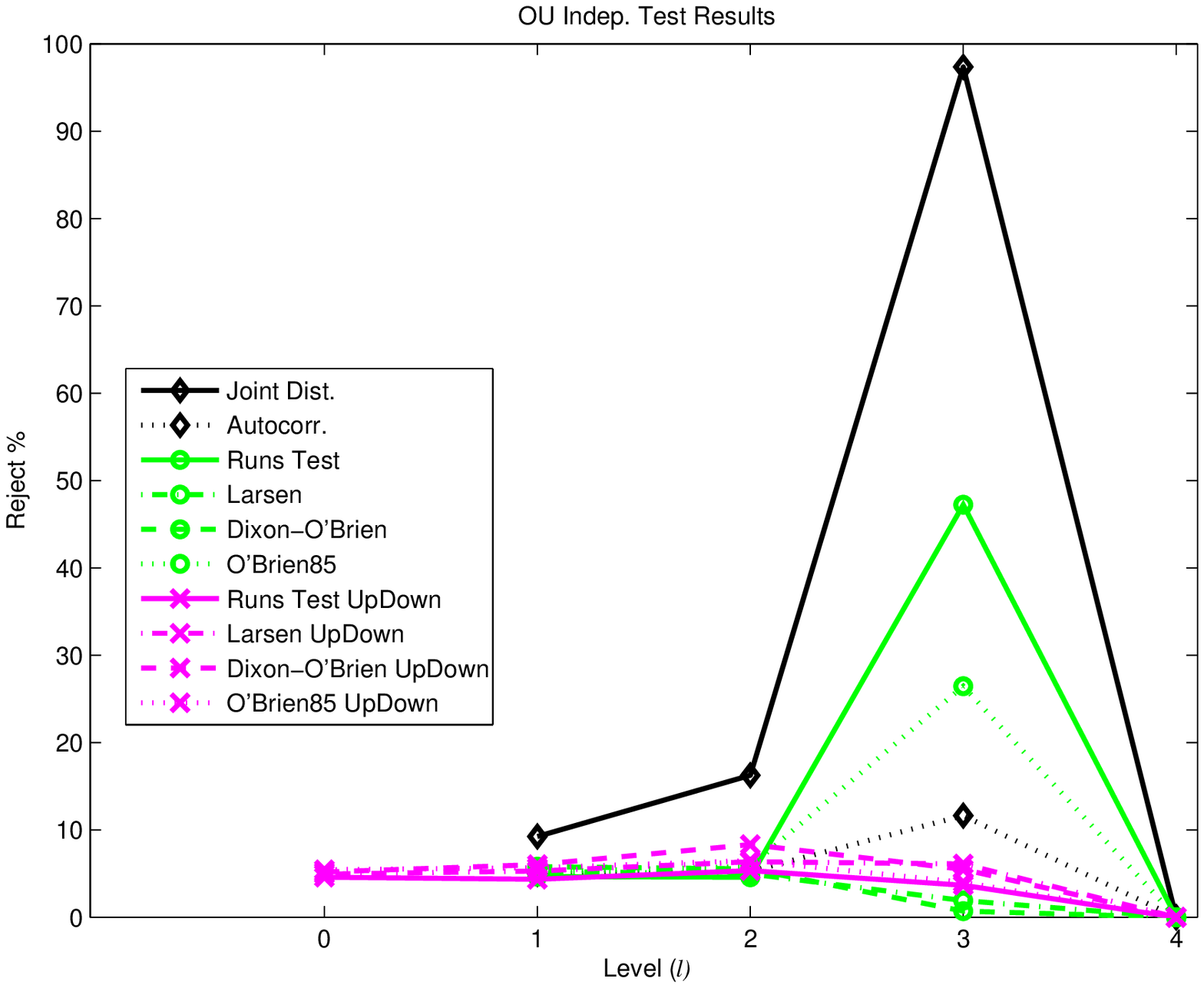}}
\caption{Crossing tree results for the Ornstein-Uhlenbeck process with $\al=10$, $\sigma=1$, $n = 5000$, $\de = 0.062945$. Percentage of 10,000 sample paths rejected using distribution tests (left) and independence tests (right). Here, 95\% confidence intervals are 1\% or smaller. \label{OU4.fig}}
\end{figure}

\begin{table}[thb!]

 {\small 
 \begin{center} 
 \begin{tabular}{ |c || *{5}{c|}} 
\hline \multicolumn{6}{|c|}{$\al=10$, $\sigma=1$, $n = 5000$, $\de = 0.062945$ }  \\ 
 \hline \hline 
& \multicolumn{5}{c|}{\raisebox{0ex}[12pt]{} {\bf \% of all} (\% of tested; \# tested)} \\ \hline {\bf levels }& 0 & 1 & 2 & 3 & 4 \\ \cline{2-6} 
 \hline \hline\raisebox{0ex}[12pt]{{\bf $\chi^2$ (+2 df)}}&  & {\bf  12.7} & {\bf  34.3} & {\bf  77.5} & {\bf   0.0} (  NaN;      0)\\ \hline 
{\bf Twos Test} &  & {\bf   9.8} & {\bf  16.1} & {\bf   1.3} & {\bf   0.0} (  0.0;    250)\\ \hline 
{\bf G Test} &  & {\bf   8.9} & {\bf  22.9} & {\bf  45.6} & {\bf   0.0} (  0.0;    250)\\ \hline 
{\bf KS Discrete} &  & {\bf  11.7} & {\bf  24.9} & {\bf  52.3} & {\bf   0.0}\\ \hline 
{\bf KLP98 Test} &  & {\bf   4.8} & {\bf   7.3} & {\bf  64.3} & {\bf   0.0}\\ \hline 
\hline 
{\bf Joint Dist.}&  & {\bf   9.3} & {\bf  16.3} & {\bf  97.4} & {\bf   0.0} (  NaN;      0)\\ \hline 
{\bf Autocorr.} &  & {\bf   5.0} & {\bf   5.2} & {\bf  11.6} & {\bf   0.0} (  0.0;      1)\\ \hline 
{\bf Runs Test} &  & {\bf   4.7} & {\bf   4.6} & {\bf  47.2} & {\bf   0.0}\\ \hline 
{\bf Larsen Test} &  & {\bf   5.2} & {\bf   5.0} & {\bf   1.9} & {\bf   0.0} (  0.0;    250)\\ \hline 
{\bf  Dix.-OBri.} &  & {\bf   5.8} & {\bf   5.5} & {\bf   0.7} & {\bf   0.0} (  0.0;    250)\\ \hline 
{\bf OBri85} &  & {\bf   5.2} & {\bf   5.8} & {\bf  26.4} ( 28.2;   9374) & {\bf   0.0} (  NaN;      0)\\ \hline \hline 
{\bf Runs UD} & {\bf   4.6}  & {\bf   4.3} & {\bf   5.3} & {\bf   3.6} & {\bf   0.0}\\ \hline 
{\bf Larsen UD} & {\bf   4.9}  & {\bf   5.3} & {\bf   6.3} & {\bf   6.1} (  6.1;   9970) & {\bf   0.0} (  0.0;    117)\\ \hline 
{\bf  Dix.-OBri. UD} & {\bf   5.2} & {\bf   6.0} & {\bf   8.3} & {\bf   5.5} & {\bf   0.0} (  0.0;    246)\\ \hline 
{\bf OBri85 UD} & {\bf   5.4}  & {\bf   5.7} & {\bf   6.5} & {\bf   4.0} (  5.1;   7809) & {\bf   0.0} (  NaN;      0)\\ \hline 
\end{tabular} 
 \end{center}
 } 
\caption{Crossing tree results for the Ornstein-Uhlenbeck process with $\al=10$, $\sigma=1$, $n = 5000$, $\de = 0.062945$. Percentage of 10,000 sample paths rejected, with an average of 5000 crossings by time 20. At higher levels, insufficient data length means some datasets are not tested. } \label{OU4.table}
\end{table}

\begin{figure}[htb!]
\resizebox{3in}{!}{\includegraphics{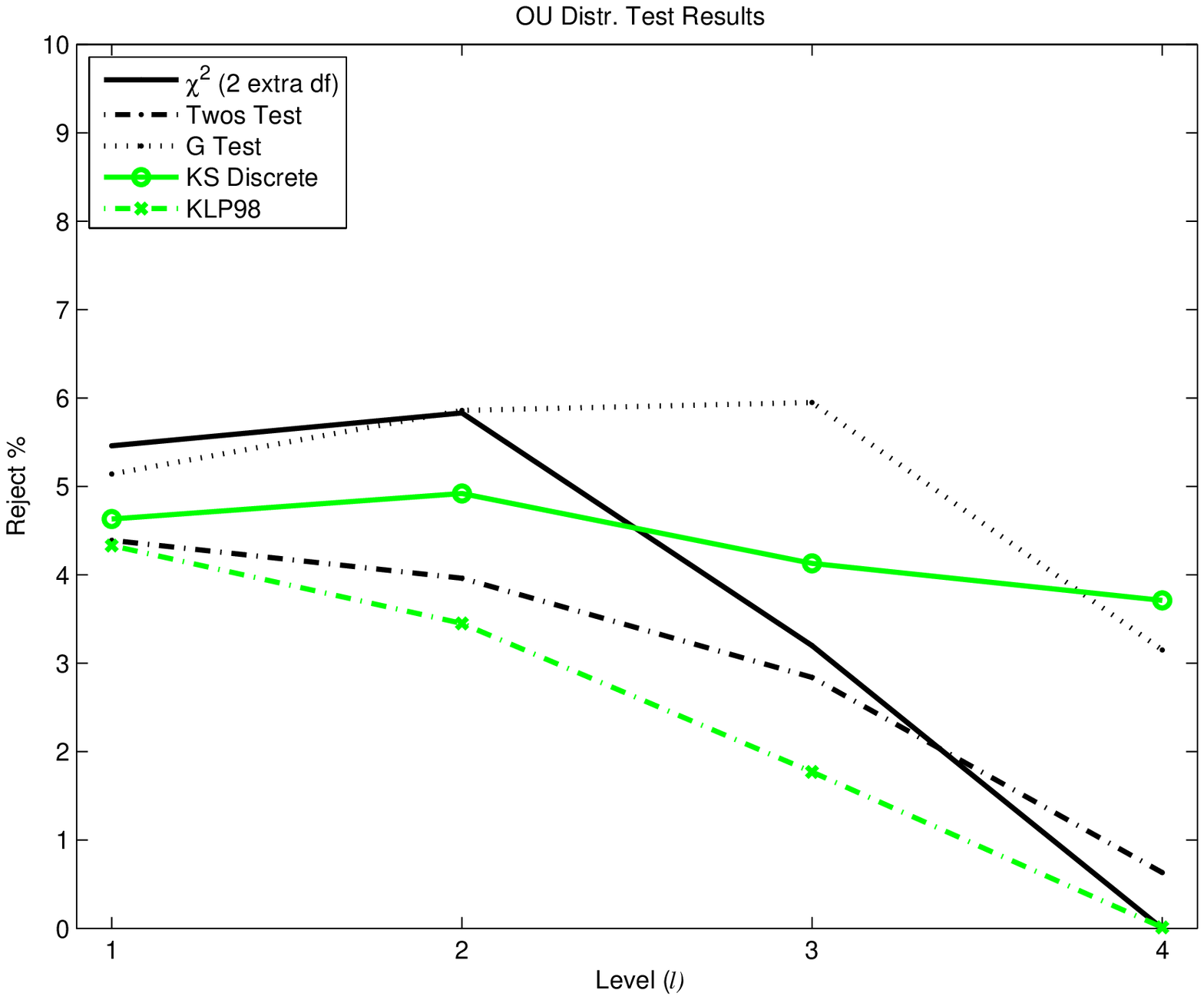}}\resizebox{3in}{!}{\includegraphics{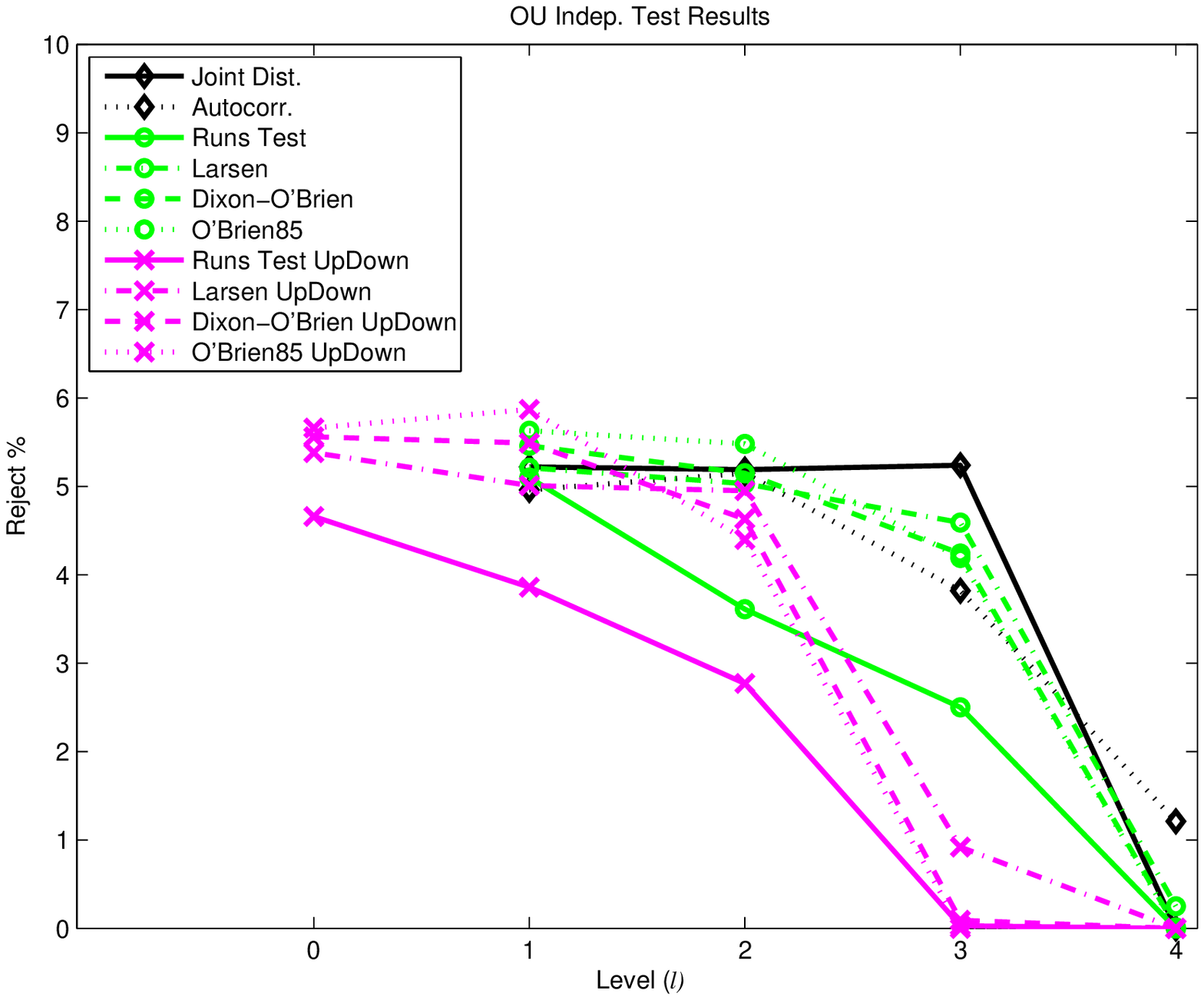}}
\caption{Crossing tree results for the Ornstein-Uhlenbeck process with $\al=1$, $\sigma=1$, $n = 1250$, $\de = 0.063220$. Percentage of 10,000 sample paths rejected using distribution tests (left) and independence tests (right). Here, 95\% confidence intervals are 1\% or smaller. \label{OU6.fig}}
\end{figure}

\begin{table}[thb!]

 {\small 
 \begin{center} 
 \begin{tabular}{ |c || *{5}{c|}} 
\hline \multicolumn{6}{|c|}{$\al=1$, $\sigma=1$, $n = 1250$, $\de = 0.063220$ }  \\ 
 \hline \hline 
& \multicolumn{5}{c|}{\raisebox{0ex}[12pt]{} {\bf \% of all} (\% of tested; \# tested)} \\ \hline {\bf levels }& 0 & 1 & 2 & 3 & 4 \\ \cline{2-6} 
 \hline \hline\raisebox{0ex}[12pt]{{\bf $\chi^2$ (+2 df)}}&  & {\bf   5.5} & {\bf   5.8} & {\bf   3.2} (  3.6;   8964) & {\bf   0.0} (  NaN;      0)\\ \hline 
{\bf Twos Test} &  & {\bf   4.4} & {\bf   4.0} & {\bf   2.8} & {\bf   0.6} (  0.6;   9794)\\ \hline 
{\bf G Test} &  & {\bf   5.1} & {\bf   5.9} & {\bf   5.9} & {\bf   3.1} (  3.2;   9794)\\ \hline 
{\bf KS Discrete} &  & {\bf   4.6} & {\bf   4.9} & {\bf   4.1} & {\bf   3.7}\\ \hline 
{\bf KLP98 Test} &  & {\bf   4.3} & {\bf   3.5} & {\bf   1.8} & {\bf   0.0}\\ \hline 
\hline 
{\bf Joint Dist.}&  & {\bf   5.2} & {\bf   5.2} & {\bf   5.2} (  5.3;   9953) & {\bf   0.0} (  0.0;     16)\\ \hline 
{\bf Autocorr.} &  & {\bf   5.0} & {\bf   5.1} & {\bf   3.8} & {\bf   1.2} (  4.0;   3044)\\ \hline 
{\bf Runs Test} &  & {\bf   5.1} & {\bf   3.6} & {\bf   2.5} & {\bf   0.0}\\ \hline 
{\bf Larsen Test} &  & {\bf   5.2} & {\bf   5.0} & {\bf   4.6} (  4.6;   9992) & {\bf   0.3} (  0.3;   8063)\\ \hline 
{\bf  Dix.-OBri.} &  & {\bf   5.5} & {\bf   5.1} & {\bf   4.2} & {\bf   0.0} (  0.0;   9794)\\ \hline 
{\bf OBri85} &  & {\bf   5.6} & {\bf   5.5} & {\bf   4.2} (  5.5;   7632) & {\bf   0.0} (  0.0;      5)\\ \hline \hline 
{\bf Runs UD} & {\bf   4.7}  & {\bf   3.9} & {\bf   2.8} & {\bf   0.0} & {\bf   0.0}\\ \hline 
{\bf Larsen UD} & {\bf   5.4}  & {\bf   5.0} & {\bf   5.0} (  5.0;   9999) & {\bf   0.9} (  1.0;   9274) & {\bf   0.0} (  0.0;   4678)\\ \hline 
{\bf  Dix.-OBri. UD} & {\bf   5.6} & {\bf   5.5} & {\bf   4.6} & {\bf   0.1} (  0.1;   9984) & {\bf   0.0} (  0.0;   7612)\\ \hline 
{\bf OBri85 UD} & {\bf   5.7}  & {\bf   5.9} & {\bf   4.4} (  5.5;   7962) & {\bf   0.0} (  0.0;     69) & {\bf   0.0} (  NaN;      0)\\ \hline 
\end{tabular} 
 \end{center}
 } 
\caption{Crossing tree results for the Ornstein-Uhlenbeck process with $\al=1$, $\sigma=1$, $n = 1250$, $\de = 0.063220$. Percentage of 10,000 sample paths rejected, with an average of 1250 crossings by time 5. At higher levels, insufficient data length means some datasets are not tested. } \label{OU6.table}
\end{table}

\begin{figure}[htb!]
\resizebox{3in}{!}{\includegraphics{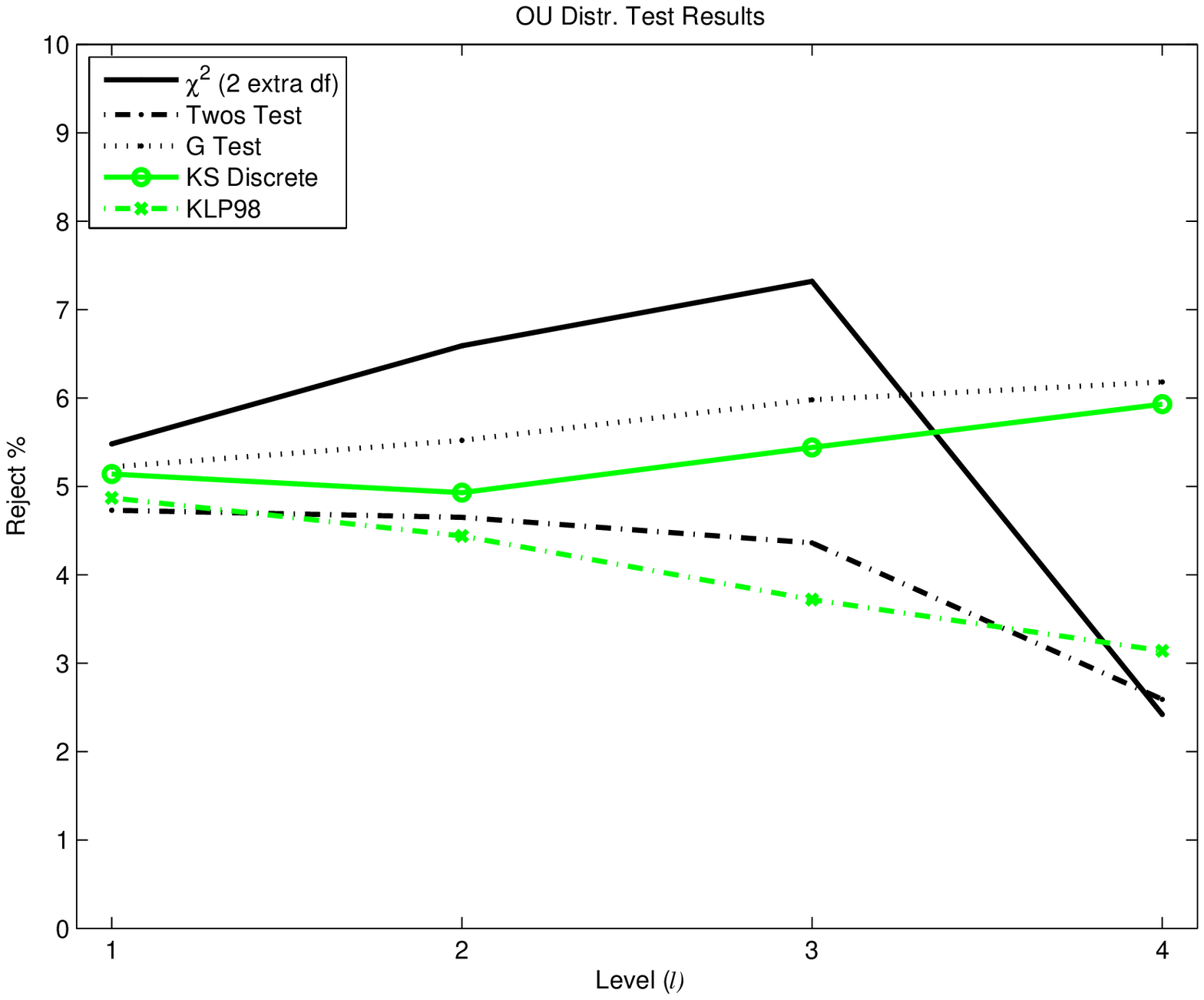}}\resizebox{3in}{!}{\includegraphics{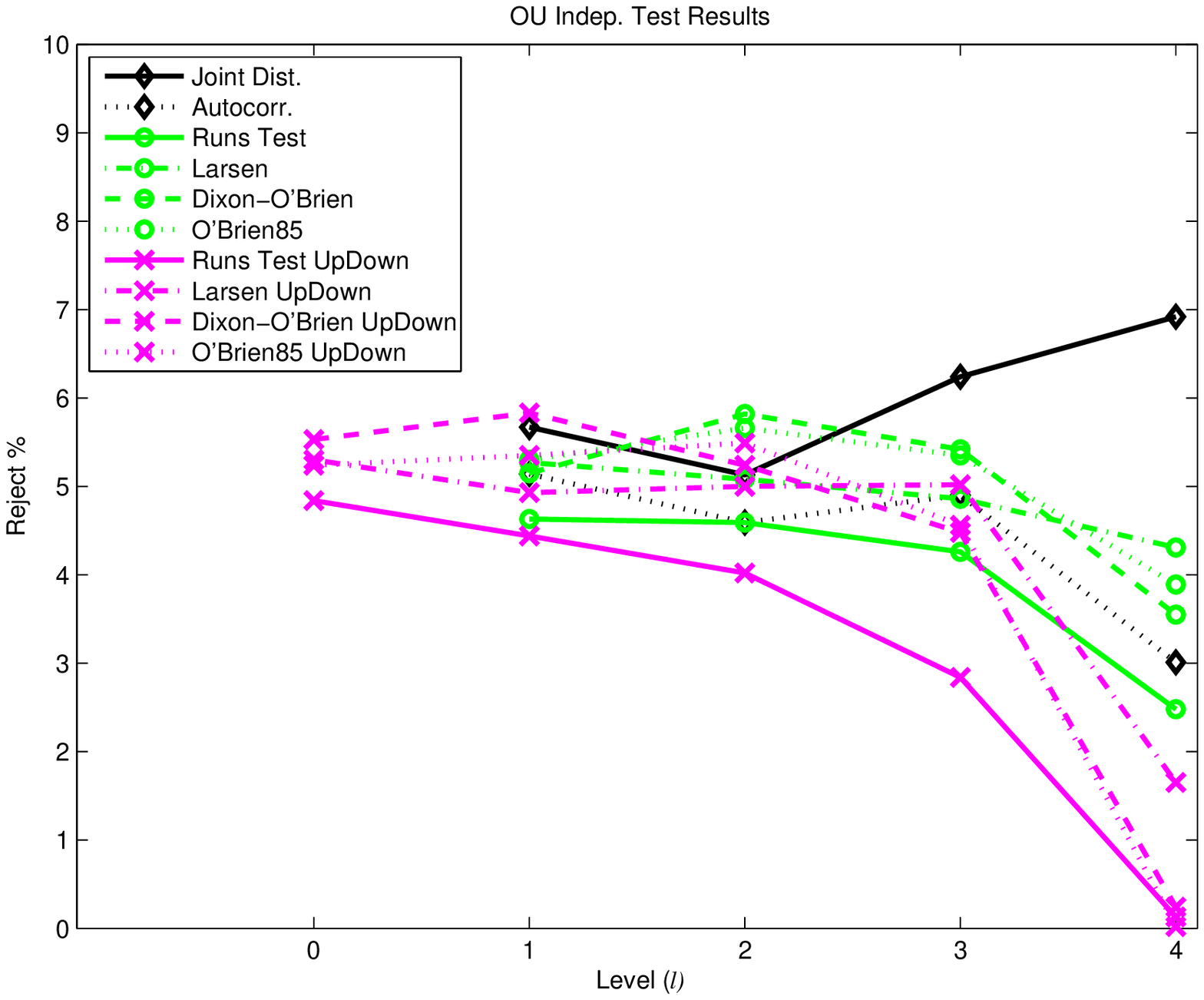}}
\caption{Crossing tree results for the Ornstein-Uhlenbeck process with $\al=1$, $\sigma=1$, $n = 5000$, $\de = 0.063220$. Percentage of 10,000 sample paths rejected using distribution tests (left) and independence tests (right). Here, 95\% confidence intervals are 1\% or smaller. \label{OU5.fig}}
\end{figure}

\begin{table}[thb!]

 {\small 
 \begin{center} 
 \begin{tabular}{ |c || *{5}{c|}} 
\hline \multicolumn{6}{|c|}{$\al=1$, $\sigma=1$, $n = 5000$, $\de = 0.063220$ }  \\ 
 \hline \hline 
& \multicolumn{5}{c|}{\raisebox{0ex}[12pt]{} {\bf \% of all} (\% of tested; \# tested)} \\ \hline {\bf levels }& 0 & 1 & 2 & 3 & 4 \\ \cline{2-6} 
 \hline \hline\raisebox{0ex}[12pt]{{\bf $\chi^2$ (+2 df)}}&  & {\bf   5.5} & {\bf   6.6} & {\bf   7.3} & {\bf   2.4} (  3.2;   7642)\\ \hline 
{\bf Twos Test} &  & {\bf   4.7} & {\bf   4.7} & {\bf   4.4} & {\bf   2.6}\\ \hline 
{\bf G Test} &  & {\bf   5.2} & {\bf   5.5} & {\bf   6.0} & {\bf   6.2}\\ \hline 
{\bf KS Discrete} &  & {\bf   5.1} & {\bf   4.9} & {\bf   5.4} & {\bf   5.9}\\ \hline 
{\bf KLP98 Test} &  & {\bf   4.9} & {\bf   4.4} & {\bf   3.7} & {\bf   3.1}\\ \hline 
\hline 
{\bf Joint Dist.}&  & {\bf   5.7} & {\bf   5.1} & {\bf   6.2} & {\bf   6.9} (  7.1;   9725)\\ \hline 
{\bf Autocorr.} &  & {\bf   5.1} & {\bf   4.6} & {\bf   4.9} & {\bf   3.0} (  3.0;   9999)\\ \hline 
{\bf Runs Test} &  & {\bf   4.6} & {\bf   4.6} & {\bf   4.3} & {\bf   2.5}\\ \hline 
{\bf Larsen Test} &  & {\bf   5.3} & {\bf   5.1} & {\bf   4.9} & {\bf   4.3} (  4.3;   9989)\\ \hline 
{\bf  Dix.-OBri.} &  & {\bf   5.1} & {\bf   5.8} & {\bf   5.4} & {\bf   3.5}\\ \hline 
{\bf OBri85} &  & {\bf   5.3} & {\bf   5.7} & {\bf   5.3} & {\bf   3.9} (  5.9;   6624)\\ \hline \hline 
{\bf Runs UD} & {\bf   4.8}  & {\bf   4.4} & {\bf   4.0} & {\bf   2.8} & {\bf   0.1}\\ \hline 
{\bf Larsen UD} & {\bf   5.3}  & {\bf   4.9} & {\bf   5.0} & {\bf   5.0} (  5.0;   9999) & {\bf   1.7} (  1.8;   9319)\\ \hline 
{\bf  Dix.-OBri. UD} & {\bf   5.5} & {\bf   5.8} & {\bf   5.2} & {\bf   4.5} & {\bf   0.2} (  0.2;   9968)\\ \hline 
{\bf OBri85 UD} & {\bf   5.2}  & {\bf   5.3} & {\bf   5.5} & {\bf   4.6} (  5.4;   8392) & {\bf   0.0} (  0.7;    295)\\ \hline 
\end{tabular} 
 \end{center}
 } 
\caption{Crossing tree results for the Ornstein-Uhlenbeck process with $\al=1$, $\sigma=1$, $n = 5000$, $\de = 0.063220$. Percentage of 10,000 sample paths rejected, with an average of 5000 crossings by time 20. At higher levels, insufficient data length means some datasets are not tested. } \label{OU5.table}
\end{table}

\begin{figure}[hb!]
\resizebox{3in}{!}{\includegraphics{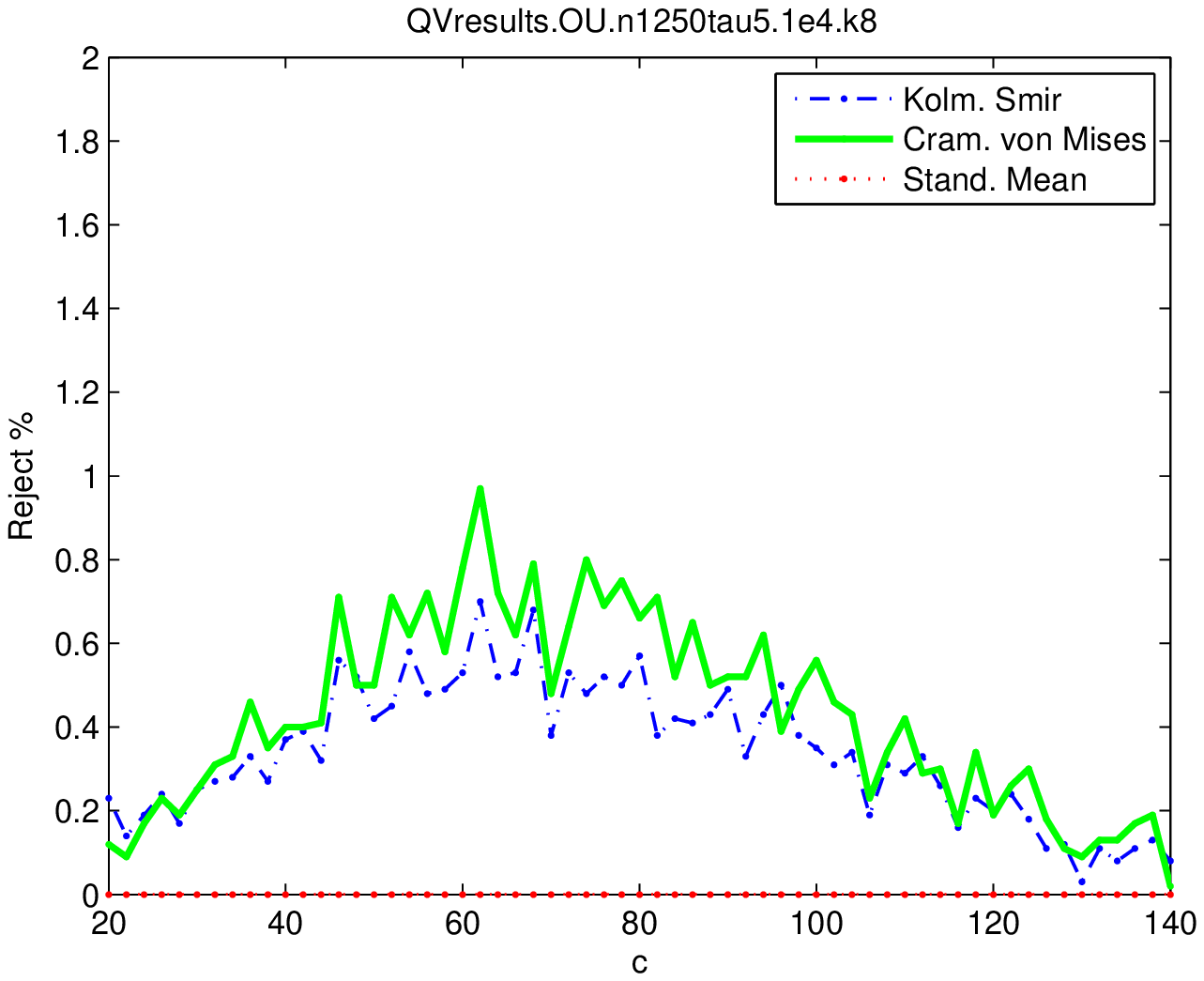}}\resizebox{3in}{!}{\includegraphics{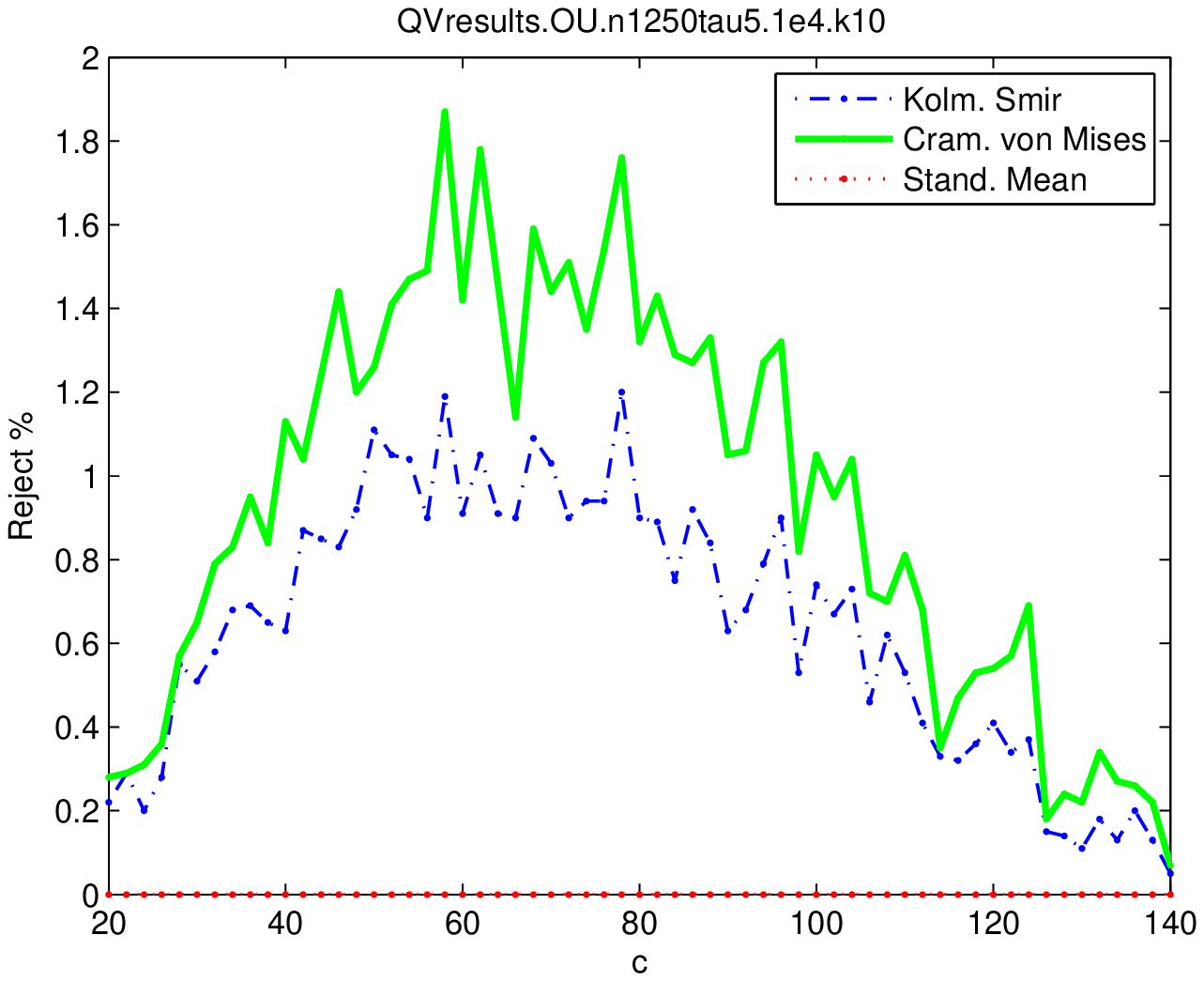}}
\caption{Rejection rates using sample quadratic variation for Ornstein-Uhlenbeck processes. Short ($n=1250$) datasets. Left has drift coefficient $\al=8$, right has drift coefficient $\al=10$. Here, 95\% confidence intervals are $\leq 1\%$. \label{OUshort.fig}}
\end{figure}

\begin{figure}[hb!]
\resizebox{3in}{!}{\includegraphics{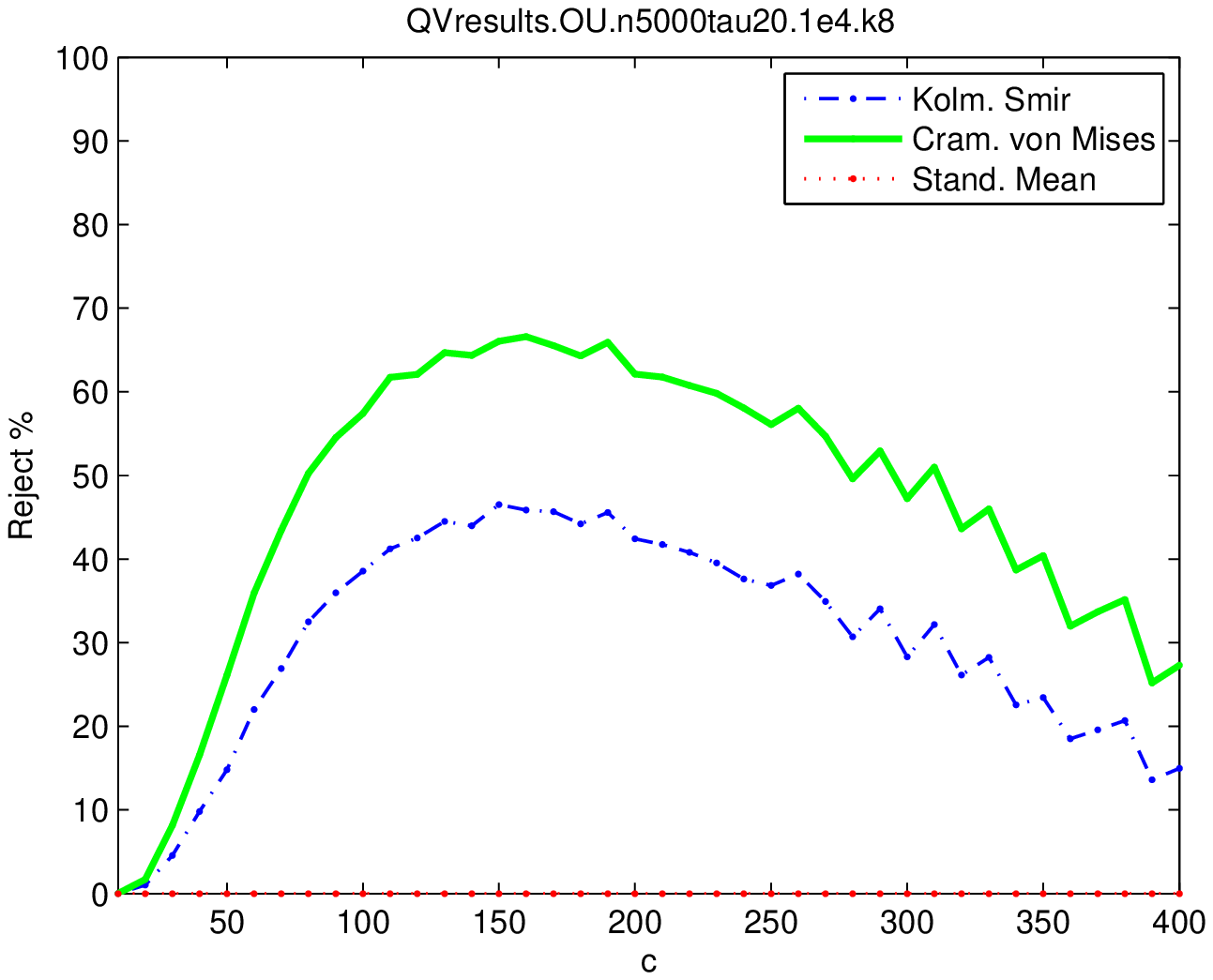}}\resizebox{3in}{!}{\includegraphics{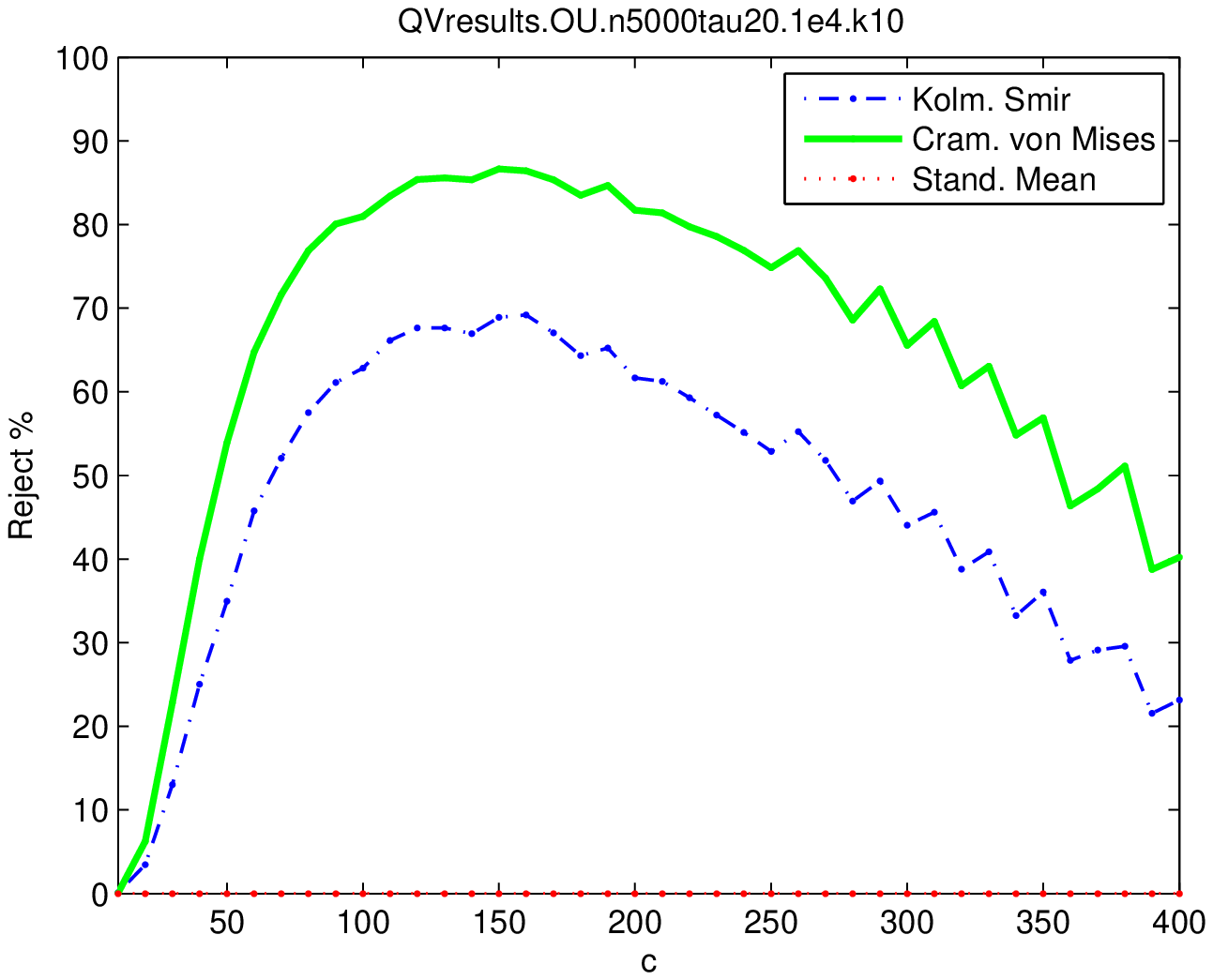}}
\caption{Rejection rates using sample quadratic variation for Ornstein-Uhlenbeck processes. Long ($n=5000$) datasets. Left has drift coefficient $\al=8$, right has drift coefficient $\al=10$. Here, 95\% confidence intervals are $\leq 1\%$. \label{OUlong.fig}}
\end{figure}

\begin{figure}[hb!]
\resizebox{3in}{!}{\includegraphics{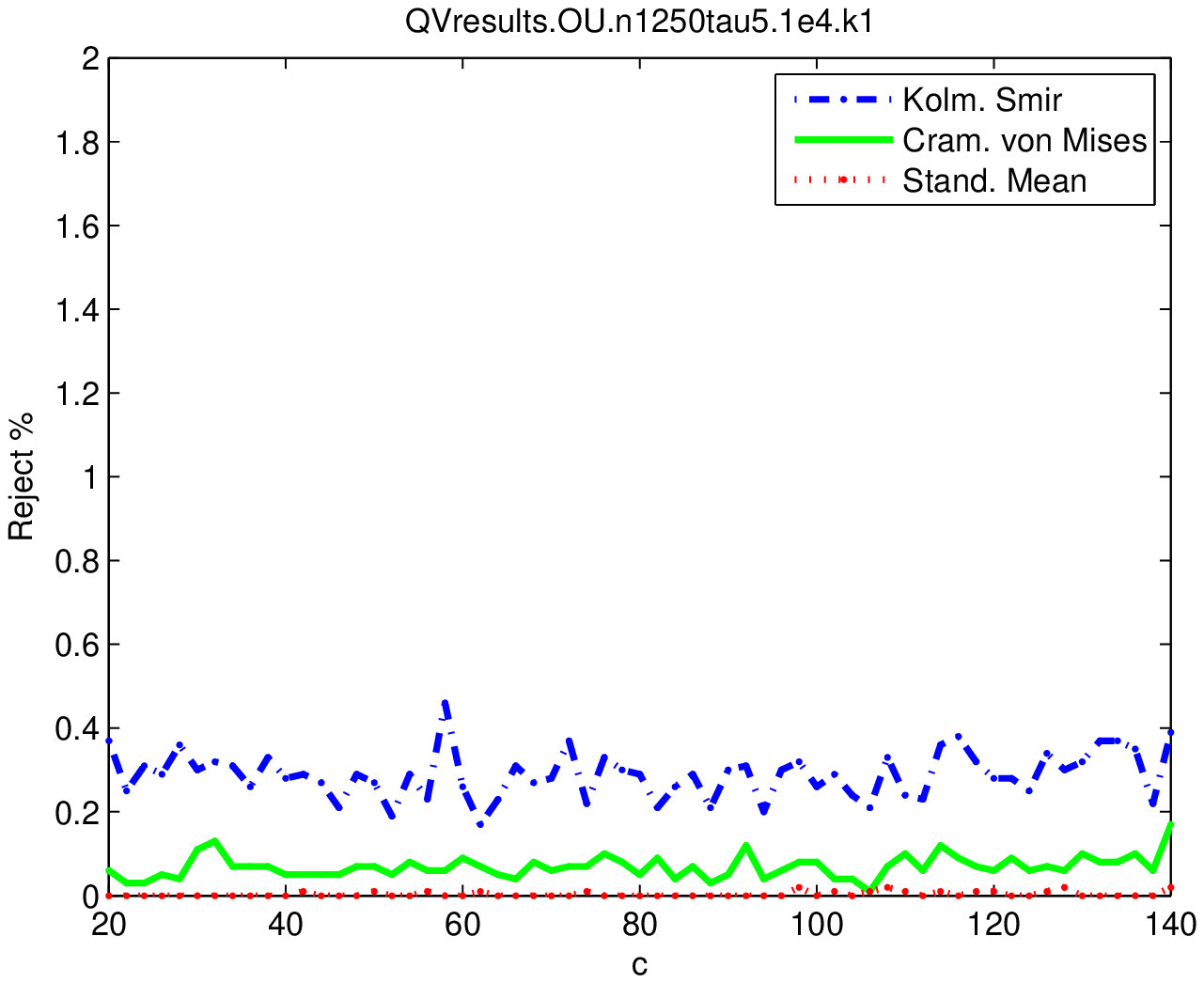}}
\resizebox{3in}{!}{\includegraphics{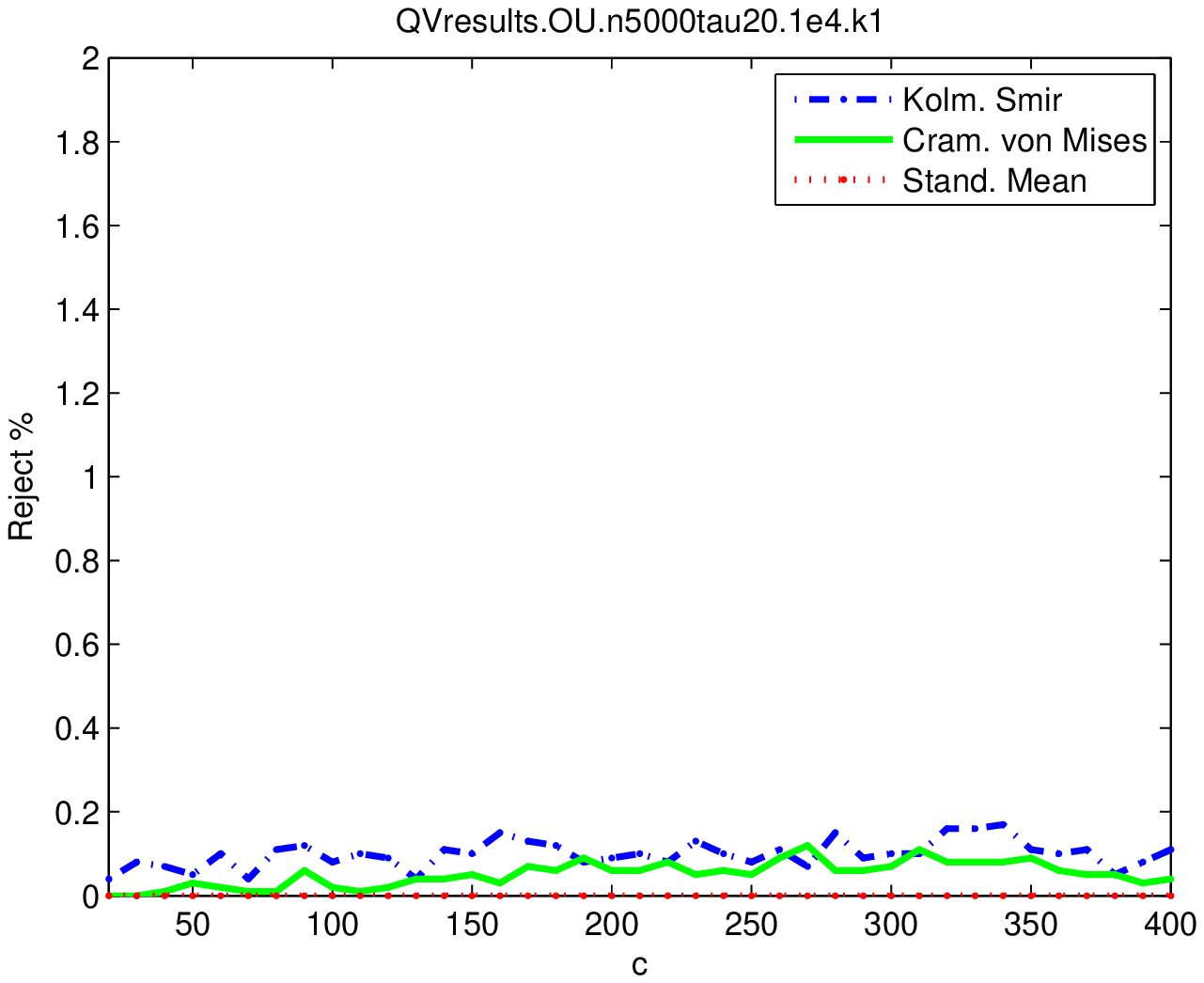}}
\caption{Rejection rates using sample quadratic variation for Ornstein-Uhlenbeck processes with drift coefficient $\al=1$. Left panel uses short ($n=1250$) datasets, right panel uses long ($n=5000$) datasets.  Here, 95\% confidence intervals are 1\% or smaller. \label{OUk1.fig}}
\end{figure}

\clearpage 
\subsection{Feller's Square-root Diffusion} 
Another example of a mean reverting process in use is Feller's square root diffusion \citep{Feller51}, which underlies the Cox-Ingersoll-Ross model \citep{CIR85a}. It is described by  \[dX(t) = \ka (\mu - X(t)) dt + \si \sqrt{X(t)} dW(t)\]
for $\ka, \mu, \si > 0$, $2\ka\mu/\si^2 \geq 1$. The stationary distribution is Gamma$(a,b)$ with shape parameter $a=2\ka\mu/\si^2$ and scale parameter $b=\sigma^2/(2\ka)$. We assume the stationary process throughout. We again considered two lengths (1250 and 5000) and two parameter values ($\ka=6$ and $\ka=8$), with $\si=1$ and $\mu=0.2$. Note that the mean is non-zero for this process.

To simulate 1250 crossings with $\ka=6$, the value $\de=0.028163$ was used for the size of the level 0 crossings. Figure \ref{Feller1.fig} (left) shows the results from our distribution tests. As with the Ornstein-Uhlenbeck process, our $\chi^2$ test seems to have the most power, rejecting almost 19\% at level 2. At level 3 it again rejects less than other distribution-based tests due to insufficient data length. The KS discrete and G test are comparable at level 2, with the former showing more power at level 3, although almost all the paths were tested using both tests. For this process, the KLP98 test seems to have little power.

Figure \ref{Feller1.fig} (right) shows the results from our independence tests. As with the Ornstein-Uhlenbeck process, the joint distribution test appears to have the most power, although its power does not dominate in these tests.  In fact, its power appears to exceed only the Twos test and the KLP98 test amongst our distribution tests. The other independence tests don't really reject beyond their significance level.

These results suggest 11-19\% of the sample paths are rejected. As with the Ornstein-Uhlenbeck process, this is much better than with the quadratic variation method. Figure \ref{Fellershort.fig} shows based on our implementation we would reject about 0.3\% with both the KS and CVM tests. Park and Vasudev report no results for length 1250 due to low power.

Figure \ref{Feller2.fig} shows results from the same process, but with 5000 crossings simulated using $\de=0.028474$. With the additional length, more paths are rejected. The left panel shows, again, our $\chi^2$ test rejects the most, 72\% at level 3. The KS discrete test is not far behind, rejecting 65\% at level3. From our independence tests, again the joint distribution test dominates, rejecting about 67\% at level 3. At level 2 its power is less than all but the KLP98 test.

The Dixon-O'Brien symmetry test on the updown pairs (Dix.-OBri. UD) is noteworthy here, in that it rejects about 15\% of the paths at level 2.  This test rejects a path if the more numerous symbol in the binary sequence (either 0 or 1, representing up-down and down-up pairs), shows clustering. This is consistent with a mean reverting process in that the first crossing of an up-down pair is more likely to be in the direction of the mean. And so, amongst only down-up and up-down pairs, clustering is consistent with a path that favors return to the mean.

In summary, for this process with these parameters our results show 62--67\% of the paths are rejected. Our implementation of the quadratic variation method (Figure \ref{Fellerlong.fig}) suggests that method rejects from 14\% (KS) to 24\% (CVM) of paths. So the crossing tree method appears more powerful, rejecting an additional 48\% of the paths. Park and Vasudev again reported smaller rejection rates: 4\% (KS) and 9\% (CVM).

For comparison we also used drift coefficient $\ka=8$. Figure \ref{Feller3.fig} shows our results for 1250 simulated crossings using $\de=0.028276$. The results are qualitatively similar to those for $\ka=6$ and the same length. This was also seen with the Ornstein-Uhlenbeck process when the drift coefficient was increased. A difference worth noting is that at level 3 the KS discrete test rejects about 10\% more paths than the $\chi^2$ tests. This is a reversal of usual results in which the $\chi^2$ test shows the highest power among the tests of distribution.  In summary, for this process the crossing tree would reject 15-18\% of the paths. Figure \ref{Fellershort.fig} (right) shows, with our quadratic variation implemenation, only 0.6-0.7\% of paths are rejected, continuing the results for $\ka=6$ showing increased power from the crossing tree.

\begin{figure}[t!]
\resizebox{3in}{!}{\includegraphics{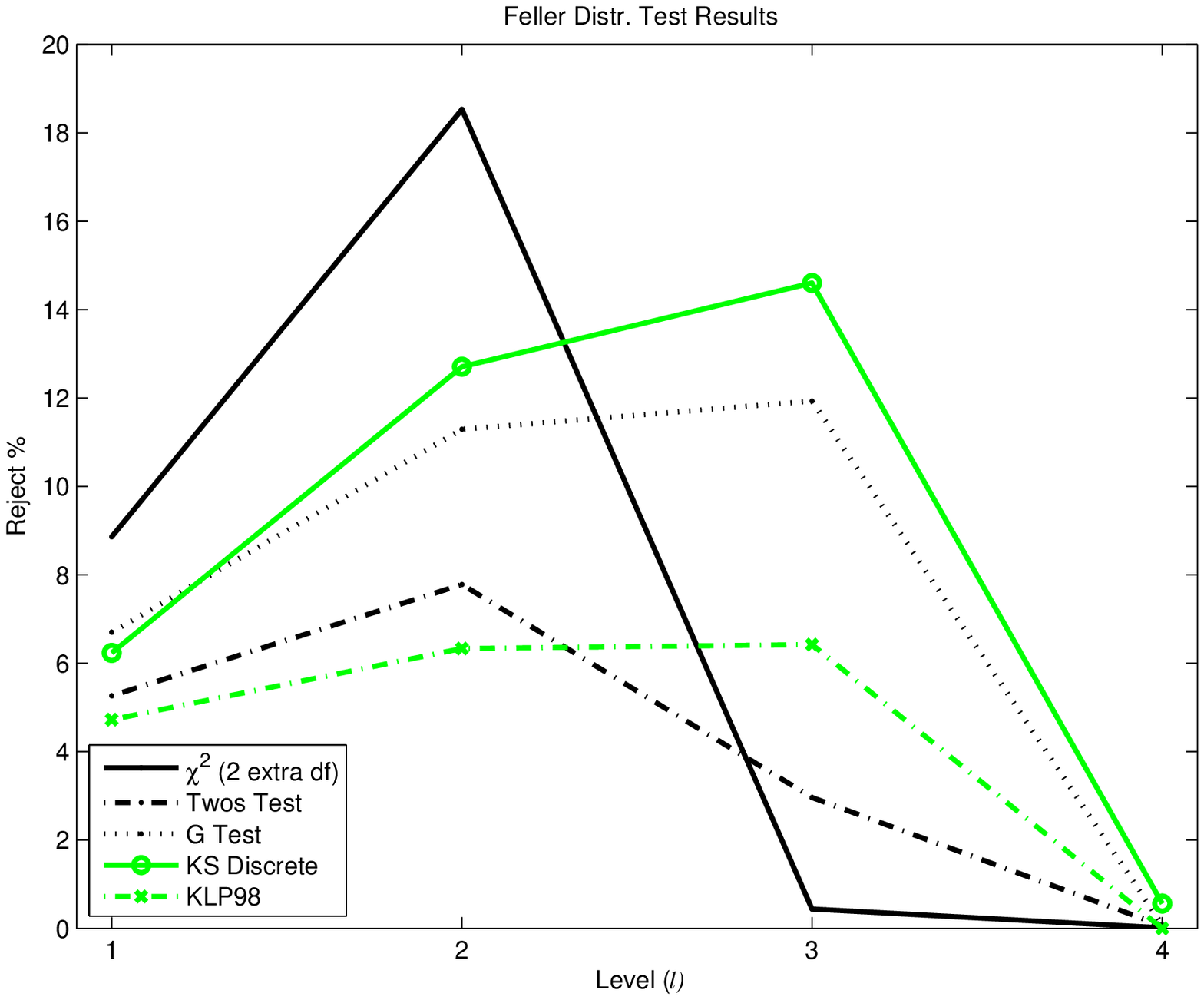}}\resizebox{3in}{!}{\includegraphics{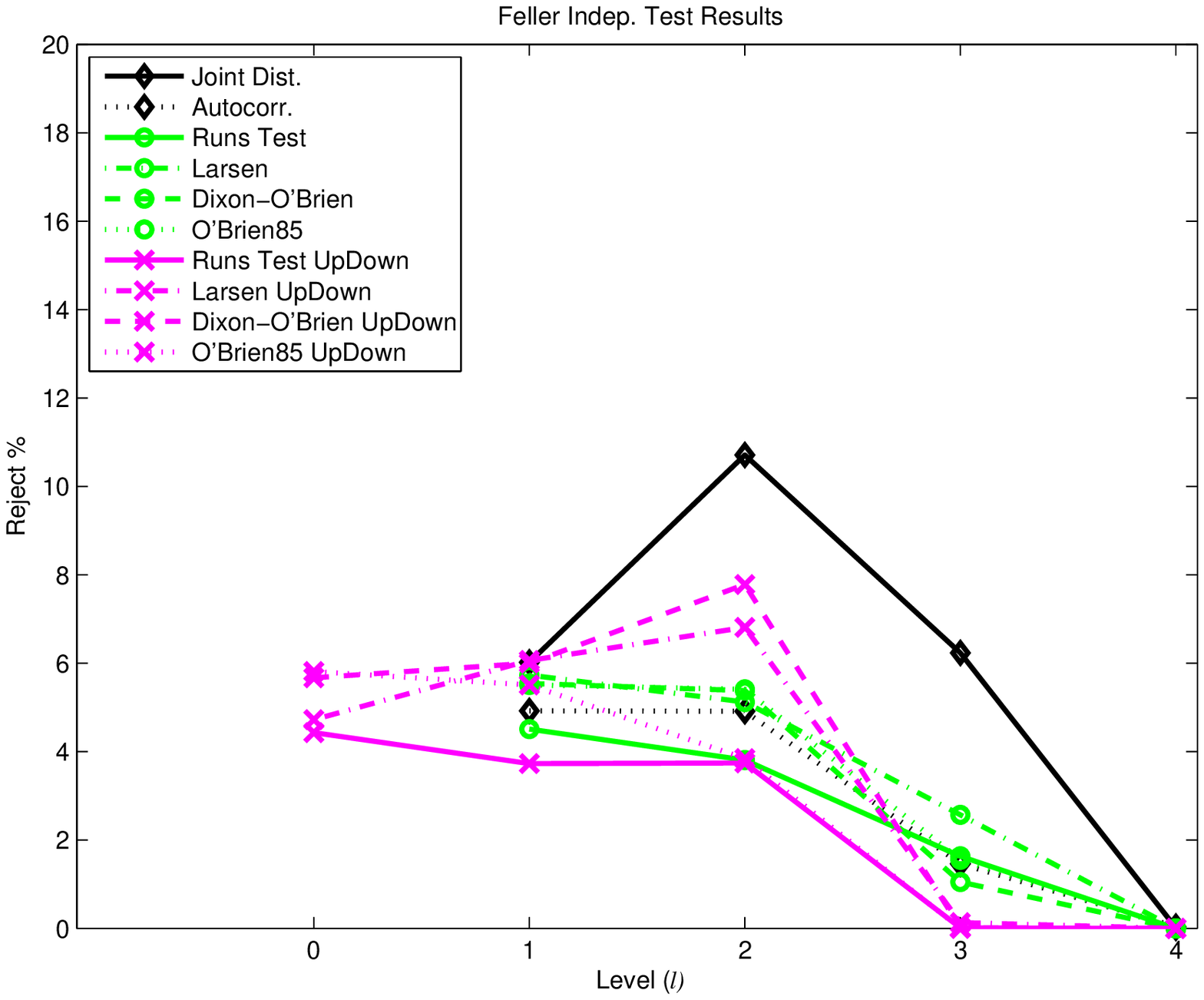}}
\caption{Crossing tree results for Feller's square root diffusion process with $\ka=6$, $\sigma=1$, $\mu=0.2$, $n = 1250$, $\de = 0.028163$. Percentage of 10,000 sample paths rejected using distribution tests (left) and independence tests (right). Here, 95\% confidence intervals are 1\% or smaller. \label{Feller1.fig}}
\end{figure}

\begin{table}[thb!]

 {\small 
 \begin{center} 
 \begin{tabular}{ |c || *{5}{c|}} 
\hline \multicolumn{6}{|c|}{$\ka=6$, $\mu=0.2$, $\sigma=1$, $n = 1250$, $\de = 0.028163$ }  \\ 
 \hline \hline 
& \multicolumn{5}{c|}{\raisebox{0ex}[12pt]{} {\bf \% of all} (\% of tested; \# tested)} \\ \hline {\bf levels }& 0 & 1 & 2 & 3 & 4 \\ \cline{2-6} 
 \hline \hline\raisebox{0ex}[12pt]{{\bf $\chi^2$ (+2 df)}}&  & {\bf   8.9} & {\bf  18.5} & {\bf   0.4} (  2.5;   1731) & {\bf   0.0} (100.0;      1)\\ \hline 
{\bf Twos Test} &  & {\bf   5.3} & {\bf   7.8} & {\bf   3.0} (  3.0;   9967) & {\bf   0.1} (  0.5;   1276)\\ \hline 
{\bf G Test} &  & {\bf   6.7} & {\bf  11.3} & {\bf  11.9} ( 12.0;   9967) & {\bf   0.1} (  0.6;   1276)\\ \hline 
{\bf KS Discrete} &  & {\bf   6.2} & {\bf  12.7} & {\bf  14.6} & {\bf   0.6}\\ \hline 
{\bf KLP98 Test} &  & {\bf   4.7} & {\bf   6.3} & {\bf   6.4} & {\bf   0.0}\\ \hline 
\hline 
{\bf Joint Dist.}&  & {\bf   6.0} & {\bf  10.7} & {\bf   6.2} ( 11.9;   5243) & {\bf   0.0} (100.0;      1)\\ \hline 
{\bf Autocorr.} &  & {\bf   4.9} & {\bf   4.9} & {\bf   1.5} (  1.6;   9134) & {\bf   0.0} (  7.7;     52)\\ \hline 
{\bf Runs Test} &  & {\bf   4.5} & {\bf   3.8} & {\bf   1.6} & {\bf   0.0}\\ \hline 
{\bf Larsen Test} &  & {\bf   5.7} & {\bf   5.1} & {\bf   2.6} (  2.7;   9509) & {\bf   0.0} (  0.0;   1003)\\ \hline 
{\bf  Dix.-OBri.} &  & {\bf   5.5} & {\bf   5.4} & {\bf   1.1} (  1.1;   9967) & {\bf   0.0} (  0.0;   1276)\\ \hline 
{\bf OBri85} &  & {\bf   5.5} & {\bf   5.4} (  5.4;   9994) & {\bf   1.6} (  6.3;   2467) & {\bf   0.0} (  NaN;      0)\\ \hline \hline 
{\bf Runs UD} & {\bf   4.4}  & {\bf   3.7} & {\bf   3.7} & {\bf   0.0} & {\bf   0.0}\\ \hline 
{\bf Larsen UD} & {\bf   4.7}  & {\bf   6.1} & {\bf   6.8} (  6.8;   9999) & {\bf   0.1} (  0.1;   9056) & {\bf   0.0} (  0.0;    874)\\ \hline 
{\bf  Dix.-OBri. UD} & {\bf   5.7} & {\bf   6.0} & {\bf   7.8} & {\bf   0.0} (  0.0;   9840) & {\bf   0.0} (  0.0;   1082)\\ \hline 
{\bf OBri85 UD} & {\bf   5.8}  & {\bf   5.5} & {\bf   3.8} (  5.6;   6826) & {\bf   0.0} (  0.0;      5) & {\bf   0.0} (  NaN;      0)\\ \hline 
\end{tabular} 
 \end{center}
 } 
\caption{Crossing tree results for Feller's square root diffusion process with $\ka=6$, $\sigma=1$, $\mu=0.2$, $n = 1250$, $\de = 0.028163$. Percentage of 10,000 sample paths rejected, with an average of 1250 crossings in time 5. At higher levels, insufficient data length means some datasets are not tested.\label{Feller1.table}}
\end{table}

Finally, with drift coefficient $\ka=8$ we also simulated 5000 crossings using $\de=0.028330$.  Figure \ref{Feller4.fig} shows our results, which  are similar to those for $\ka=6$ with length 5000. The power of the joint distribution tests again dominates. The Dixon-O'Brien symmetry test (Dix.-OBri. UD) is again noteworthy at level 2 for rejecting about 16\% of the paths, showing evidence of clustering in the excursions, consistent with reversion to the mean. In summary, these results suggest the crossing tree rejects about 77--81\%  of sample paths. This is almost 30\% better than the quadratic variation method (Figure \ref{Fellerlong.fig}) which rejects about 30\% (KS) and 50\% (CVM). Again, Park and Vasudev report lower rates: 14\% (KS) and 26\% (CVM).

\begin{figure}[t!]
\resizebox{3in}{!}{\includegraphics{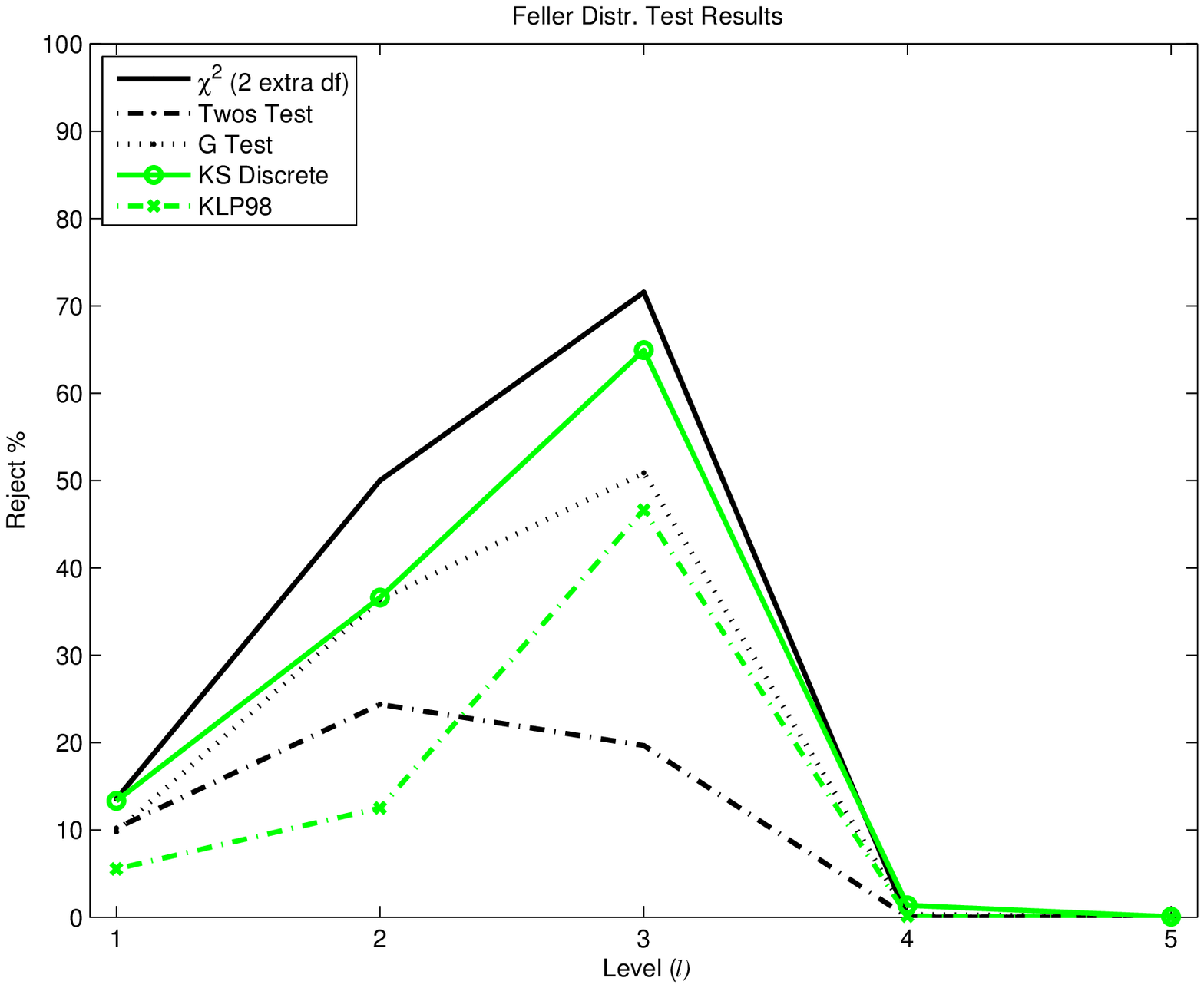}}\resizebox{3in}{!}{\includegraphics{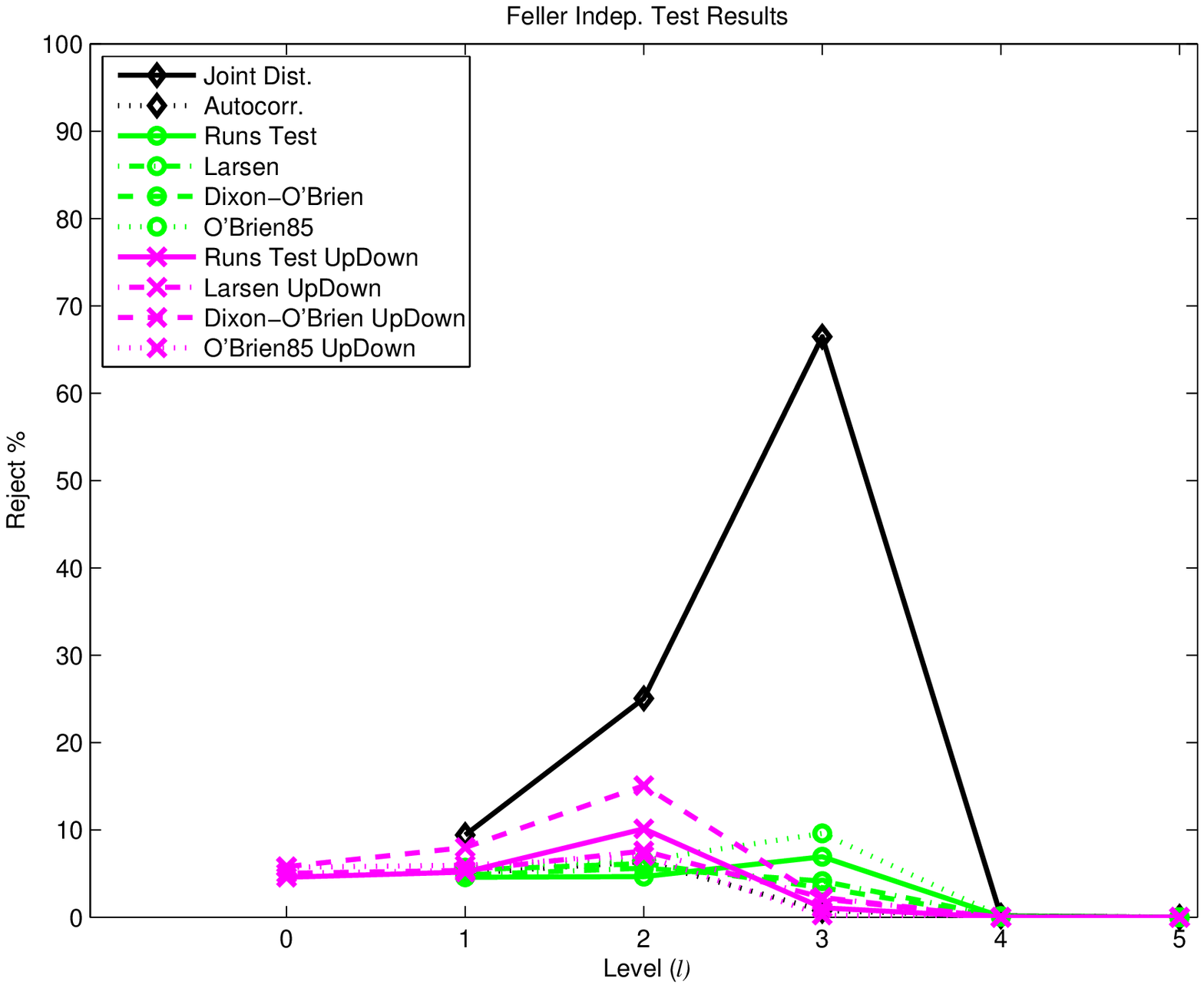}}
\caption{Crossing tree results for Feller's square root diffusion process with $\ka=6$, $\sigma=1$, $\mu=0.2$, $n = 5000$, $\de = 0.028474$. Percentage of 10,000 sample paths rejected using distribution tests (left) and independence tests (right). Here, 95\% confidence intervals are 1\% or smaller.\label{Feller2.fig}}
\end{figure}

\begin{table}[thb!]

 {\small 
 \begin{center} 
 \begin{tabular}{ |c || *{6}{c|}} 
\hline \multicolumn{7}{|c|}{$\ka=6$, $\mu=0.2$, $\sigma=1$, $n = 5000$, $\de = 0.028474$ }  \\ 
 \hline \hline 
& \multicolumn{6}{c|}{\raisebox{0ex}[12pt]{} {\bf \% of all} (\% of tested; \# tested)} \\ \hline {\bf levels }& 0 & 1 & 2 & 3 & 4 & 5 \\ \cline{2-7} 
 \hline \hline\raisebox{0ex}[12pt]{{\bf $\chi^2$ (+2 df)}}&  & {\bf  13.5} & {\bf  50.0} & {\bf  71.6} & {\bf   0.0} (  1.5;    135) & {\bf   0.0} (  NaN;      0)\\ \hline 
{\bf Twos Test} &  & {\bf  10.2} & {\bf  24.4} & {\bf  19.7} & {\bf   0.1} (  0.2;   3569) & {\bf   0.0} (  0.0;     98)\\ \hline 
{\bf G Test} &  & {\bf   9.8} & {\bf  36.4} & {\bf  50.9} & {\bf   0.3} (  0.9;   3569) & {\bf   0.0} (  1.0;     98)\\ \hline 
{\bf KS Discrete} &  & {\bf  13.3} & {\bf  36.6} & {\bf  64.9} & {\bf   1.4} & {\bf   0.1}\\ \hline 
{\bf KLP98 Test} &  & {\bf   5.5} & {\bf  12.5} & {\bf  46.6} & {\bf   0.1} & {\bf   0.0}\\ \hline 
\hline 
{\bf Joint Dist.}&  & {\bf   9.4} & {\bf  25.1} & {\bf  66.5} & {\bf   0.1} (  5.8;    258) & {\bf   0.0} (  NaN;      0)\\ \hline 
{\bf Autocorr.} &  & {\bf   5.0} & {\bf   6.4} & {\bf   0.8} & {\bf   0.1} (  3.7;    377) & {\bf   0.0} ( 33.3;      3)\\ \hline 
{\bf Runs Test} &  & {\bf   4.6} & {\bf   4.7} & {\bf   6.9} & {\bf   0.1} & {\bf   0.0}\\ \hline 
{\bf Larsen Test} &  & {\bf   4.8} & {\bf   5.6} & {\bf   4.1} & {\bf   0.1} (  0.3;   3015) & {\bf   0.0} (  0.0;     88)\\ \hline 
{\bf  Dix.-OBri.} &  & {\bf   5.4} & {\bf   6.1} & {\bf   3.4} & {\bf   0.1} (  0.2;   3569) & {\bf   0.0} (  0.0;     98)\\ \hline 
{\bf OBri85} &  & {\bf   5.7} & {\bf   6.5} & {\bf   9.6} ( 10.7;   8982) & {\bf   0.1} (  7.4;    162) & {\bf   0.0} (  NaN;      0)\\ \hline \hline 
{\bf Runs UD} & {\bf   4.6}  & {\bf   5.2} & {\bf  10.1} & {\bf   1.1} & {\bf   0.0} & {\bf   0.0}\\ \hline 
{\bf Larsen UD} & {\bf   5.0}  & {\bf   5.4} & {\bf   7.6} & {\bf   2.3} (  2.3;   9940) & {\bf   0.0} (  0.0;   3280) & {\bf   0.0} (  0.0;     58)\\ \hline 
{\bf  Dix.-OBri. UD} & {\bf   5.8} & {\bf   8.0} & {\bf  15.1} & {\bf   2.1} & {\bf   0.0} (  0.0;   3456) & {\bf   0.0} (  0.0;     65)\\ \hline 
{\bf OBri85 UD} & {\bf   5.7}  & {\bf   5.9} & {\bf   7.0} (  7.0;   9977) & {\bf   0.2} (  5.8;    412) & {\bf   0.0} (  NaN;      0) & {\bf   0.0} (  NaN;      0)\\ \hline 
\end{tabular} 
 \end{center}
 } 
\caption{Crossing tree results for Feller's square root diffusion process with $\ka=6$, $\sigma=1$, $\mu=0.2$, $n = 5000$, $\de = 0.028474$. Percentage of 10,000 sample paths rejected, with an average of 5000 crossings in time 20. At higher levels, insufficient data length means some datasets are not tested.} \label{Feller2.table}
\end{table}

It is important to note here that for all these processes we actually did the tests three times, building the crossing tree relative to a value $\de_0$, where $\de_0$ is one of 0, $X(1)$ (the first data point), or the mean of the first 30 crossings built about a temporary crossing tree based at $0$. It is results for the last choice we report since that has shown the highest power. For example, with the Feller processes with $\ka=8$, $n=5000$ the other two choices were comparable, rejecting about 44\% ($\chi^2$), 31\% (KS discrete) and 44\% (joint distribution) which is much less than 66\%, 77\% and 81\%, respectively. We attribute this substantial difference to the non-zero mean of the process. Again, these processes don't satisfy the null hypothesis, so making the choice that gives the highest power seems most appropriate. The effect on the type 1 error appears negligible between these three choices. For the other processes we simulated, this choice always had superior power to the other two, although the difference was not always as large.

\begin{figure}[htb!]
\resizebox{3in}{!}{\includegraphics{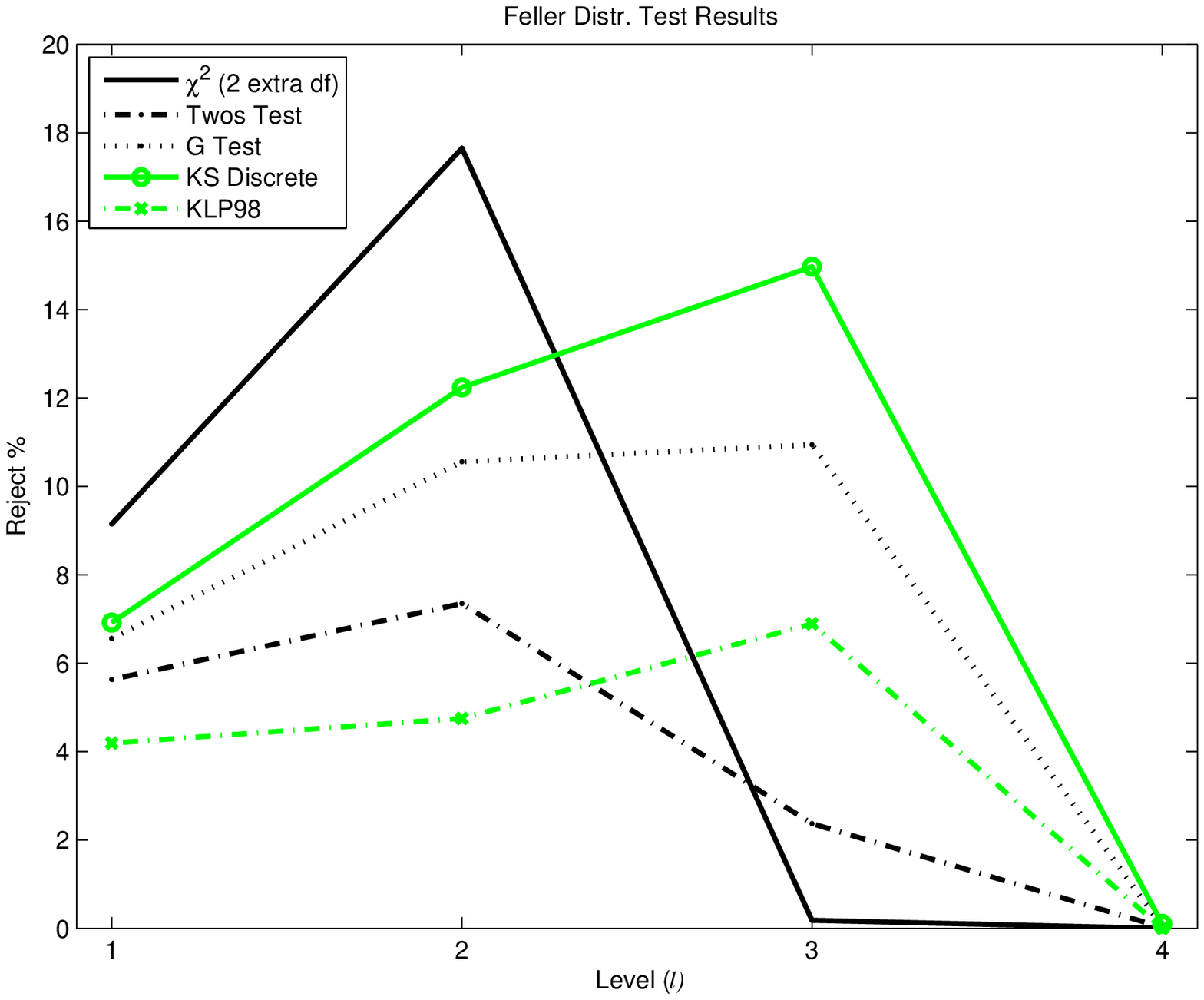}}\resizebox{3in}{!}{\includegraphics{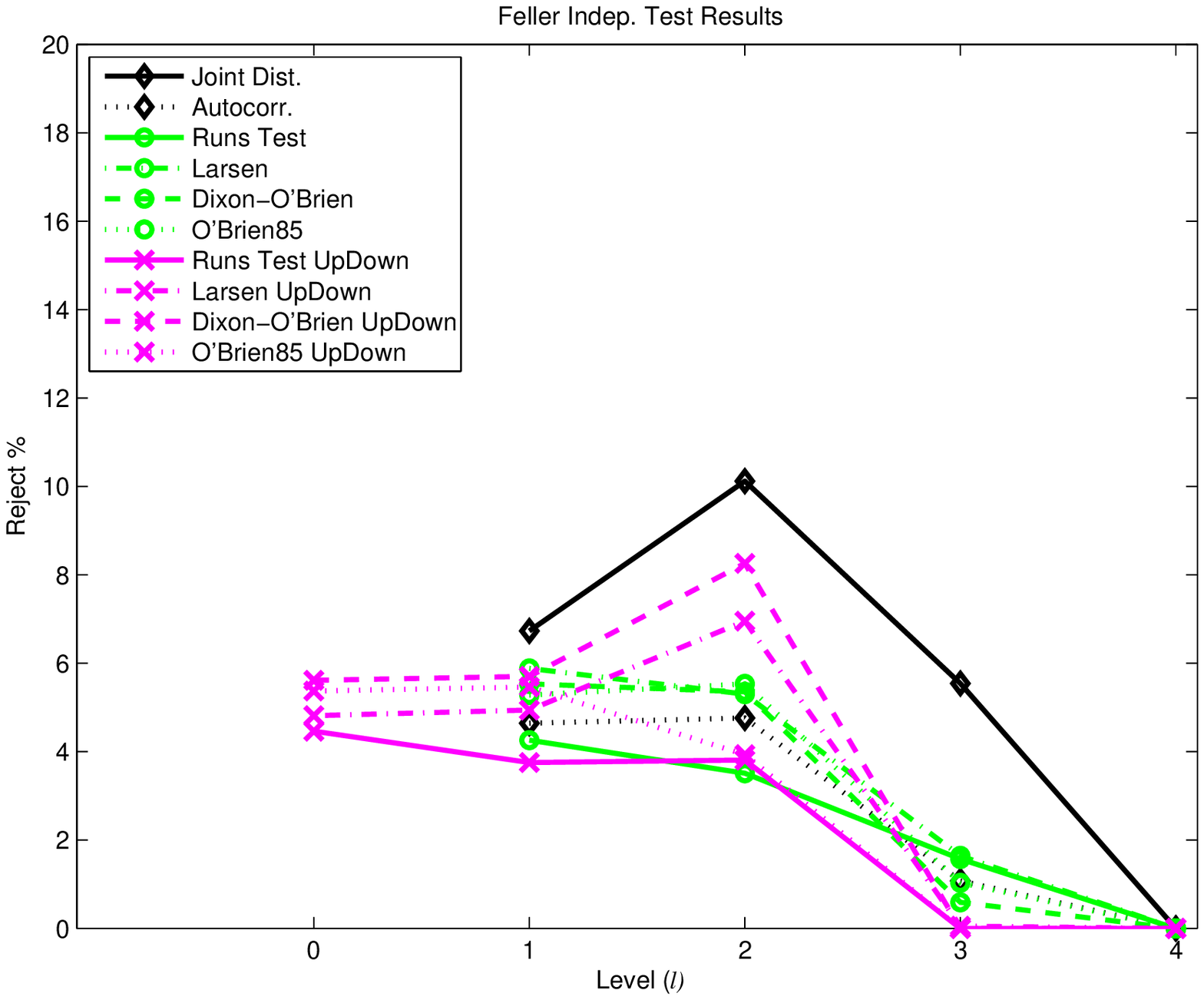}}
\caption{Crossing tree results for Feller's square root diffusion process with $\ka=8$, $\sigma=1$, $\mu=0.2$, $n = 1250$, $\de = 0.028276$. Percentage of 10,000 sample paths rejected using distribution tests (left) and independence tests (right). Here, 95\% confidence intervals are 1\% or smaller.\label{Feller3.fig}}
\end{figure}

\begin{table}[thb!]

 {\small 
 \begin{center} 
 \begin{tabular}{ |c || *{5}{c|}} 
\hline \multicolumn{6}{|c|}{$\ka=8$, $\mu=0.2$, $\sigma=1$, $n = 1250$, $\de = 0.028276$ }  \\ 
 \hline \hline 
& \multicolumn{5}{c|}{\raisebox{0ex}[12pt]{} {\bf \% of all} (\% of tested; \# tested)} \\ \hline {\bf levels }& 0 & 1 & 2 & 3 & 4 \\ \cline{2-6} 
 \hline \hline\raisebox{0ex}[12pt]{{\bf $\chi^2$ (+2 df)}}&  & {\bf   9.2} & {\bf  17.6} & {\bf   0.2} (  2.1;    849) & {\bf   0.0} (  NaN;      0)\\ \hline 
{\bf Twos Test} &  & {\bf   5.6} & {\bf   7.3} & {\bf   2.4} (  2.4;   9935) & {\bf   0.0} (  0.5;    429)\\ \hline 
{\bf G Test} &  & {\bf   6.6} & {\bf  10.6} & {\bf  10.9} ( 11.0;   9935) & {\bf   0.0} (  0.9;    429)\\ \hline 
{\bf KS Discrete} &  & {\bf   6.9} & {\bf  12.2} & {\bf  15.0} & {\bf   0.1}\\ \hline 
{\bf KLP98 Test} &  & {\bf   4.2} & {\bf   4.8} & {\bf   6.9} & {\bf   0.0}\\ \hline 
\hline 
{\bf Joint Dist.}&  & {\bf   6.7} & {\bf  10.1} & {\bf   5.5} ( 14.3;   3871) & {\bf   0.0} (  NaN;      0)\\ \hline 
{\bf Autocorr.} &  & {\bf   4.6} & {\bf   4.8} & {\bf   1.1} (  1.3;   8523) & {\bf   0.0} (  0.0;      9)\\ \hline 
{\bf Runs Test} &  & {\bf   4.3} & {\bf   3.5} & {\bf   1.6} & {\bf   0.0}\\ \hline 
{\bf Larsen Test} &  & {\bf   5.9} & {\bf   5.3} & {\bf   1.6} (  1.8;   9365) & {\bf   0.0} (  0.0;    350)\\ \hline 
{\bf  Dix.-OBri.} &  & {\bf   5.5} & {\bf   5.3} & {\bf   0.6} (  0.6;   9935) & {\bf   0.0} (  0.0;    429)\\ \hline 
{\bf OBri85} &  & {\bf   5.3} & {\bf   5.5} (  5.5;   9993) & {\bf   1.0} (  6.4;   1600) & {\bf   0.0} (  NaN;      0)\\ \hline \hline 
{\bf Runs UD} & {\bf   4.5}  & {\bf   3.8} & {\bf   3.8} & {\bf   0.0} & {\bf   0.0}\\ \hline 
{\bf Larsen UD} & {\bf   4.8}  & {\bf   4.9} & {\bf   7.0} & {\bf   0.1} (  0.1;   8905) & {\bf   0.0} (  0.0;    261)\\ \hline 
{\bf  Dix.-OBri. UD} & {\bf   5.6} & {\bf   5.7} & {\bf   8.3} & {\bf   0.0} (  0.0;   9807) & {\bf   0.0} (  0.0;    363)\\ \hline 
{\bf OBri85 UD} & {\bf   5.4}  & {\bf   5.5} & {\bf   3.9} (  5.7;   6927) & {\bf   0.0} (  NaN;      0) & {\bf   0.0} (  NaN;      0)\\ \hline 
\end{tabular} 
 \end{center}
 } 
\caption{Crossing tree results for Feller's square root diffusion process with $\ka=8$, $\sigma=1$, $\mu=0.2$, $n = 1250$, $\de = 0.028276$. Percentage of 10,000 sample paths rejected, with an average of 1250 crossings in time 5. At higher levels, insufficient data length means some datasets are not tested. } \label{Feller3.table}
\end{table}

It is also worth noting again that where we see significant rejections by the joint distribution test, they are due to primarily to dependence, not distribution. For the longer ($n=5000$) datasets and  $\ka=6$, at level 3 about 32\% of all paths are rejected. After permuting the data, about 14\% of paths are rejected. So, about 18\% of all paths are rejected due to bivariate dependence. With $\ka=8$ the difference is even more dramatic. At level 3 about 43\% of paths are rejected, which falls to about 16\% after permuting the data. Thus, 27\% of the paths are rejected due bivariate dependence. That's about two-thirds of all the paths rejected by that test. With short ($n=1250$) datasets, for both parameter values the differences between using permuted and non-permuted data are small, with both tests rejecting around the level of the test. For example, with $\kappa=6$ the number of rejections was comparable and in the range 4--8\% of all paths at levels 1--3.

\begin{figure}[htb!]
\resizebox{3in}{!}{\includegraphics{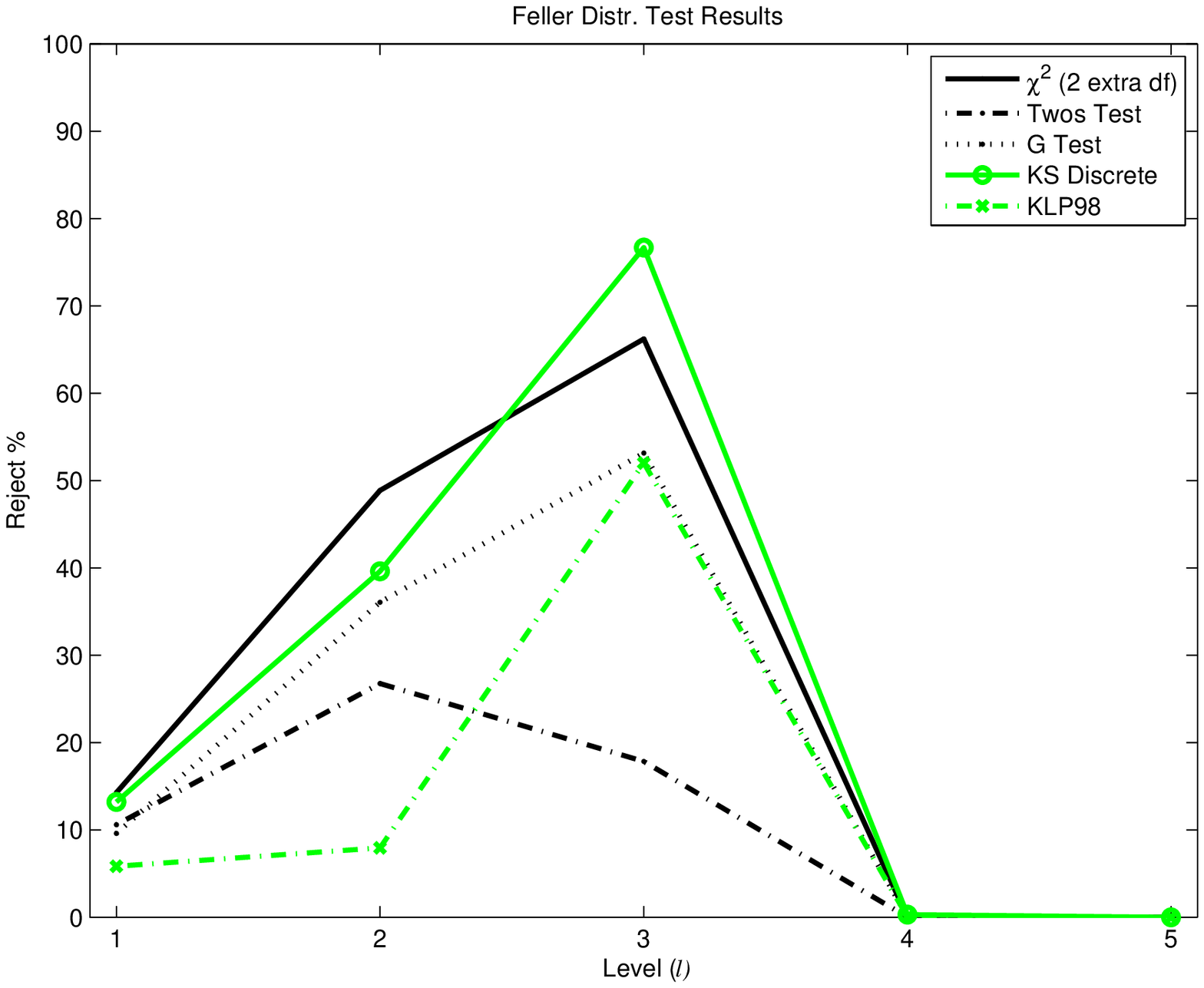}}\resizebox{3in}{!}{\includegraphics{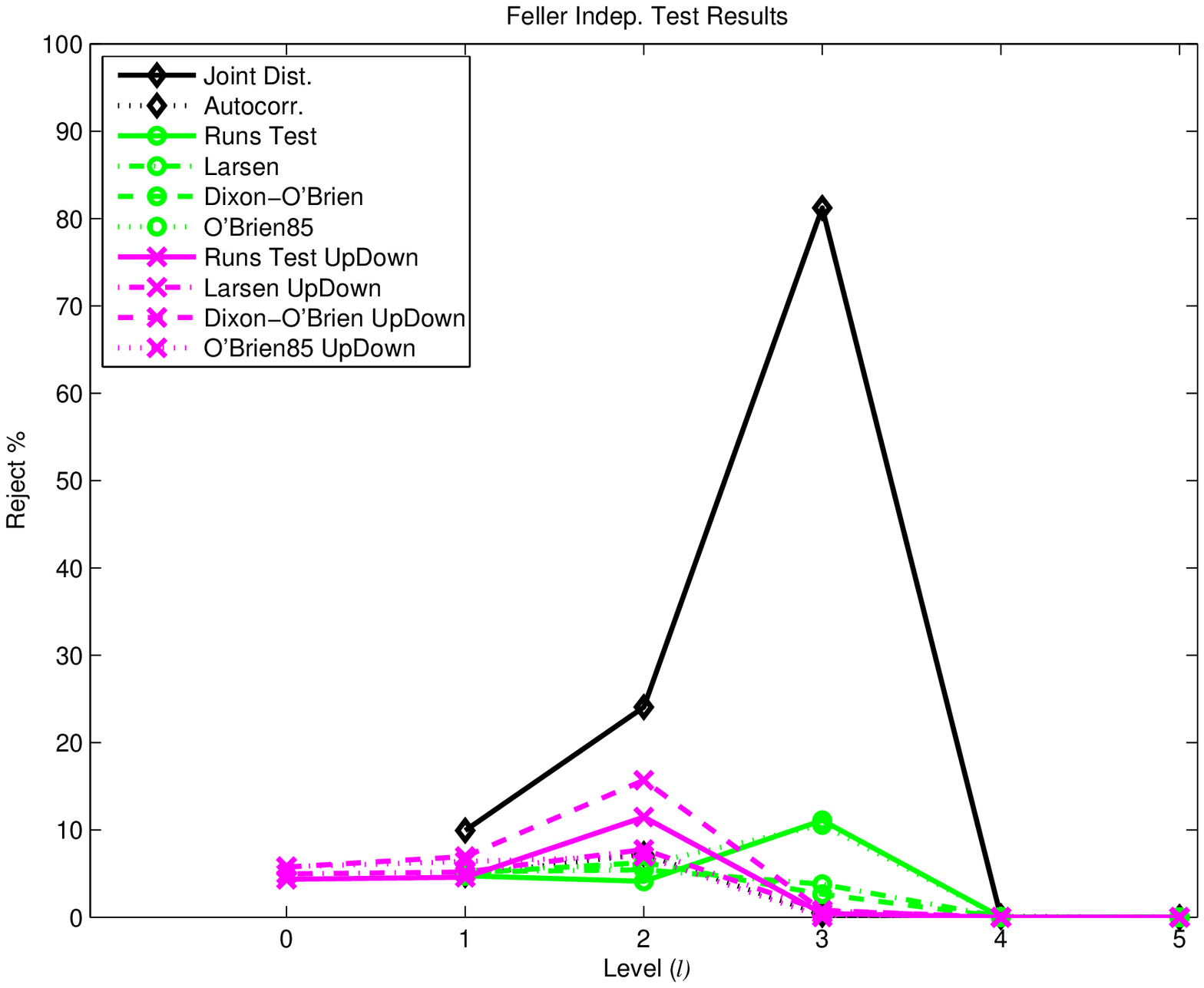}}
\caption{Crossing tree results for Feller's square root diffusion process with $\ka=8$, $\sigma=1$, $\mu=0.2$, $n = 5000$, $\de = 0.028330$. Percentage of 10,000 sample paths rejected using distribution tests (left) and independence tests (right). Here, 95\% confidence intervals are 1\% or smaller.\label{Feller4.fig}}
\end{figure}

\begin{table}[thb!]

 {\small 
 \begin{center} 
 \begin{tabular}{ |c || *{6}{c|}} 
\hline \multicolumn{7}{|c|}{$\ka=8$, $\mu=0.2$, $\sigma=1$, $n = 5000$, $\de = 0.028330$ }  \\ 
 \hline \hline 
& \multicolumn{6}{c|}{\raisebox{0ex}[12pt]{} {\bf \% of all} (\% of tested; \# tested)} \\ \hline {\bf levels }& 0 & 1 & 2 & 3 & 4 & 5 \\ \cline{2-7} 
 \hline \hline\raisebox{0ex}[12pt]{{\bf $\chi^2$ (+2 df)}}&  & {\bf  14.3} & {\bf  48.9} & {\bf  66.2} ( 66.2;   9994) & {\bf   0.0} (  0.0;     33) & {\bf   0.0} (  NaN;      0)\\ \hline 
{\bf Twos Test} &  & {\bf  10.6} & {\bf  26.8} & {\bf  17.8} & {\bf   0.0} (  0.1;   1319) & {\bf   0.0} (  0.0;     18)\\ \hline 
{\bf G Test} &  & {\bf   9.6} & {\bf  36.0} & {\bf  53.1} & {\bf   0.0} (  0.3;   1319) & {\bf   0.0} (  0.0;     18)\\ \hline 
{\bf KS Discrete} &  & {\bf  13.2} & {\bf  39.6} & {\bf  76.7} & {\bf   0.3} & {\bf   0.0}\\ \hline 
{\bf KLP98 Test} &  & {\bf   5.8} & {\bf   8.0} & {\bf  52.0} & {\bf   0.0} & {\bf   0.0}\\ \hline 
\hline 
{\bf Joint Dist.}&  & {\bf   9.9} & {\bf  24.1} & {\bf  81.2} & {\bf   0.0} (  2.8;    108) & {\bf   0.0} (  NaN;      0)\\ \hline 
{\bf Autocorr.} &  & {\bf   4.8} & {\bf   7.2} & {\bf   0.3} & {\bf   0.1} (  5.0;    160) & {\bf   0.0} (  NaN;      0)\\ \hline 
{\bf Runs Test} &  & {\bf   4.7} & {\bf   4.1} & {\bf  11.1} & {\bf   0.0} & {\bf   0.0}\\ \hline 
{\bf Larsen Test} &  & {\bf   5.2} & {\bf   5.4} & {\bf   3.8} & {\bf   0.0} (  0.3;   1193) & {\bf   0.0} (  0.0;     18)\\ \hline 
{\bf  Dix.-OBri.} &  & {\bf   5.1} & {\bf   6.3} & {\bf   2.6} & {\bf   0.0} (  0.2;   1319) & {\bf   0.0} (  0.0;     18)\\ \hline 
{\bf OBri85} &  & {\bf   5.3} & {\bf   6.0} & {\bf  10.6} ( 13.0;   8162) & {\bf   0.0} (  7.7;     52) & {\bf   0.0} (  NaN;      0)\\ \hline \hline 
{\bf Runs UD} & {\bf   4.3}  & {\bf   4.6} & {\bf  11.5} & {\bf   0.4} & {\bf   0.0} & {\bf   0.0}\\ \hline 
{\bf Larsen UD} & {\bf   5.0}  & {\bf   5.2} & {\bf   7.8} & {\bf   0.8} (  0.8;   9766) & {\bf   0.0} (  0.0;   1156) & {\bf   0.0} (  0.0;     12)\\ \hline 
{\bf  Dix.-OBri. UD} & {\bf   5.7} & {\bf   7.0} & {\bf  15.7} & {\bf   0.8} & {\bf   0.0} (  0.0;   1286) & {\bf   0.0} (  0.0;     14)\\ \hline 
{\bf OBri85 UD} & {\bf   5.8}  & {\bf   6.4} & {\bf   7.0} (  7.0;   9981) & {\bf   0.0} (  2.8;    109) & {\bf   0.0} (  NaN;      0) & {\bf   0.0} (  NaN;      0)\\ \hline 
\end{tabular} 
 \end{center}
 } 
\caption{Crossing tree results for Feller's square root diffusion process with $\ka=8$, $\sigma=1$, $\mu=0.2$, $n = 5000$, $\de = 0.028330$. Percentage of 10,000 sample paths rejected, with an average of 5000 crossings in time 20. At higher levels, insufficient data length means some datasets are not tested. Here, 95\% confidence intervals are 1\% or smaller.} 
\label{Feller4.table}
\end{table}

\clearpage

\begin{figure}[hb!]
\resizebox{3in}{!}{\includegraphics{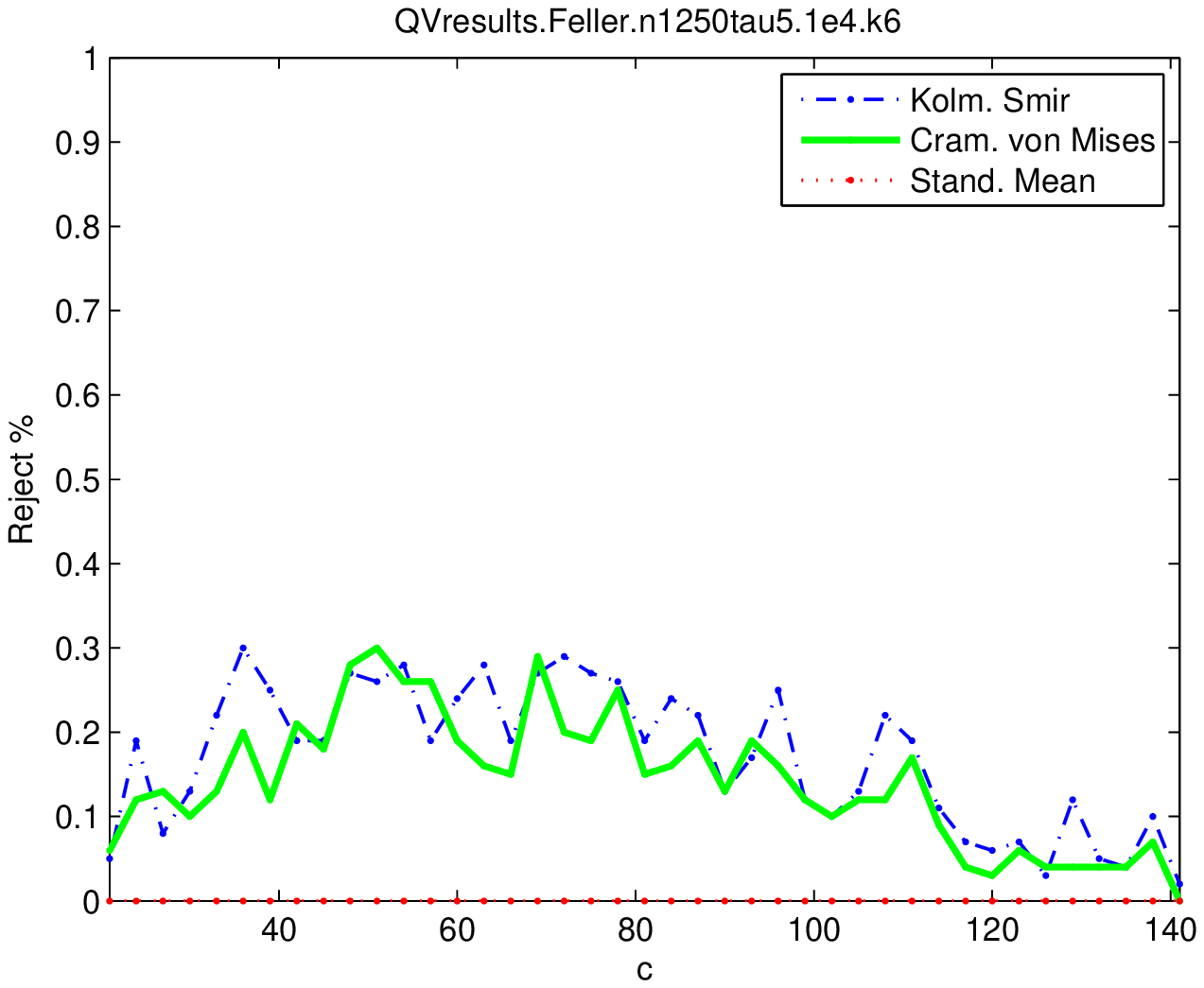}}\resizebox{3in}{!}{\includegraphics{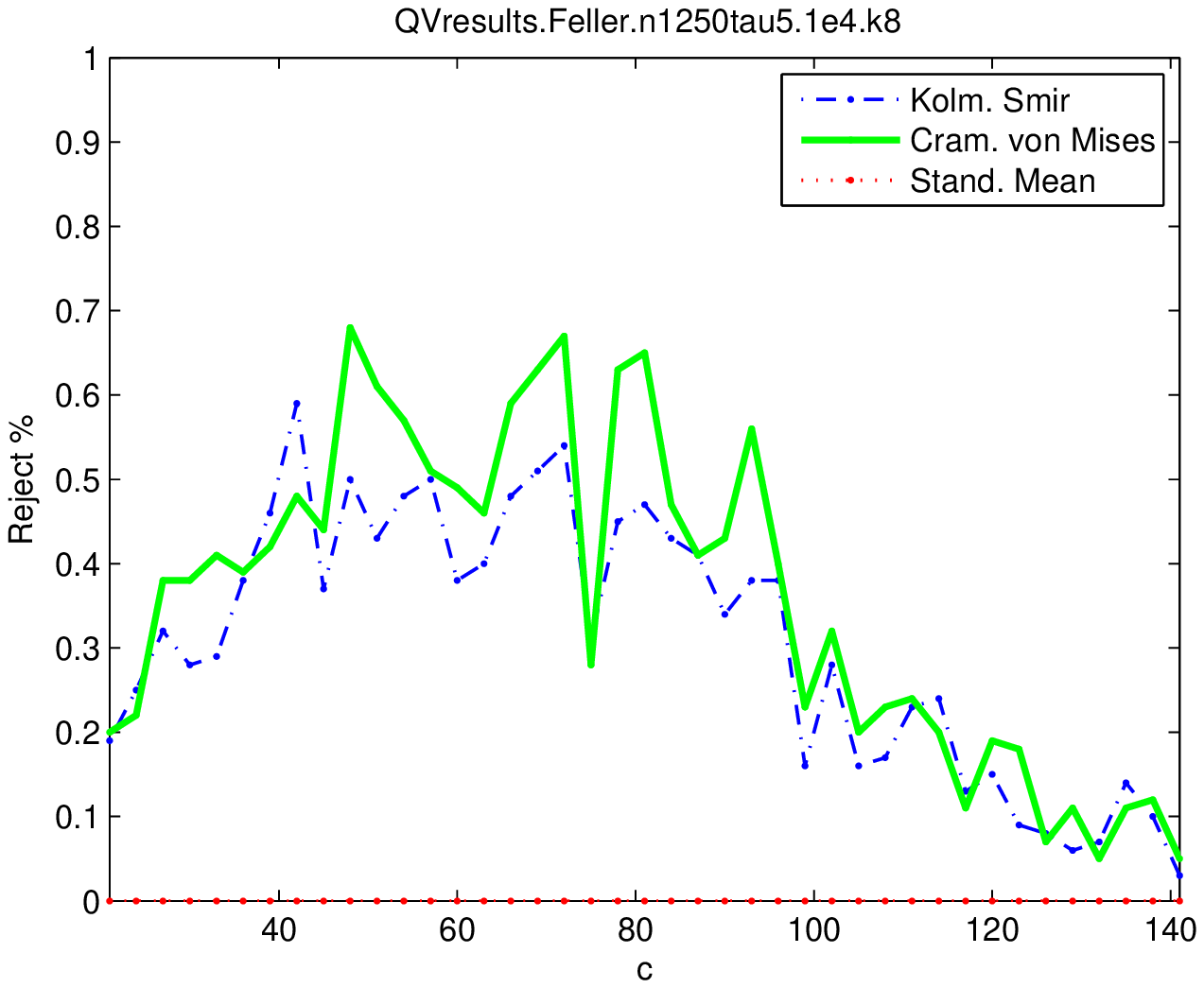}}
\caption{Rejection rates using sample quadratic variation for Feller's square-root diffusion. Short ($n=1250$) datasets. Left has drift coefficient $\ka=6$, right has drift coefficient $\ka=8$. Here, 95\% confidence intervals are 1\% or smaller.}
\label{Fellershort.fig}
\end{figure}

\begin{figure}[hb!]
\resizebox{3in}{!}{\includegraphics{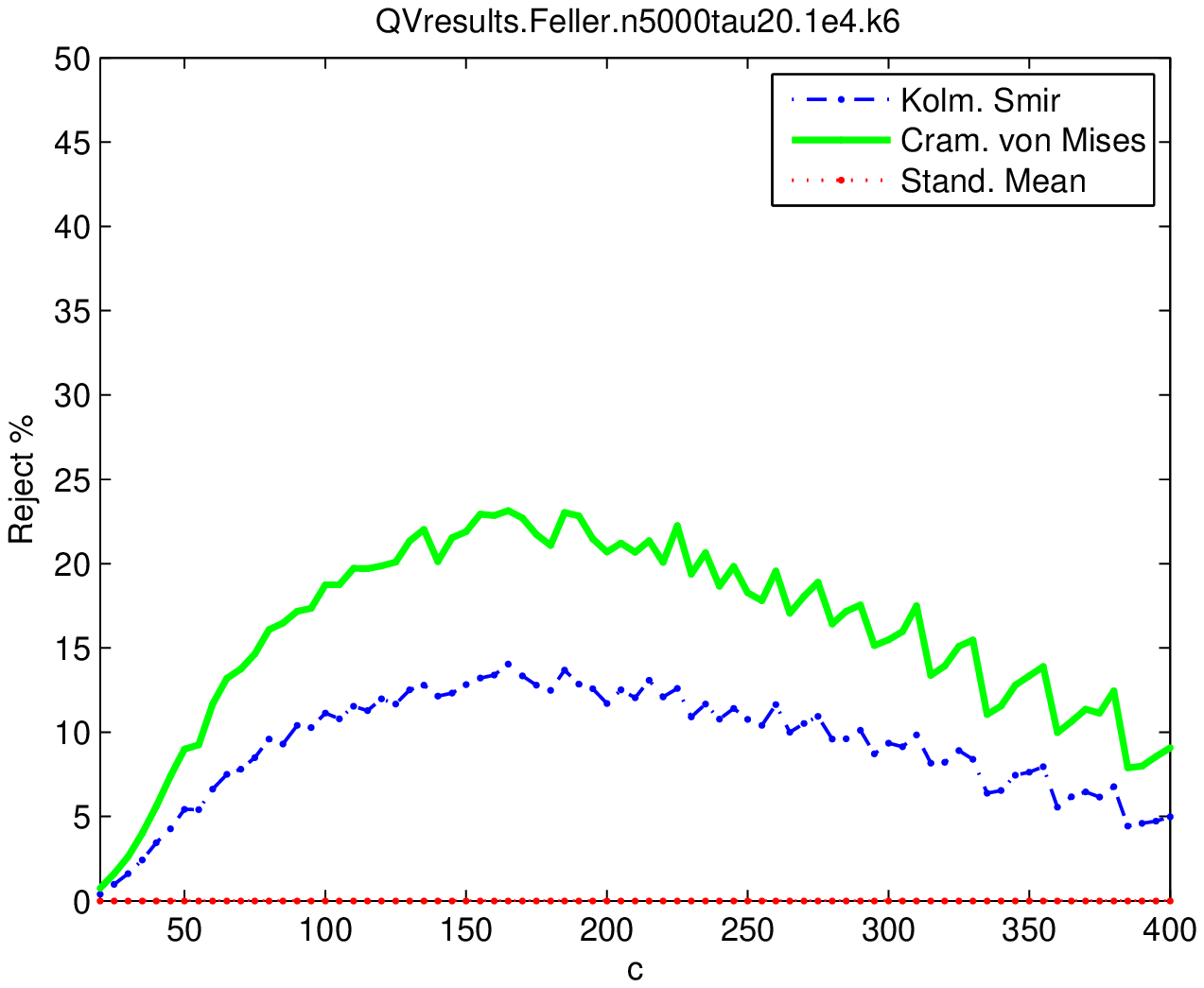}}\resizebox{3in}{!}{\includegraphics{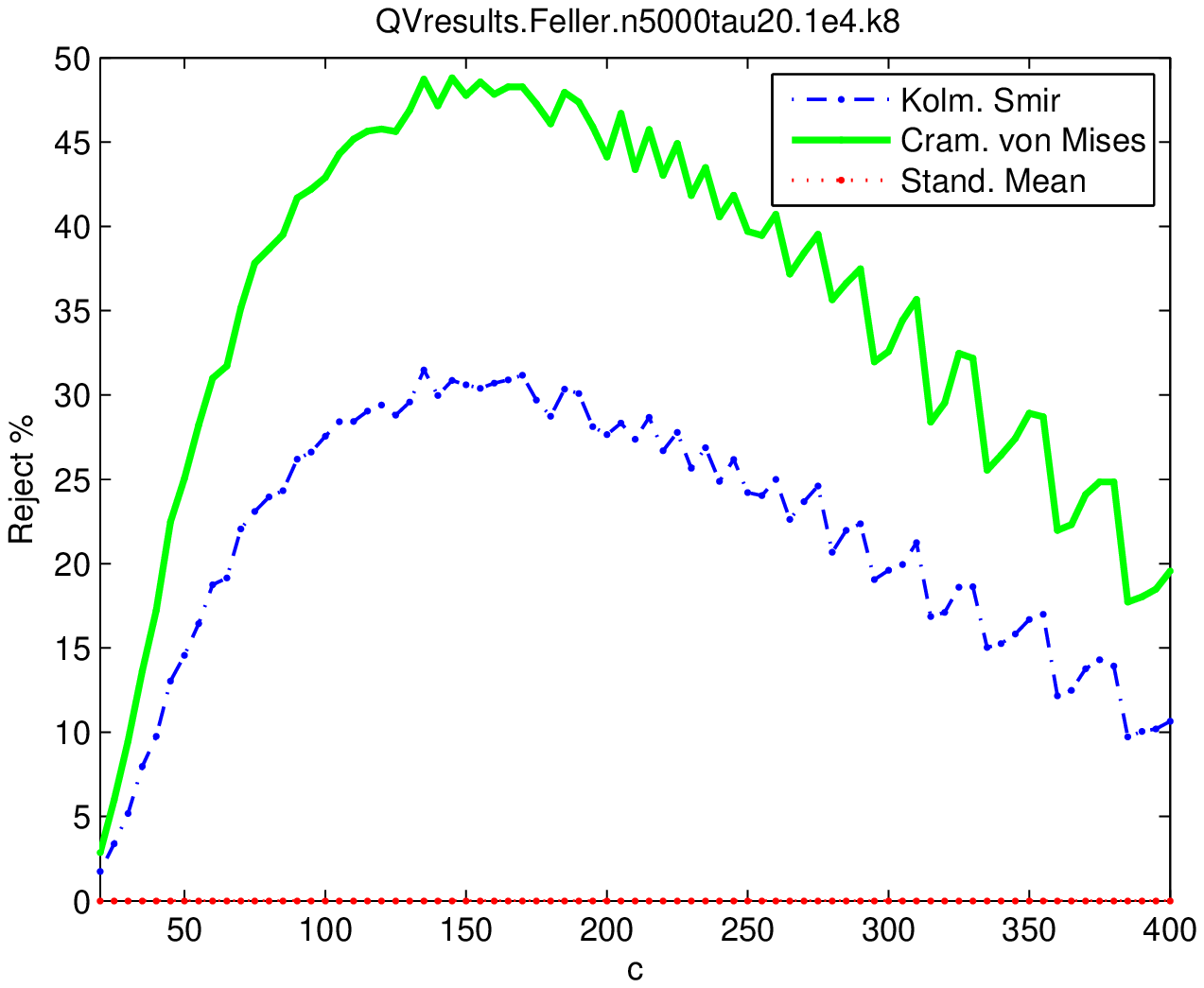}}
\caption{Rejection rates using sample quadratic variation for Feller's square-root diffusion. Long ($n=5000$) datasets. Left has drift coefficient $\ka=6$, right has drift coefficient $\ka=8$. Here, 95\% confidence intervals are 1\% or smaller.}
 \label{Fellerlong.fig}
\end{figure}

\subsection{Fractional Brownian Motion}
We simulated Fractional Brownian Motion (FBM), mainly as an example of a process with stationary increments that is not strong Markov. We simulated 1250 crossings by simulating long paths of Fractional Brownian Motion using the Wood and Chan algorithm \citep{woodchan94}. The method was used to produce a sequence $\{0,X_1,\ldots,X_m\}$, $m$ large, where $\Var(X(n))=\si^2n^{2H}$ with $\si^2=1/250$. With the transformation $\tilde{X}_i=\gamma^H\,X_i$ we use the self-similarity of FBM to produce a sequence, not on the positive integers, but at times separated by $\gamma=10^{-5}$ in our case. The choice of $\si^2=1/250$ allows direct comparison with our results for Brownian motion too, in that the Hurst parameter is different, but the variance is the same. Simulating long paths in this way, with a fine time resolution, we then find level 0 crossings with a time accuracy of $10^{-5}$. We then did a search for the crossing size $\de$ that gave an average of 1250 crossings by time 5. The crossings for this $\de$ were used as the data for the crossing tree method.

\begin{figure}[h!]
\resizebox{3in}{!}{\includegraphics{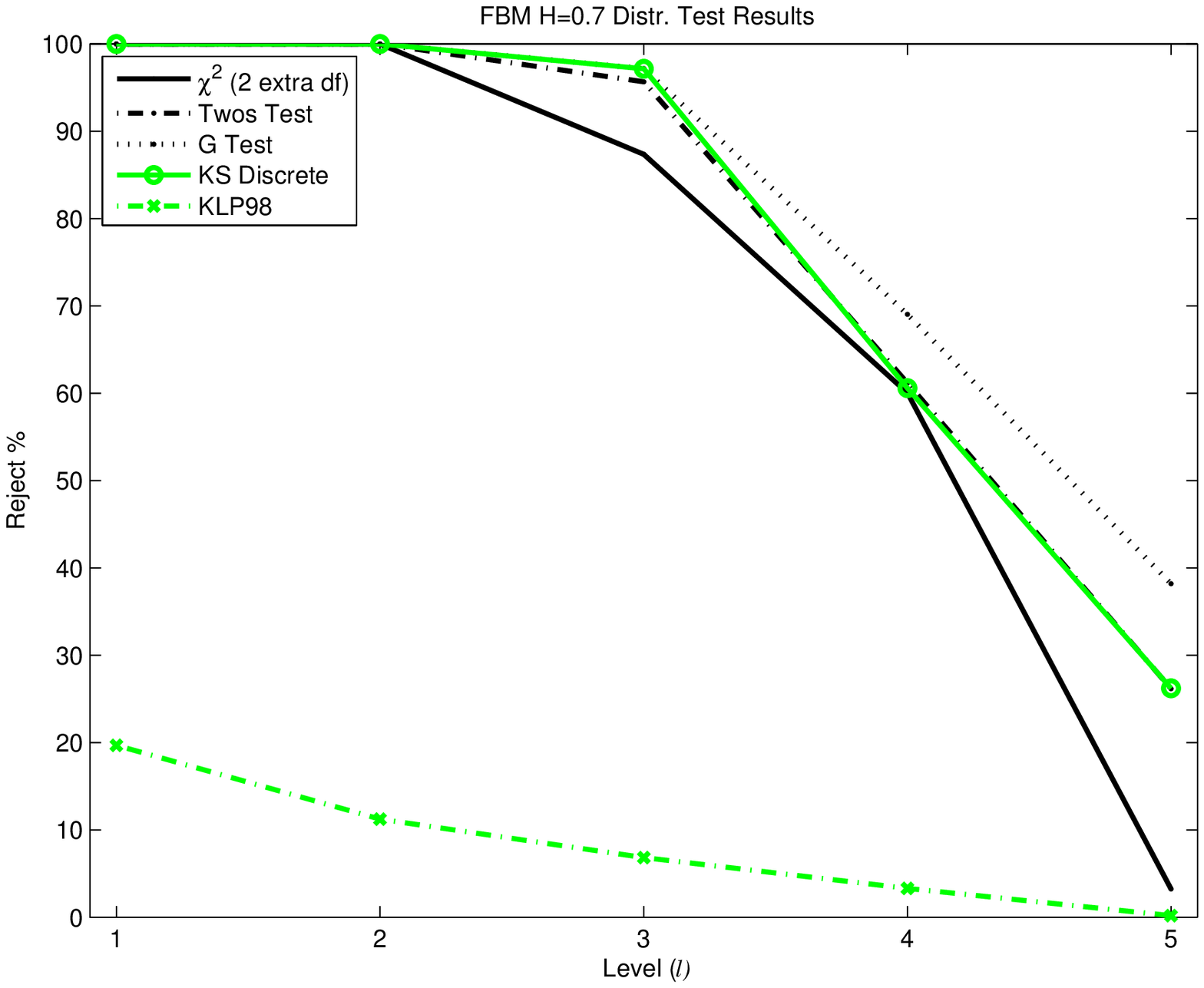}}\resizebox{3in}{!}{\includegraphics{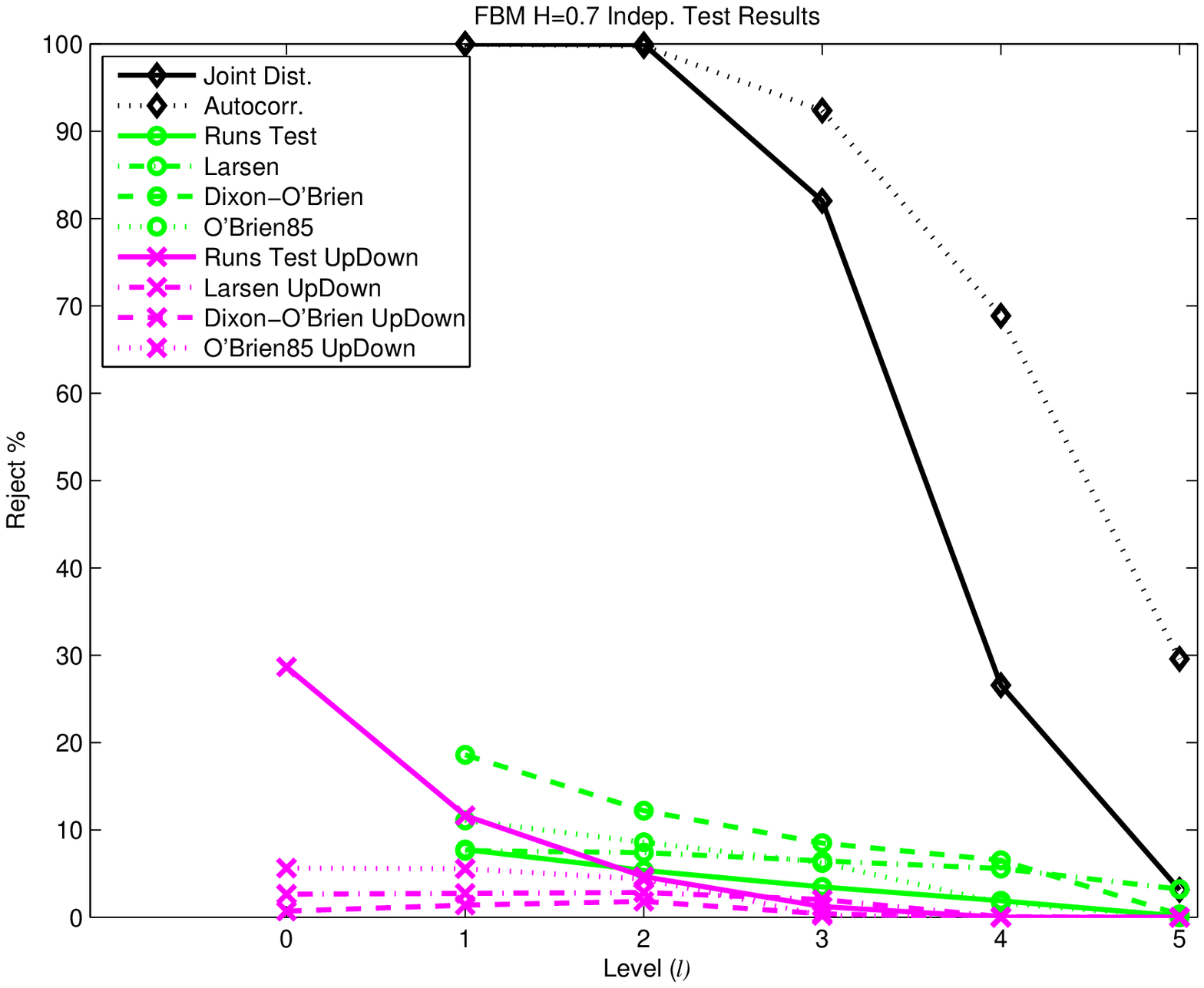}}
\caption{Crossing tree results for Fractional Brownian Motion  with $H=0.7$, $n = 1250$, $\de = 0.0010176$. Percentage of 10,000 sample paths rejected using distribution tests (left) and independence tests (right). Here, 95\% confidence intervals are 1\% or smaller.\label{FBM7short.fig}}
\end{figure}

\begin{table}[h!]
 {\scriptsize 
 \begin{center} 
 \begin{tabular}{ |c || *{6}{c|}} 
\hline \multicolumn{7}{|c|}{$H=0.7$, $\sigma=1$, $n = 1250$, $\de = 0.00101760$ }  \\ 
 \hline \hline 
& \multicolumn{6}{c|}{\raisebox{0ex}[12pt]{} {\bf \% of all} (\% of tested; \# tested)} \\ \hline {\bf levels }& 0 & 1 & 2 & 3 & 4 & 5 \\ \cline{2-7} 
 \hline \hline\raisebox{0ex}[12pt]{{\bf $\chi^2$ (+2 df)}}&  & {\bf 100.0} & {\bf 100.0} & {\bf  87.4} & {\bf  60.0} ( 60.6;   9903) & {\bf   3.2} ( 99.7;    325)\\ \hline 
{\bf Twos Test} &  & {\bf 100.0} & {\bf 100.0} & {\bf  95.7} & {\bf  61.2} & {\bf  26.2} ( 26.2;   9999)\\ \hline 
{\bf G Test} &  & {\bf 100.0} & {\bf 100.0} & {\bf  97.2} & {\bf  69.0} & {\bf  38.2} ( 38.2;   9999)\\ \hline 
{\bf KS Discrete} &  & {\bf 100.0} & {\bf 100.0} & {\bf  97.2} & {\bf  60.6} & {\bf  26.2}\\ \hline 
{\bf KLP98 Test} &  & {\bf  19.7} & {\bf  11.2} & {\bf   6.8} & {\bf   3.3} & {\bf   0.2}\\ \hline 
\hline 
{\bf Joint Dist.}&  & {\bf 100.0} & {\bf  99.9} & {\bf  82.0} & {\bf  26.6} ( 26.6;   9998) & {\bf   3.1} ( 10.6;   2953)\\ \hline 
{\bf Autocorr.} &  & {\bf 100.0} & {\bf  99.8} & {\bf  92.4} & {\bf  68.9} ( 68.9;   9997) & {\bf  29.6} ( 34.4;   8598)\\ \hline 
{\bf Runs Test} &  & {\bf   7.8} & {\bf   5.4} & {\bf   3.5} & {\bf   1.9} & {\bf   0.1}\\ \hline 
{\bf Larsen Test} &  & {\bf   7.6} & {\bf   7.4} & {\bf   6.5} & {\bf   5.6} & {\bf   3.2} (  3.2;   9915)\\ \hline 
{\bf  Dix.-OBri.} &  & {\bf  18.6} & {\bf  12.2} & {\bf   8.5} & {\bf   6.5} & {\bf   0.3} (  0.3;   9999)\\ \hline 
{\bf OBri85} &  & {\bf  11.1} & {\bf   8.6} (  8.6;   9995) & {\bf   6.2} (  7.1;   8779) & {\bf   1.6} (  4.5;   3607) & {\bf   0.0} (  0.0;     49)\\ \hline \hline 
{\bf Runs UD} & {\bf  28.6}  & {\bf  11.6} & {\bf   4.6} & {\bf   1.2} & {\bf   0.0} & {\bf   0.0}\\ \hline 
{\bf Larsen UD} & {\bf   2.6}  & {\bf   2.7} & {\bf   2.8} & {\bf   2.0} (  2.1;   9880) & {\bf   0.1} (  0.1;   7913) & {\bf   0.0} (  0.0;   4073)\\ \hline 
{\bf  Dix.-OBri. UD} & {\bf   0.7} & {\bf   1.4} & {\bf   1.8} & {\bf   0.4} & {\bf   0.0} (  0.0;   9591) & {\bf   0.0} (  0.0;   6577)\\ \hline 
{\bf OBri85 UD} & {\bf   5.6}  & {\bf   5.5} & {\bf   4.4} (  5.4;   8166) & {\bf   0.2} (  2.1;    999) & {\bf   0.0} (  0.0;      1) & {\bf   0.0} (  NaN;      0)\\ \hline 
\end{tabular} 
 \end{center}
 } 
\caption{Crossing tree results for Fractional Brownian Motion  with $H=0.7$, $n = 1250$, $\de = 0.0010176$. Percentage of 10,000 sample paths rejected, with an average of 1250 crossings in time 5. At higher levels, insufficient data length means some datasets are not tested.} \label{FBM7short.table}
\end{table}

\begin{figure}[htb!]
\resizebox{3in}{!}{\includegraphics{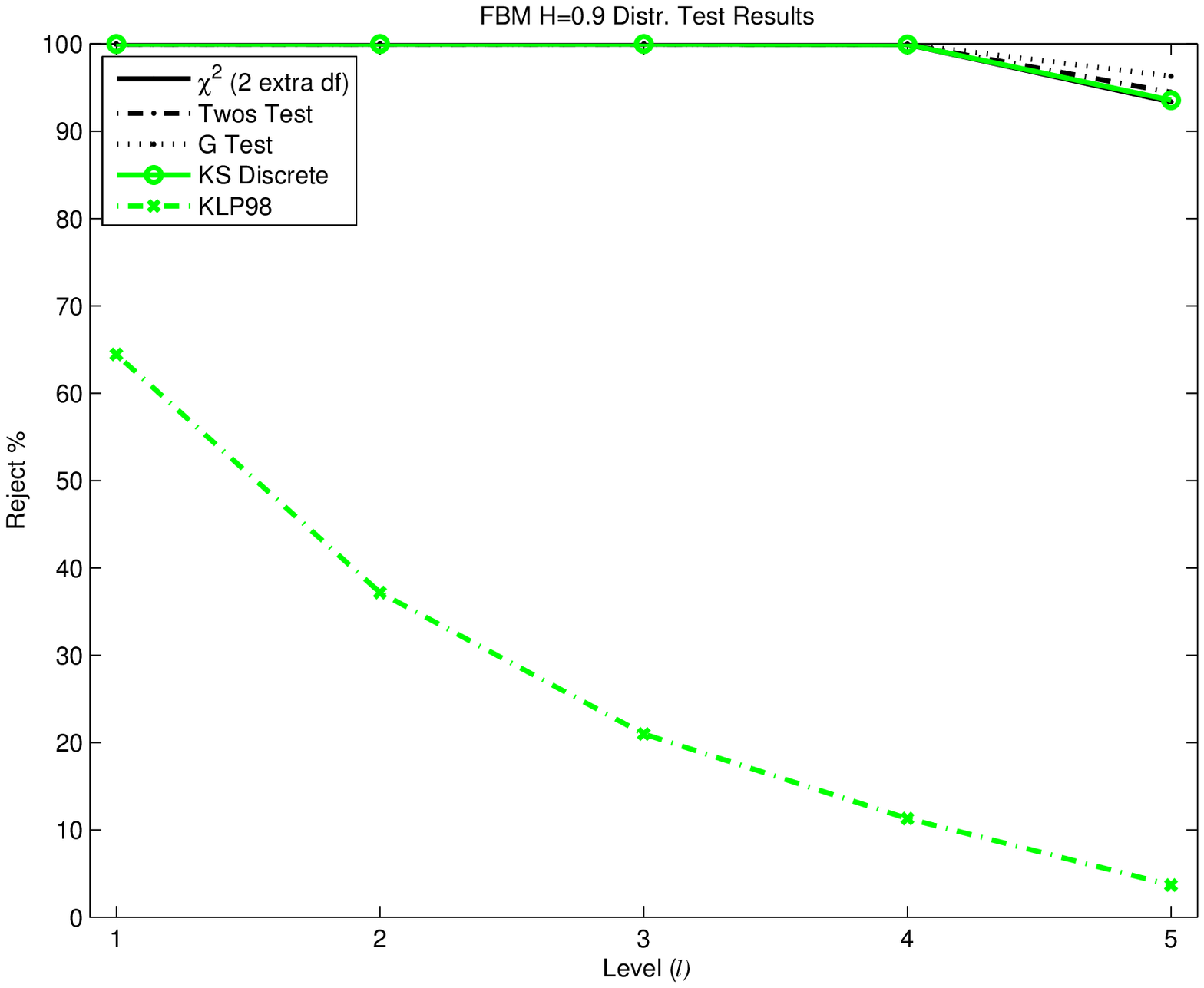}}\resizebox{3in}{!}{\includegraphics{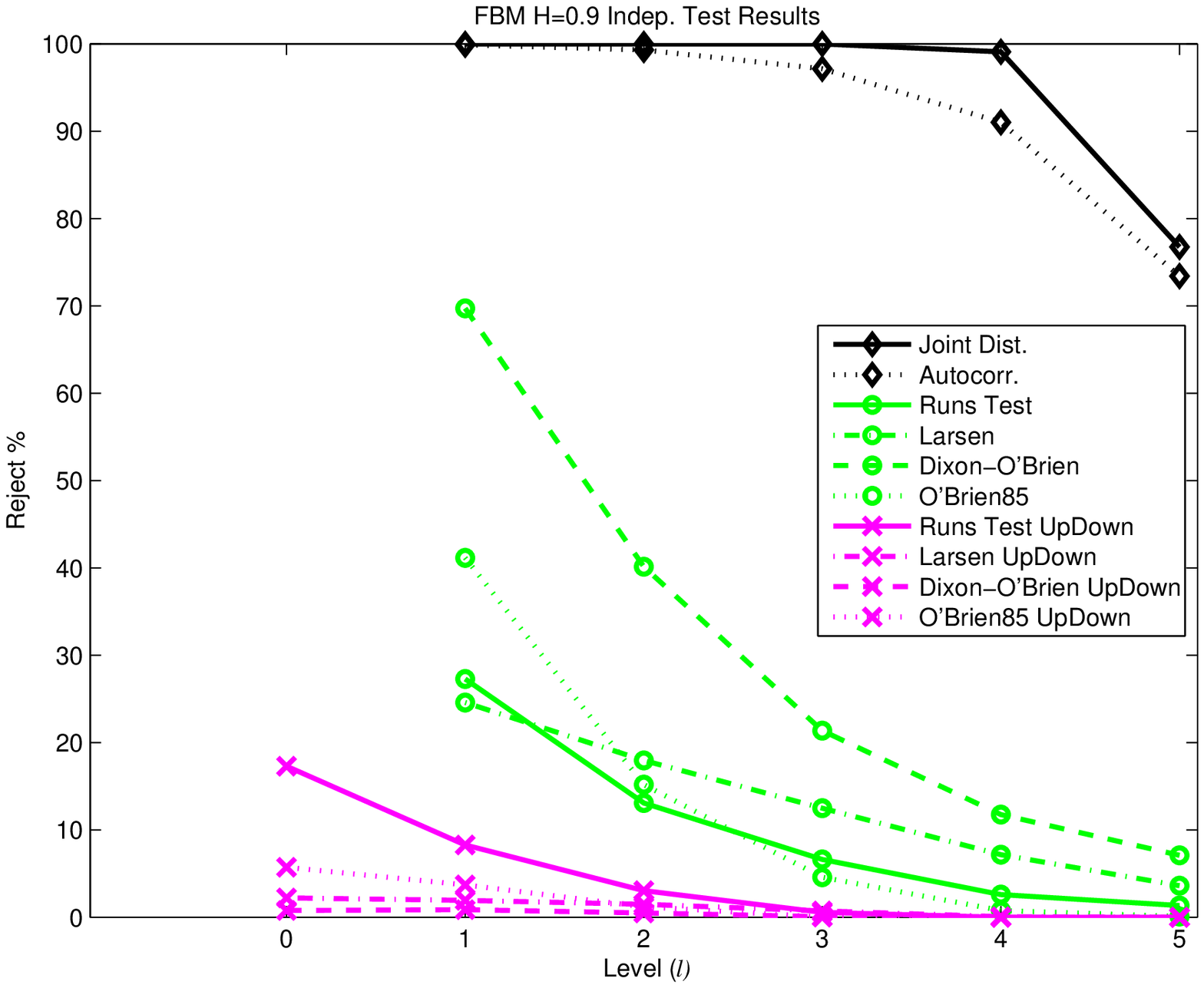}}
\caption{Crossing tree results for Fractional Brownian Motion  with $H=0.9$, $n = 1250$, $\de = 0.000312$. Percentage of 10,000 sample paths rejected using distribution tests (left) and independence tests (right). \label{FBM9short.fig}}
\end{figure}

\begin{table}[thb!]

 {\tiny
 \begin{center} 
 \begin{tabular}{ |c || *{6}{c|}} 
\hline \multicolumn{7}{|c|}{$H=0.9$, $\sigma=1$, $n = 1250$, $\de = 0.000312$ }  \\ 
 \hline \hline 
& \multicolumn{6}{c|}{\raisebox{0ex}[12pt]{} {\bf \% of all} (\% of tested; \# tested)} \\ \hline {\bf levels }& 0 & 1 & 2 & 3 & 4 & 5 \\ \cline{2-7} 
 \hline \hline\raisebox{0ex}[12pt]{{\bf $\chi^2$ (+2 df)}}&  & {\bf 100.0} & {\bf 100.0} & {\bf 100.0} & {\bf  99.9} & {\bf  93.4} ( 95.0;   9825)\\ \hline 
{\bf Twos Test} &  & {\bf 100.0} & {\bf 100.0} & {\bf 100.0} & {\bf  99.9} & {\bf  94.5}\\ \hline 
{\bf G Test} &  & {\bf 100.0} & {\bf 100.0} & {\bf 100.0} & {\bf 100.0} & {\bf  96.3}\\ \hline 
{\bf KS Discrete} &  & {\bf 100.0} & {\bf 100.0} & {\bf 100.0} & {\bf 100.0} & {\bf  93.5}\\ \hline 
{\bf KLP98 Test} &  & {\bf  64.4} & {\bf  37.2} & {\bf  21.0} & {\bf  11.3} & {\bf   3.7}\\ \hline 
\hline 
{\bf Joint Dist.}&  & {\bf 100.0} & {\bf 100.0} & {\bf 100.0} & {\bf  99.1} & {\bf  76.7} ( 76.8;   9993)\\ \hline 
{\bf Autocorr.} &  & {\bf  99.9} (100.0;   9989) & {\bf  99.3} (100.0;   9929) & {\bf  97.1} (100.0;   9714) & {\bf  91.0} ( 99.8;   9116) & {\bf  73.4} ( 95.0;   7725)\\ \hline 
{\bf Runs Test} &  & {\bf  27.3} & {\bf  13.1} & {\bf   6.6} & {\bf   2.6} & {\bf   1.3}\\ \hline 
{\bf Larsen Test} &  & {\bf  24.6} & {\bf  18.0} & {\bf  12.5} & {\bf   7.1} & {\bf   3.6}\\ \hline 
{\bf  Dix.-OBri.} &  & {\bf  69.7} & {\bf  40.1} & {\bf  21.4} & {\bf  11.8} & {\bf   7.1}\\ \hline 
{\bf OBri85} &  & {\bf  41.1} ( 46.7;   8810) & {\bf  15.2} ( 24.1;   6305) & {\bf   4.6} ( 14.4;   3182) & {\bf   0.8} (  7.3;   1064) & {\bf   0.1} (  4.0;    224)\\ \hline \hline 
{\bf Runs UD} & {\bf  17.3}  & {\bf   8.3} & {\bf   3.0} & {\bf   0.5} & {\bf   0.0} & {\bf   0.0}\\ \hline 
{\bf Larsen UD} & {\bf   2.2}  & {\bf   1.9} (  1.9;   9848) & {\bf   1.5} (  1.6;   9458) & {\bf   0.7} (  0.9;   8396) & {\bf   0.1} (  0.1;   6283) & {\bf   0.0} (  0.0;   3885)\\ \hline 
{\bf  Dx.OBr UD} & {\bf   0.8} & {\bf   0.9} (  0.9;   9935) & {\bf   0.5} (  0.5;   9719) & {\bf   0.1} (  0.1;   9165) & {\bf   0.0} (  0.0;   7904) & {\bf   0.0} (  0.0;   5707)\\ \hline 
{\bf OBri85 UD} & {\bf   5.7}  & {\bf   3.7} (  5.2;   7144) & {\bf   1.2} (  4.5;   2659) & {\bf   0.0} (  0.5;    211) & {\bf   0.0} (  0.0;      1) & {\bf   0.0} (  NaN;      0)\\ \hline 
\end{tabular} 
 \end{center}
 } 
\caption{Crossing tree results for Fractional Brownian Motion  with $H=0.9$, $n = 1250$, $\de = 0.000312$. Percentage of 10,000 sample paths rejected, with an average of 1250 crossings in time 5. At higher levels, insufficient data length means some datasets are not tested. } 
\label{FBM9short.table}
\end{table}

\begin{figure}[htb!]
\resizebox{3in}{!}{\includegraphics{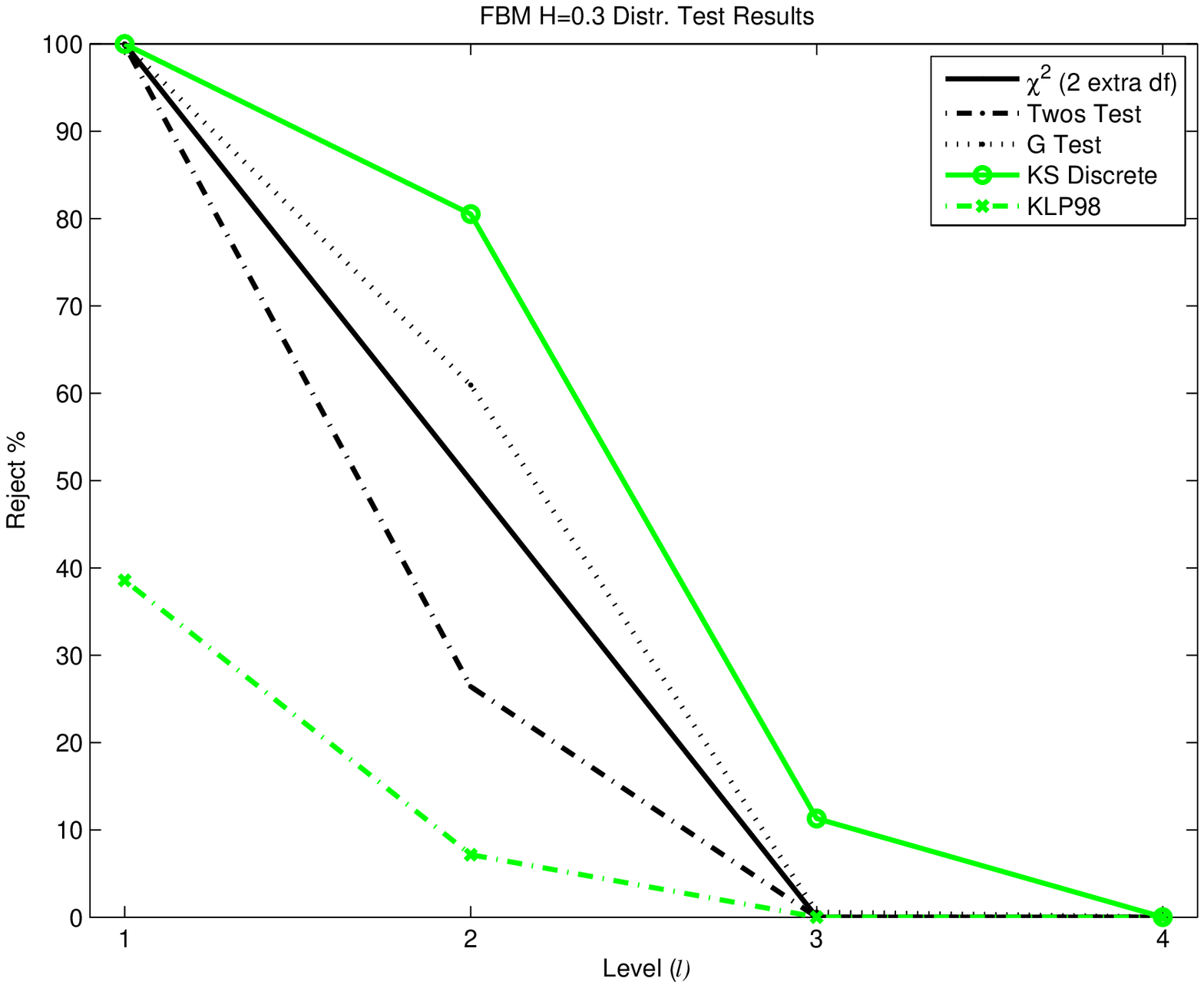}}\resizebox{3in}{!}{\includegraphics{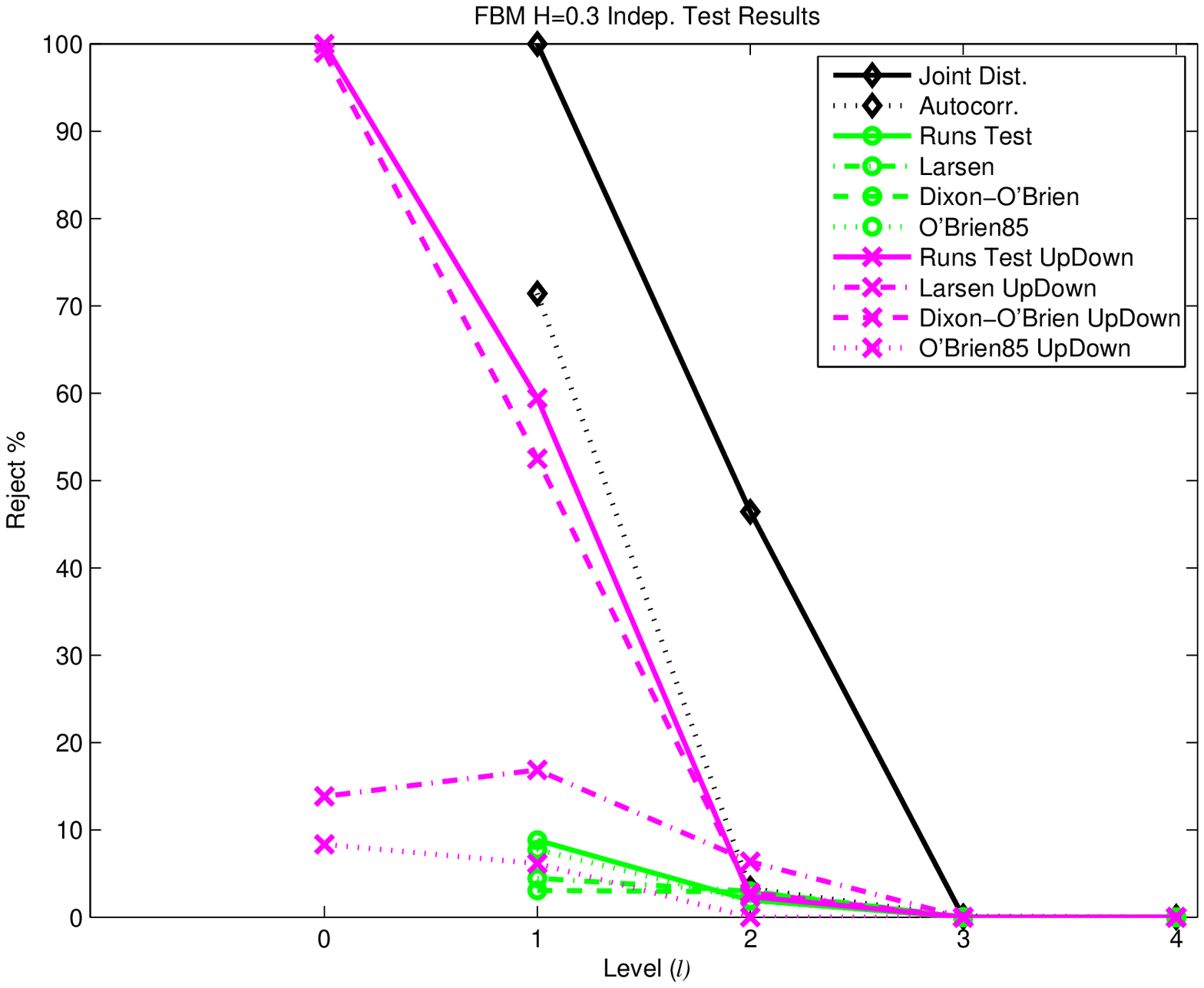}}
\caption{Crossing tree results for Fractional Brownian Motion  with $H=0.3$, $n = 1250$, $\de = 0.0174180$. Percentage of 10,000 sample paths rejected using distribution tests (left) and independence tests (right). Here, 95\% confidence intervals are 1\% or smaller.}
\label{FBM3short.fig}
\end{figure}

\begin{table}[thb!]

 {\small 
 \begin{center} 
 \begin{tabular}{ |c || *{5}{c|}} 
\hline \multicolumn{6}{|c|}{$H=0.3$, $\sigma=1$, $n = 1250$, $\de =  0.0174180$ }  \\ 
 \hline \hline 
& \multicolumn{5}{c|}{\raisebox{0ex}[12pt]{} {\bf \% of all} (\% of tested; \# tested)} \\ \hline {\bf levels }& 0 & 1 & 2 & 3 & 4 \\ \cline{2-6} 
 \hline \hline\raisebox{0ex}[12pt]{{\bf $\chi^2$ (+2 df)}}&  & {\bf 100.0} & {\bf  50.0} ( 61.9;   8084) & {\bf   0.0} (100.0;      1) & {\bf   0.0} (100.0;      1)\\ \hline 
{\bf Twos Test} &  & {\bf  99.9} & {\bf  26.4} & {\bf   0.0} (  0.0;   7708) & {\bf   0.0} (  0.2;    460)\\ \hline 
{\bf G Test} &  & {\bf 100.0} & {\bf  60.9} & {\bf   0.6} (  0.8;   7708) & {\bf   0.0} (  0.2;    460)\\ \hline 
{\bf KS Discrete} &  & {\bf 100.0} & {\bf  80.5} & {\bf  11.3} & {\bf   0.0}\\ \hline 
{\bf KLP98 Test} &  & {\bf  38.6} & {\bf   7.1} & {\bf   0.0} & {\bf   0.0}\\ \hline 
\hline 
{\bf Joint Dist.}&  & {\bf 100.0} & {\bf  46.4} ( 47.7;   9731) & {\bf   0.0} (100.0;      1) & {\bf   0.0} (100.0;      1)\\ \hline 
{\bf Autocorr.} &  & {\bf  71.4} & {\bf   3.3} (  3.3;   9997) & {\bf   0.1} (  4.6;    216) & {\bf   0.0} (100.0;      1)\\ \hline 
{\bf Runs Test} &  & {\bf   8.8} & {\bf   2.0} & {\bf   0.0} & {\bf   0.0}\\ \hline 
{\bf Larsen Test} &  & {\bf   4.5} & {\bf   3.0} (  3.0;   9899) & {\bf   0.0} (  0.1;   3957) & {\bf   0.0} (  0.0;    417)\\ \hline 
{\bf  Dix.-OBri.} &  & {\bf   3.1} & {\bf   2.9} & {\bf   0.0} (  0.0;   7708) & {\bf   0.0} (  0.0;    460)\\ \hline 
{\bf OBri85} &  & {\bf   7.7} (  7.7;   9997) & {\bf   1.9} (  5.3;   3571) & {\bf   0.0} (  NaN;      0) & {\bf   0.0} (  NaN;      0)\\ \hline \hline 
{\bf Runs UD} & {\bf 100.0}  & {\bf  59.4} & {\bf   2.4} & {\bf   0.0} & {\bf   0.0}\\ \hline 
{\bf Larsen UD} & {\bf  13.8}  & {\bf  16.9} & {\bf   6.3} (  6.8;   9310) & {\bf   0.0} (  0.0;   2361) & {\bf   0.0} (  0.0;      6)\\ \hline 
{\bf  Dix.-OBri. UD} & {\bf  99.0} & {\bf  52.5} & {\bf   2.8} (  2.8;   9995) & {\bf   0.0} (  0.0;   4594) & {\bf   0.0} (  0.0;      7)\\ \hline 
{\bf OBri85 UD} & {\bf   8.3}  & {\bf   6.2} & {\bf   0.0} (  0.4;    457) & {\bf   0.0} (  NaN;      0) & {\bf   0.0} (  NaN;      0)\\ \hline 
\end{tabular} 
 \end{center}
 } 
\caption{Crossing tree results for Fractional Brownian Motion  with $H=0.3$, $n = 1250$, $\de = 0.0174180$. Percentage of 10,000 sample paths rejected, with an average of 1250 crossings in time 5. At higher levels, insufficient data length means some datasets are not tested.} \label{FBM3short.table}
\end{table}

Three values for the Hurst parameter ($H$) were chosen: 0.7, 0.9, and 0.3. Values of $H$ in (1/2,\,1) provide long-range dependent increments, while values in $(0,\,1/2)$ provide ``anti-persistent'' increments. The results for 0.7 and 0.9 are largely similar. Amongst the distribution tests (Figures \ref{FBM7short.fig} (left) and \ref{FBM9short.fig} (left)), all but KLP98 are effective in rejecting nearly all the paths across a range of levels. The range is larger for higher $H$, suggesting that fewer subcrossings numbers are needed to reject when the dependence is stronger. Amongst the independence tests (Figures \ref{FBM7short.fig} (right) and \ref{FBM9short.fig} (right)), there is a slight difference between the comparison of the joint distribution and autocorrelation tests. For $H=0.7$ the power of the autocorrelation test is similar to that of the distribution tests, while power of the joint distribution test decays much faster with increasing level. With $H=0.9$, the power of both remains high- $H=0.9$ is just more easily rejected at higher levels. In comparison, Figure \ref{FBM3PVshort.fig} shows that with our implementation of the quadratic method, the power is essentially the same, in that both the KS and CVM tests have near 100\% rejection in the best case, althought for $H=0.7$ the power would  depend more critically on the value of $c$ used. For $H=0.9$ rejection is clearer, with near 100\% rejection for all values of $c$.

The anti-persistent case $H=0.3$ shows a striking difference from the long-range dependent case. For all the tests, the power falls off dramatically with higher power. Amongst the distribution tests, the KS discrete test decays slowest, while amongst the independence tests, the joint distribution test decays slowest. Based on level 0, 100\% of the paths would be rejected which is higher than from our implementation of the quadratic  variation method: 46\%(KS) and 65\% (CVM).

\begin{figure}[hb!]
\resizebox{3in}{!}{\includegraphics{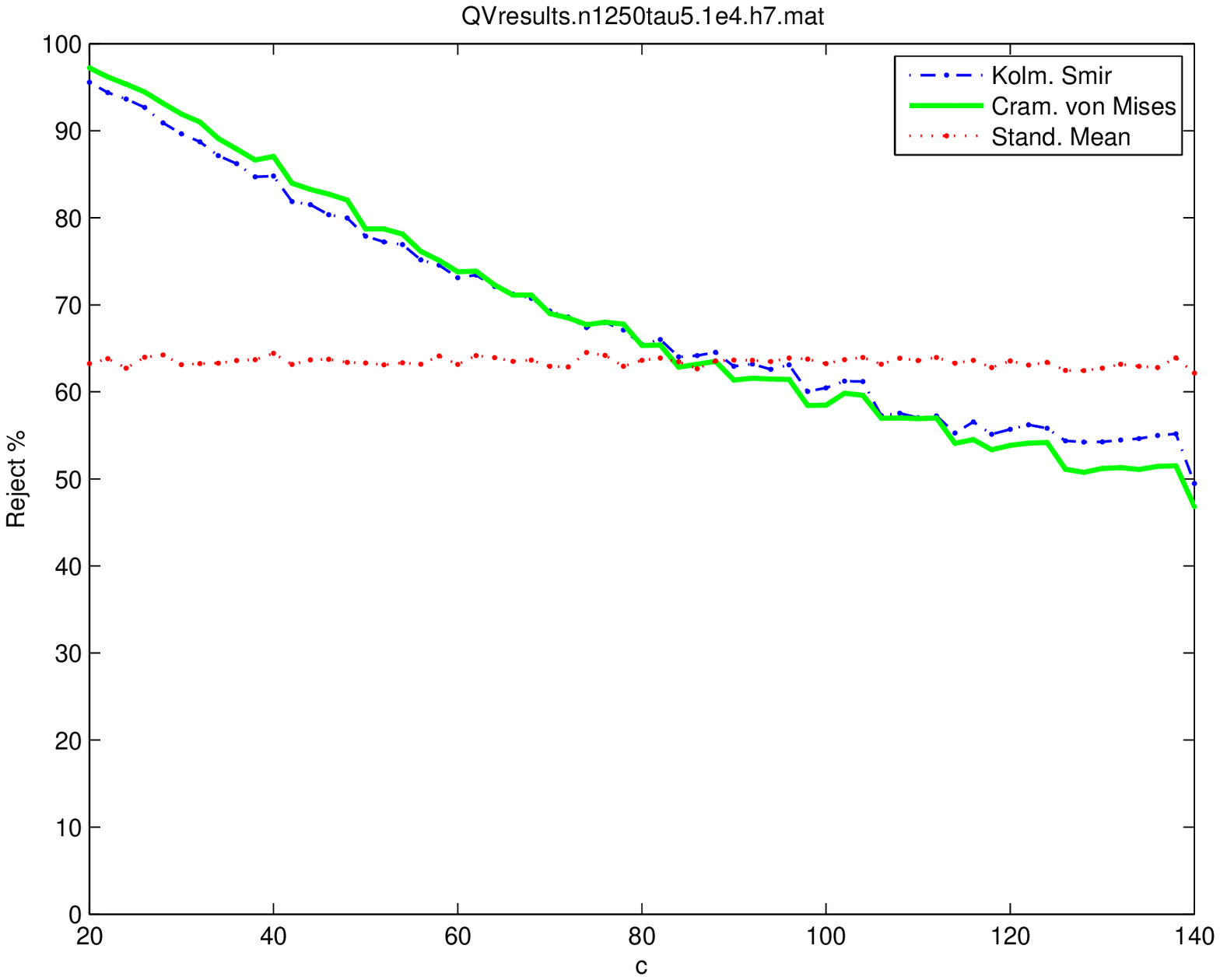}}\resizebox{3in}{!}{\includegraphics{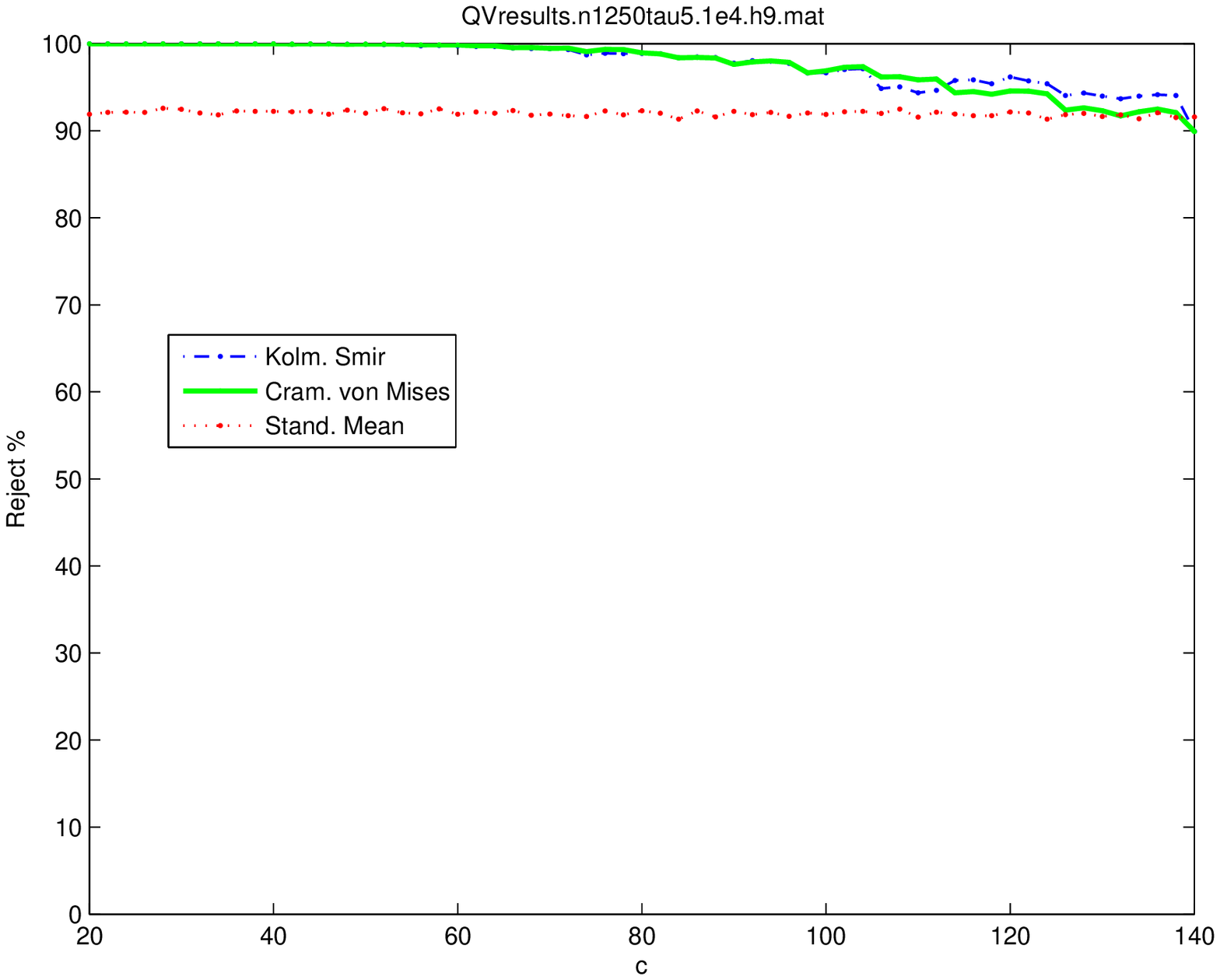}}
\caption{Quadratic variation results for Fractional Brownian Motion with length 1250 and $H=0.7$ (left) and $H=0.9$ (right). Here, 95\% confidence intervals are 1\% or smaller.}
\label{FBMshort.fig}
\end{figure}

\begin{figure}[hb!]
\resizebox{3in}{!}{\includegraphics{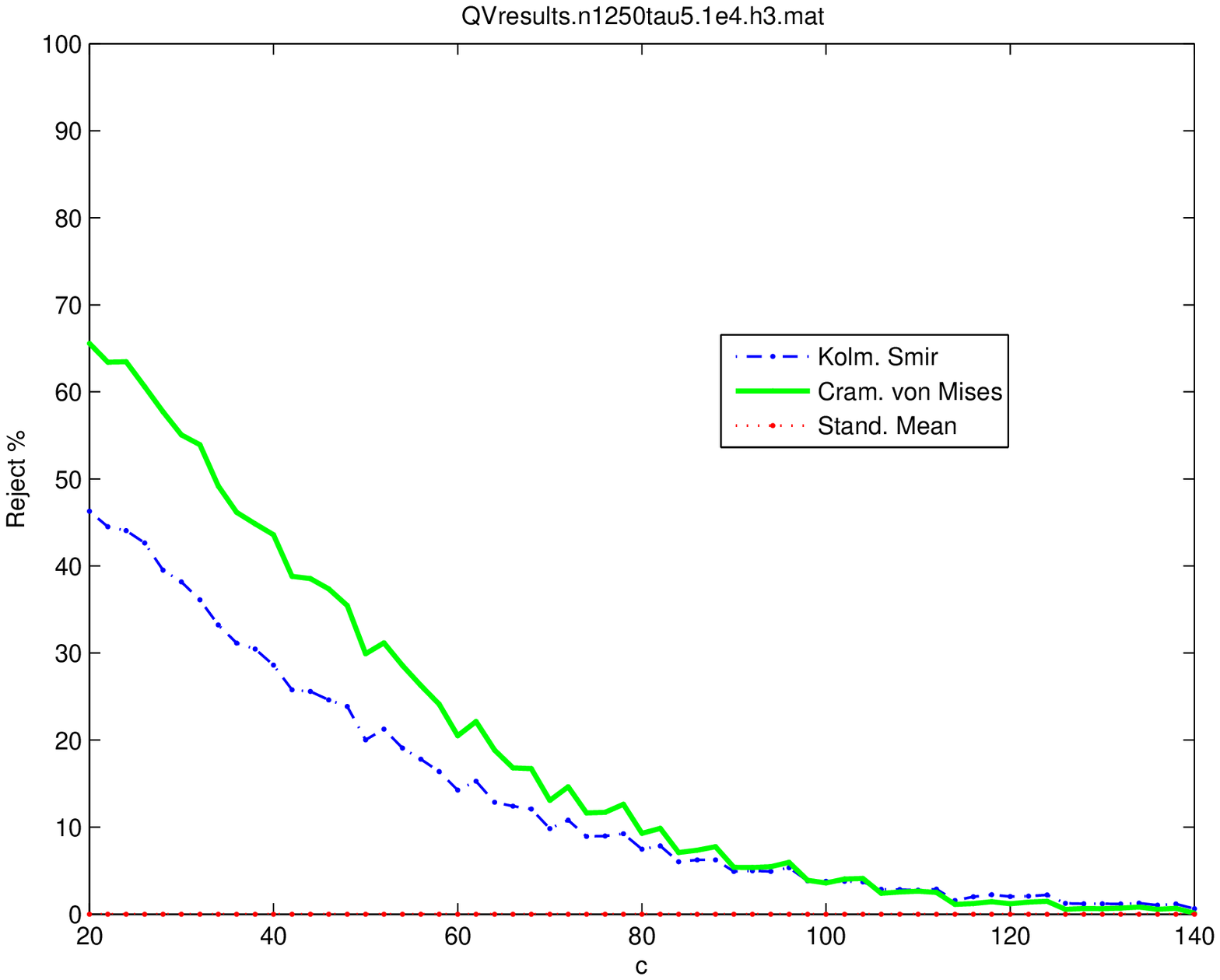}}
\caption{Quadratic variation results for Fractional Brownian Motion with length 1250 and $H=0.3$. Here, 95\% confidence intervals are 1\% or smaller.}
\label{FBM3PVshort.fig}
\end{figure}
 
\section{Analysis of foreign exchange rates}\label{FX.sec}

In this section we show the results of applying the crossing tree-based test to foreign exchange rate tick data obtained from the Securities Industry Research Centre of Asia-Pacific (SIRCA). We have five series each representing a different exchange rate, for the period January to December 2003:
\begin{enumerate}
\item Australian Dollar to US Dollar (AUD-USD), 
\item Euro to UK Pound(EUR-GBP),
\item Euro to US Dollar (EUR-USD), 
\item UK Pound to US Dollar (GBP-USD), 
\item Japanese Yen to US Dollar (JPY-USD).
\end{enumerate}
 As is common for considering finance data, we work with log-transformed data. Table \ref{FXdata_summary.table} lists some of the statistical properties of the increments of the log-transformed data. The sample mean is small in relation to the standard deviation, suggesting a negligible effect from drift can be expected.

\begin{table}[hb]
\begin{center}
\begin{tabular}{ |l c c c c c|}
\hline 
 Series & Length & Sample & Standard  & Sample& Sample \\ 
  & & Mean & Deviation & Skewness & Kurtosis \\ \hline
AUD-USD & 450,641 & \raisebox{0ex}[12pt] 6.45$\times 10^{-7} $ & 2.47$\times 10^{-4}$  & 0.2242 & 33.3344 \\
EUR-GBP & 166,024 & 4.81$\times 10^{-7}$  & 2.69$\times 10^{-4}$  & -1.0331 & 63.2093 \\
EUR-USD & 6,086,353 & 3.00$\times 10^{-8}$  & 1.15$\times 10^{-4}$  & -0.4466 & 52.9049 \\
GBP-USD & 3,608,072 & 2.86$\times 10^{-8}$  & 9.48$\times 10^{-5}$  & -0.1496 & 22.4587 \\
JPY-USD & 4,184,020 & -2.41$\times 10^{-8}$  & 1.05$\times 10^{-4}$  & -0.2611 & 56.8353 \\ \hline
\end{tabular}
\end{center}
\caption{Statistical properties of the increments of the log-transformed 2003 financial exchange rate tick data.} \label{FXdata_summary.table}
\end{table}

\subsection{AUD-USD}
Table \ref{AUDUSD.table} shows the results from applying the crossing tree test to the AUD-USD data at levels 0 to 7. Because these are results for one dataset, the numbers in the body of this table are quite different from those in the tables for the simulation data. Here we show $p$-values from applying tests at particular levels. Where our tests simply reject or not, without $p$-values, we show ``$<0.05$'' to denote rejecting the null hypothesis and ``$>0.05$'' where the null hypothesis is not rejected. Where a test rejects at 5\% significance we also show the result in boldface. The number of subcrossing numbers available to the tests at each level is shown in the second row.

A clear observation is that the tests of distribution and independence universally reject the null hypothesis at levels 1 and 2. Further investigation reveals the reason is the assumption of continuity in the null hypothesis. When the data are not exactly crossings already, crossings are found by linear interpolation between known data points. A large number of crossings between any two adjacent data points suggests $\de$ is small in relation to the vertical distance between the points. In these cases there won't be excursions (i.e., up-down or down-up pairs.) When a high percentage of the crossings arise in this way the tests will reject, essentially as a consequence of small $\de$, an indication that the assumption of continuity of paths cannot be supported by the data. Recall, for example, that the subcrossings of level 1 are the crossings of level 0. As shown in Table \ref{xing.table}, at level 0, 34\% of the crossings result from two or more crossings between data points, and 7.8\% of the crossings result from 4 or more crossings between data points. 

\begin{table}[htb!]

 {\small 
 \begin{center} 
 \begin{tabular}{ |c || *{8}{c|}} 
 \hline  \multicolumn{9}{|c|}{{\bf AUD-USD}} \\ \hline 
 \hline {\bf levels ($l$)}& 0 & 1 & 2 & 3 & 4 & 5 & 6 & 7 \\ \hline 
{\bf \# SubX} & & 99157 & 26341 & 6734 & 1702 & 415 & 100 & 25 \\ \hline 
{\bf mean xing length} & &  &  &  &  &  &  &  \\  
{\bf level ($l-1$)} & & \raisebox{1.5ex}[0pt]{1.2 min} & \raisebox{1.5ex}[0pt]{3.9 min} & \raisebox{1.5ex}[0pt]{14.6 min} & \raisebox{1.5ex}[0pt]{57 min} & \raisebox{1.5ex}[0pt]{3.8 h} & \raisebox{1.5ex}[0pt]{15.4 h} & \raisebox{1.5ex}[0pt]{63 h} \\ \hline \hline 
\raisebox{0ex}[12pt]{ $\mathbf{ \chi^2}$ {\bf (+2 df)} }  &  & {\bf 0.000} & {\bf 0.000} & 0.107 & 0.808 & 0.441 & 0.628 & $>0.05$ \\ \hline 
{\bf Twos Test} &  & {\bf 0.000} & {\bf 0.000} & {\bf 0.007} & 0.645 & 0.492 & 1.000 & 0.690 \\ \hline 
{\bf G Test} &  & {\bf 0.000} & {\bf 0.000} & 0.067 & 0.734 & 0.341 & 0.648 & 0.468 \\ \hline 
{\bf KS} &  & $\mathbf{<0.05}$ & $\mathbf{<0.05}$ & $\mathbf{<0.05}$ & $>0.05$ & $>0.05$ & $>0.05$ & $>0.05$ \\ \hline 
{\bf KLP98} &  & {\bf 0.000} & {\bf 0.000} & 0.084 & 0.468 & 0.091 & 0.603 & 0.374 \\ \hline \hline 
{\bf Joint Dist.} &  & {\bf 0.000} & {\bf 0.000} & 0.322 & 0.865 & 0.239 & 0.576 & 0.573 \\ \hline 
{\bf Autocorr} &  & $\mathbf{<0.05}$ & $\mathbf{<0.05}$ & $>0.05$ & $>0.05$ & $>0.05$ & $\mathbf{<0.05}$ & $>0.05$ \\ \hline 
{\bf Runs} &  & {\bf 0.000} & {\bf 0.000} & 0.867 & 0.676 & 0.901 & 0.055 & 0.447 \\ \hline 
{\bf Larsen} &  & {\bf 0.000} & {\bf 0.003} & 0.223 & 0.267 & 0.463 & 0.807 & $>0.05$ \\ \hline 
{\bf DixOB} &  & {\bf 0.000} & {\bf 0.000} & {\bf 0.027} & 0.712 & 0.647 & 0.941 & 0.321 \\ \hline 
{\bf OBri85} &  & {\bf 0.000} & {\bf 0.000} & {\bf 0.004} & 0.378 & 0.545 & 0.279 & {\bf 0.034} \\ \hline \hline 
{\bf \# UD pairs}  & 60946 & 23237 & 6421 & 1652 & 427 & 102 & 22 & 4 \\ \hline 
{\bf Runs UD} & {\bf 0.000} & 0.782 & 0.080 & 0.104 & 0.893 & 0.920 & 1.000 & 1.000 \\ \hline 
{\bf Larsen UD} & 0.969 & {\bf 0.002} & 0.394 & 0.863 & 0.519 & 0.678 & $>0.05$ & $>0.05$ \\ \hline 
{\bf DixOB UD} & 1.000 & {\bf 0.003} & 0.109 & 0.983 & 0.422 & 0.626 & 0.510 & $>0.05$ \\ \hline 
{\bf OBri85 UD} & {\bf 0.000} & {\bf 0.000} & 0.556 & 0.724 & 0.233 & 0.785 & 0.635 & NA \\ \hline 
\end{tabular} 
 \end{center}
 } 
\caption{Crossing tree test results for 2003 AUD-USD exchange rate tick data. Distribution and independence tests clearly reject the null hypothesis at levels 1 and 2, most likely because the data does not support the continuous path assumption at those scales.} \label{AUDUSD.table}
\end{table}

At higher levels, apart from the a couple of anomalous results, the tests do not reject the null hypothesis at levels 4 and above. Level 3 is clearly the transition scale. A decision to reject level 3 depends on how one treats the tests. Our $\chi^2$, joint distribution autocorrelation and G tests do not reject. Our simulation results show that  our $\chi^2$ and joint distribution tests are the  most powerful in the absence of drift.  With drift, the G test and autocorrelation tests were the most powerful. Moreover, only with Feller's square root diffusion process was the power of the discrete KS test noteworthy, and even then it was second (or third) most powerful after our $\chi^2$ and joint distribution test. Perhaps the key here is that the Twos test rejects at level 3 suggesting the number of twos in the subcrossings is unusually high or low. A high percentage of twos is consistent with an effect of multiple crossings from interpolation, suggesting that rejecting level 3 is appropriate on account of small $\de$. 

Row 3 of Table \ref{AUDUSD.table} shows how the levels correspond to mean crossing times. For example, the subcrossings of level 1 are determined using the crossings of level 0. Since there are are 320,810 level 0 crossings, and approximately 6,396 hours for trading currency in a year, the mean length of a level 0 crossing is about 1.2 minutes. So we say the results of level 1 correspond to an average crossing length of 1.2 minutes. Since our tests do not reject above level 3, we would say the crossing tree technique does not reject the null hypothesis at a time scale between 15 and 57 minutes and above.

One might ask if the consistent rejection of the tests at lower levels is simply because of the size of the datasets tested. For example, at level 1, 99,157 numbers are tested in the distribution and independence tests. As a simple check for this, we split the data into non-overlapping blocks of length 1024 and applied a subset of our tests to the separate blocks. Partial blocks were discarded. So, for the AUD-USD data this yields 96 blocks of length 1024 at level 1. If, for a test, the results from Table \ref{AUDUSD.table} are obtained consistently from the blocks it would be strong evidence that the results for lower levels in that table are not an artifact of data length. For our $\chi^2$ test, 95 of 96 blocks are rejected. The numbers for some other tests are similar: Twos test (96), G test (95), KS discrete (96) and joint distribution test (94). Of the tests we tried, the autocorrelation test rejected the fewest of the blocks, rejecting 81 out of 96. These results are consistent across all five of our exchange rate datasets, including the outlier behaviour of the autocorrelation test rejecting from 31\% (EUR-USD) to 100\% (EUR-GBP) of the blocks. Beyond being a curiosity, the behaviour of the autocorrelation test here doesn't affect our results since the other tests consistently reject almost all blocks. As such, we won't comment more on these block results. 

\small
\begin{table}
\begin{tabular}{|c||*7{c|c||}}
\hline levels & \multicolumn{2}{c||}{0} & \multicolumn{2}{c||}{1} & \multicolumn{2}{c||}{2} & \multicolumn{2}{c||}{3} & \multicolumn{2}{c||}{4} & \multicolumn{2}{c||}{5} & \multicolumn{2}{c||}{6}  \\ \hline
	& $\ge 2$ & $\ge 4$ & $\ge 2$ & $\ge 4$ & $\ge 2$ & $\ge 4$ & $\ge 2$ & $\ge 4$ & $\ge 2$ & $\ge 4$ & $\ge 2$ & $\ge 4$ & $\ge 2$ & $\ge 4$ \\ \hline
	AUD-USD & 34.2& 7.8 &8.8  &0.8 &1.1 &0.4 & 0.6 &0.3 &0.5 &0.3 &0   &0  &0  &0\\
	EUR-GBP &48.6 &14.8 &15.3 &2.4 &3   &1   &1.5  &0.7 &1.2 &0.9 &0.7 &0  &0  &0\\
	EUR-USD &44   &11   &13.6 &0.7 &1.2 &0.2 &0.6  &0.4 &0.8 &0.5 &1   &0.3&0.6&0\\
	GBP-USD &50.7 &18.5 &22   &2.3 &3.2 &0.3 &0.8  &0.4 &0.7 &0.3 &0.5 &0  &0  &0\\ 
	JPY-USD &44.6 &13   &16.6 &2   &3   &0.3 &0.7  &0.4 &0.8 &0.5 &0.6 &0.3 &0 &0 \\ \hline
\end{tabular}
	\caption{Percentage of crossings arising as either $\ge 2$ crossings (left column for each level) or $\ge 4$ crossings (right colunn for each level) between pairs of data points. High percentages suggest the data doesn't support the continuity assumption.}
	\label{xing.table}
\end{table}
\normalsize

\subsection{EUR-GBP}
Table \ref{EURGBP.table} shows results for the crossing tests applied to the EUR-GBP exchange rate data. This dataset is about one third of the length of the AUD-USD dataset, but the number of subcrossings at level one is about half. This difference is easily explained by the higher percentage of multiple crossings arising from the interpolation between data points in the EUR-GBP data, as shown in Table \ref{xing.table}. It appears levels 3 and below reject because continuity isn't supported. 

Level 4 is a transition level. Rejection by our $\chi^2$ test is noteworthy, and one might reject the null at level 4 on the basis of this test. However, the result from the Twos test doesn't suggest the number of twos in the subcrossings is unusual, suggesting the rejection by the $\chi^2$ test is due to more than just small $\de$. This is further supported by the observation that the percentage of multiple crossings from interpolation is not very different between levels 4 and 5, yet several tests reject at level 4 and none reject at level 5. Again it seems there are additional features in the data causing the tests to reject at level 4. Finally, no tests reject at level 5, making the results at level 6 more meaningful- $\de$ is clearly not too small. From the results of levels 4 and 6, one cannot clearly say continuous time-changed Brownian motion is not rejected {\it above some fixed level}. Thus, the use of a continuous time-changed Brownian motion above some threshold level seems unsupported for this dataset.

\begin{table}[h]

 {\small 
 \begin{center} 
 \begin{tabular}{ |c || *{8}{c|}} 
 \hline  \multicolumn{9}{|c|}{{\bf EUR-GBP}} \\ \hline 
 \hline {\bf levels ($l$)}& 0 & 1 & 2 & 3 & 4 & 5 & 6 & 7 \\ \hline 
{\bf \# SubX} & & 52775 & 15728 & 4245 & 1030 & 253 & 80 & 20 \\ \hline 
{\bf mean xing length} & &  &  &  &  &  &  &  \\  
{\bf level ($l-1$)} & & \raisebox{1.5ex}[0pt]{2.6 min} & \raisebox{1.5ex}[0pt]{7.3 min} & \raisebox{1.5ex}[0pt]{24.4 min} & \raisebox{1.5ex}[0pt]{1.5 h} & \raisebox{1.5ex}[0pt]{6.2 h} & \raisebox{1.5ex}[0pt]{25 h} & \raisebox{1.5ex}[0pt]{79 h} \\ \hline \hline 
\raisebox{0ex}[12pt]{ $\mathbf{ \chi^2}$ {\bf (+2 df)} }  &  & {\bf 0.000} & {\bf 0.000} & {\bf 0.000} & {\bf 0.012} & 0.930 & 0.181 & $>0.05$ \\ \hline 
{\bf Twos Test} &  & {\bf 0.000} & {\bf 0.000} & {\bf 0.000} & 0.334 & 0.530 & {\bf 0.018} & 0.824 \\ \hline 
{\bf G Test} &  & {\bf 0.000} & {\bf 0.000} & {\bf 0.000} & {\bf 0.026} & 0.891 & {\bf 0.034} & 0.598 \\ \hline 
{\bf KS} &  & $\mathbf{<0.05}$ & $\mathbf{<0.05}$ & $\mathbf{<0.05}$ & $>0.05$ & $>0.05$ & $\mathbf{<0.05}$ & $>0.05$ \\ \hline 
{\bf KLP98} &  & {\bf 0.000} & {\bf 0.000} & {\bf 0.000} & {\bf 0.031} & 0.542 & 0.675 & 0.359 \\ \hline \hline 
{\bf Joint Dist.} &  & {\bf 0.000} & {\bf 0.000} & {\bf 0.000} & 0.172 & 0.120 & 0.251 & 0.093 \\ \hline 
{\bf Autocorr} &  & $\mathbf{<0.05}$ & $\mathbf{<0.05}$ & $>0.05$ & $>0.05$ & $>0.05$ & $>0.05$ & $>0.05$ \\ \hline 
{\bf Runs} &  & {\bf 0.000} & {\bf 0.013} & 0.382 & 0.212 & 0.924 & 0.909 & 0.110 \\ \hline 
{\bf Larsen} &  & {\bf 0.000} & {\bf 0.000} & {\bf 0.003} & {\bf 0.003} & 0.201 & $\mathbf{<0.05}$ & $>0.05$ \\ \hline 
{\bf DixOB} &  & {\bf 0.000} & {\bf 0.000} & 0.605 & 0.190 & 0.700 & 0.373 & 0.187 \\ \hline 
{\bf OBri85} &  & {\bf 0.000} & {\bf 0.000} & 0.091 & {\bf 0.024} & 0.573 & 0.521 & 0.845 \\ \hline \hline 
{\bf \# UD pairs}  & 22175 & 10533 & 3708 & 1092 & 242 & 46 & 20 & 8 \\ \hline 
{\bf Runs UD} & {\bf 0.000} & {\bf 0.000} & {\bf 0.001} & 0.665 & 0.655 & 0.891 & 0.484 & 0.057 \\ \hline 
{\bf Larsen UD} & 0.674 & 0.999 & 0.173 & 0.392 & 0.256 & $>0.05$ & $\mathbf{<0.05}$ & $>0.05$ \\ \hline 
{\bf DixOB UD} & 1.000 & 1.000 & 0.985 & 0.254 & 0.641 & 0.135 & $>0.05$ & $>0.05$ \\ \hline 
{\bf OBri85 UD} & {\bf 0.001} & 0.496 & {\bf 0.036} & 0.230 & 0.814 & 0.203 & 0.201 & NA \\ \hline 
\end{tabular} 
 \end{center}
 } 
\caption{Crossing tree test results for 2003 EUR-GBP exchange rate tick data. Distribution and independence tests clearly reject the null hypothesis at levels 1 to 3: the data does not support the continuous path assumption at those scales. Results at levels 4 and 6 seem to reject a continuous chronometer.}  \label{EURGBP.table}
\end{table}

\subsection{EUR-USD}
Table \ref{EURUSD.table} shows results from applying the crossing tree tests to the EUR-USD exchange rate data. This dataset is the largest of our five. Levels 1 to 4 show the same consistent rejection explained above for small $\de$. Level 5 appears to be the transition scale. The null hypothesis cannot be rejected for levels 6 and above, corresponding to a timescale of longer than 48 minutes.
\begin{table}[h]

 {\footnotesize
 \begin{center} 
 \begin{tabular}{ |c || *{9}{c|}} 
 \hline  \multicolumn{10}{|c|}{{\bf EUR-USD}} \\ \hline 
 \hline {\bf levels ($l$)}& 0 & 1 & 2 & 3 & 4 & 5 & 6 & 7 & 8 \\ \hline 
{\bf \# SubX} & & 1092577 & 198585 & 34804 & 7981 & 1916 & 492 & 120 & 32 \\ \hline 
{\bf mean xing length} & &  &  &  &  &  &  &  &  \\  
{\bf level ($l-1$)} & & \raisebox{1.5ex}[0pt]{0.1 min} & \raisebox{1.5ex}[0pt]{0.4 min} & \raisebox{1.5ex}[0pt]{1.9 min} & \raisebox{1.5ex}[0pt]{11 min} & \raisebox{1.5ex}[0pt]{48 min} & \raisebox{1.5ex}[0pt]{3.3 h} & \raisebox{1.5ex}[0pt]{13 h} & \raisebox{1.5ex}[0pt]{53 h} \\ \hline \hline 
\raisebox{0ex}[12pt]{ $\mathbf{ \chi^2}$ {\bf (+2 df)} }  &  & {\bf 0.000} & {\bf 0.000} & {\bf 0.000} & {\bf 0.000} & {\bf 0.020} & 0.669 & 0.915 & $>0.05$ \\ \hline 
{\bf Twos Test} &  & {\bf 0.000} & {\bf 0.000} & {\bf 0.000} & {\bf 0.013} & {\bf 0.002} & 0.822 & 0.927 & 1.000 \\ \hline 
{\bf G Test} &  & {\bf 0.000} & {\bf 0.000} & {\bf 0.000} & {\bf 0.000} & {\bf 0.008} & 0.513 & 0.830 & 0.316 \\ \hline 
{\bf KS} &  & $\mathbf{<0.05}$ & $\mathbf{<0.05}$ & $\mathbf{<0.05}$ & $\mathbf{<0.05}$ & $\mathbf{<0.05}$ & $>0.05$ & $>0.05$ & $>0.05$ \\ \hline 
{\bf KLP98} &  & {\bf 0.000} & {\bf 0.000} & {\bf 0.000} & {\bf 0.000} & {\bf 0.015} & 0.306 & 0.725 & 0.890 \\ \hline \hline 
{\bf Joint Dist.} &  & {\bf 0.000} & {\bf 0.000} & {\bf 0.000} & {\bf 0.000} & 0.119 & 0.891 & 0.765 & 0.934 \\ \hline 
{\bf Autocorr} &  & $\mathbf{<0.05}$ & $\mathbf{<0.05}$ & $\mathbf{<0.05}$ & $\mathbf{<0.05}$ & $>0.05$ & $>0.05$ & $>0.05$ & $>0.05$ \\ \hline 
{\bf Runs} &  & {\bf 0.000} & {\bf 0.000} & {\bf 0.000} & 0.060 & 0.511 & 0.622 & 0.645 & 0.862 \\ \hline 
{\bf Larsen} &  & {\bf 0.000} & {\bf 0.000} & 0.434 & 0.089 & 0.377 & 0.602 & 0.284 & $>0.05$ \\ \hline 
{\bf DixOB} &  & {\bf 0.000} & {\bf 0.000} & {\bf 0.000} & 0.155 & 0.899 & 0.340 & 0.761 & 0.752 \\ \hline 
{\bf OBri85} &  & {\bf 0.000} & {\bf 0.000} & {\bf 0.000} & 0.205 & 0.251 & 0.353 & 0.314 & 0.820 \\ \hline \hline 
{\bf \# UD pairs}  & 944811 & 347117 & 64832 & 9308 & 2038 & 480 & 126 & 28 & 9 \\ \hline 
{\bf Runs UD} & {\bf 0.000} & {\bf 0.000} & {\bf 0.000} & {\bf 0.000} & 0.735 & 0.699 & 0.054 & 0.879 & 1.000 \\ \hline 
{\bf Larsen UD} & {\bf 0.000} & 0.052 & {\bf 0.046} & 0.376 & 0.437 & 0.250 & 0.266 & $\mathbf{<0.05}$ & $>0.05$ \\ \hline 
{\bf DixOB UD} & {\bf 0.009} & {\bf 0.000} & {\bf 0.000} & {\bf 0.000} & 0.748 & 0.270 & 0.863 & 0.088 & $>0.05$ \\ \hline 
{\bf OBri85 UD} & {\bf 0.000} & {\bf 0.000} & {\bf 0.000} & {\bf 0.001} & 0.619 & 0.144 & 0.578 & {\bf 0.023} & NA \\ \hline 
\end{tabular} 
 \end{center}
 } 
\caption{Crossing tree test results for 2003 EUR-USD exchange rate tick data. Distribution and independence tests clearly reject the null hypothesis at levels 1 to 5, most likely because the data does not support the continuous path assumption at those scales. The null hypothesis cannot be rejected at levels 6 and above.} \label{EURUSD.table}
\end{table}

\subsection{GBP-USD}
Table \ref{GBPUSD.table} shows results from applying the crossing tree tests to the GBP-USD exchange rate data. Levels 1 to 3 show the same consistent rejection explained above for small $\de$. Many tests reject at level 4, but notably the Twos test does not. This suggests the number of twos is not significant, and multiple crossings from interpolations between datapoints is not a dominant feature. Thus, rejection by the tests appears indicative of some other features in the data such that the null hypothesis is not satisfied at that scale. At higher levels, the autocorrelation test result at level 6 appears anomalous. Ignoring it, the null hypothesis cannot be rejected at levels 5 and above, corresponding to timescales larger than a cutoff between 11 and 46 minutes.
\begin{table}[h]

 {\footnotesize
 \begin{center} 
 \begin{tabular}{ |c || *{9}{c|}} 
 \hline  \multicolumn{10}{|c|}{{\bf GBP-USD}} \\ \hline 
 \hline {\bf levels ($l$) }& 0 & 1 & 2 & 3 & 4 & 5 & 6 & 7 & 8 \\ \hline 
{\bf \# SubX} & & 735365 & 167751 & 34371 & 8354 & 2053 & 504 & 130 & 31 \\ \hline 
{\bf mean xing length} & &  &  &  &  &  &  &  &  \\  
{\bf level ($l-1$)} & & \raisebox{1.5ex}[0pt]{0.2 min} & \raisebox{1.5ex}[0pt]{0.5 min} & \raisebox{1.5ex}[0pt]{2.3 min} & \raisebox{1.5ex}[0pt]{11.2 min} & \raisebox{1.5ex}[0pt]{46 min} & \raisebox{1.5ex}[0pt]{3.1 h} & \raisebox{1.5ex}[0pt]{12.7 h} & \raisebox{1.5ex}[0pt]{49 h} \\ \hline \hline 
\raisebox{0ex}[12pt]{ $\mathbf{ \chi^2}$ {\bf (+2 df)} }  &  & {\bf 0.000} & {\bf 0.000} & {\bf 0.000} & {\bf 0.000} & 0.848 & 0.140 & 0.712 & $>0.05$ \\ \hline 
{\bf Twos Test} &  & {\bf 0.000} & {\bf 0.000} & {\bf 0.000} & 0.784 & 0.377 & 0.624 & 0.539 & 1.000 \\ \hline 
{\bf G Test} &  & {\bf 0.000} & {\bf 0.000} & {\bf 0.000} & {\bf 0.009} & 0.859 & 0.112 & 0.386 & 0.727 \\ \hline 
{\bf KS} &  & $\mathbf{<0.05}$ & $\mathbf{<0.05}$ & $\mathbf{<0.05}$ & $>0.05$ & $>0.05$ & $>0.05$ & $>0.05$ & $>0.05$ \\ \hline 
{\bf KLP98} &  & {\bf 0.000} & {\bf 0.000} & {\bf 0.000} & {\bf 0.000} & 0.460 & 0.678 & 0.238 & 0.675 \\ \hline \hline 
{\bf Joint Dist.} &  & {\bf 0.000} & {\bf 0.000} & {\bf 0.000} & 0.265 & 0.670 & 0.184 & 0.967 & 0.522 \\ \hline 
{\bf Autocorr} &  & $\mathbf{<0.05}$ & $\mathbf{<0.05}$ & $\mathbf{<0.05}$ & $\mathbf{<0.05}$ & $>0.05$ & $\mathbf{<0.05}$ & $>0.05$ & $>0.05$ \\ \hline 
{\bf Runs} &  & {\bf 0.000} & {\bf 0.000} & {\bf 0.000} & 0.140 & 1.000 & 0.053 & 0.690 & 1.000 \\ \hline 
{\bf Larsen} &  & {\bf 0.000} & {\bf 0.000} & {\bf 0.049} & {\bf 0.000} & 0.834 & 0.856 & 0.185 & $>0.05$ \\ \hline 
{\bf DixOB} &  & {\bf 0.000} & {\bf 0.000} & {\bf 0.000} & {\bf 0.007} & 0.190 & 0.656 & 0.289 & 0.584 \\ \hline 
{\bf OBri85} &  & {\bf 0.000} & {\bf 0.000} & {\bf 0.000} & {\bf 0.024} & 0.388 & 0.112 & 0.817 & 0.495 \\ \hline \hline 
{\bf \# UD pairs}  & 486032 & 199187 & 49413 & 8831 & 2166 & 522 & 122 & 34 & 4 \\ \hline 
{\bf Runs UD} & {\bf 0.000} & {\bf 0.009} & {\bf 0.000} & {\bf 0.000} & 0.452 & 0.149 & {\bf 0.001} & 0.093 & 1.000 \\ \hline 
{\bf Larsen UD} & {\bf 0.000} & {\bf 0.000} & {\bf 0.005} & 0.297 & 0.459 & 0.721 & 0.063 & $\mathbf{<0.05}$ & $>0.05$ \\ \hline 
{\bf DixOB UD} & 1.000 & {\bf 0.000} & {\bf 0.000} & {\bf 0.000} & 0.861 & 0.344 & 0.989 & 0.182 & $>0.05$ \\ \hline 
{\bf OBri85 UD} & {\bf 0.000} & {\bf 0.000} & {\bf 0.000} & {\bf 0.004} & 0.820 & 0.767 & 0.126 & 0.082 & NA \\ \hline 
\end{tabular} 
 \end{center}
 } 
\caption{Crossing tree test results for 2003 GBP-USD exchange rate tick data. Distribution and independence tests clearly reject the null hypothesis at levels 1 to 3, most likely because the data does not support the continuous path assumption at those scales. On the other hand, the null hypothesis cannot be rejected at levels 5 and above.} \label{GBPUSD.table}
\end{table}

\subsection{JPY-USD}
Table \ref{JPYUSD.table} shows results from applying the crossing tree tests to the JPY-USD exchange rate data. Levels 1 to 4 show the same consistent rejection explained above for small $\de$. The null hypothesis cannot be rejected for levels 5 and above, so timescales larger than a cutoff between 12 and 52 minutes. At level 7 the Runs test shows a significant result that we dismiss, since it is an isolated result and the $p$-value is very close to 0.05 anyways.

\begin{table}[h]

 {\footnotesize 
 \begin{center} 
 \begin{tabular}{ |c || *{9}{c|}} 
 \hline  \multicolumn{10}{|c|}{{\bf JPY-USD}} \\ \hline 
 \hline {\bf levels ($l$)}& 0 & 1 & 2 & 3 & 4 & 5 & 6 & 7 & 8 \\ \hline 
{\bf \# SubX} & & 829941 & 171042 & 33165 & 7404 & 1810 & 435 & 107 & 24 \\ \hline 
{\bf mean xing length} & &  &  &  &  &  &  &  &  \\  
{\bf level ($l-1$)} & & \raisebox{1.5ex}[0pt]{0.1 min} & \raisebox{1.5ex}[0pt]{0.5 min} & \raisebox{1.5ex}[0pt]{2.2 min} & \raisebox{1.5ex}[0pt]{12 min} & \raisebox{1.5ex}[0pt]{52 min} & \raisebox{1.5ex}[0pt]{3.5 h} & \raisebox{1.5ex}[0pt]{14.7 h} & \raisebox{1.5ex}[0pt]{59 h} \\ \hline \hline 
\raisebox{0ex}[12pt]{ $\mathbf{ \chi^2}$ {\bf (+2 df)} }  &  & {\bf 0.000} & {\bf 0.000} & {\bf 0.000} & {\bf 0.000} & 0.122 & 0.689 & 0.951 & $>0.05$ \\ \hline 
{\bf Twos Test} &  & {\bf 0.000} & {\bf 0.000} & {\bf 0.000} & {\bf 0.000} & 0.359 & 0.701 & 0.699 & 0.839 \\ \hline 
{\bf G Test} &  & {\bf 0.000} & {\bf 0.000} & {\bf 0.000} & {\bf 0.000} & 0.188 & 0.672 & 0.857 & 0.885 \\ \hline 
{\bf KS} &  & $\mathbf{<0.05}$ & $\mathbf{<0.05}$ & $\mathbf{<0.05}$ & $\mathbf{<0.05}$ & $>0.05$ & $>0.05$ & $>0.05$ & $>0.05$ \\ \hline 
{\bf KLP98} &  & {\bf 0.000} & {\bf 0.000} & {\bf 0.000} & {\bf 0.000} & 0.115 & 0.825 & 0.962 & 0.770 \\ \hline \hline 
{\bf Joint Dist.} &  & {\bf 0.000} & {\bf 0.000} & {\bf 0.000} & {\bf 0.000} & 0.402 & 0.499 & 0.077 & 0.610 \\ \hline 
{\bf Autocorr} &  & $\mathbf{<0.05}$ & $\mathbf{<0.05}$ & $\mathbf{<0.05}$ & $\mathbf{<0.05}$ & $>0.05$ & $>0.05$ & $>0.05$ & $>0.05$ \\ \hline 
{\bf Runs} &  & {\bf 0.000} & {\bf 0.000} & {\bf 0.000} & 0.329 & 0.640 & 0.182 & {\bf 0.048} & 0.502 \\ \hline 
{\bf Larsen} &  & {\bf 0.000} & 0.437 & 0.229 & 0.379 & 0.826 & 0.117 & 0.675 & $>0.05$ \\ \hline 
{\bf DixOB} &  & {\bf 0.000} & {\bf 0.000} & {\bf 0.000} & 0.062 & 0.570 & 0.121 & 0.885 & 0.742 \\ \hline 
{\bf OBri85} &  & {\bf 0.000} & {\bf 0.000} & {\bf 0.000} & {\bf 0.000} & 0.621 & 0.694 & 0.299 & 0.653 \\ \hline \hline 
{\bf \# UD pairs}  & 651261 & 244039 & 52356 & 9173 & 1971 & 452 & 104 & 29 & 7 \\ \hline 
{\bf Runs UD} & {\bf 0.000} & {\bf 0.000} & {\bf 0.000} & {\bf 0.000} & 0.065 & 0.173 & 0.186 & 0.449 & 0.914 \\ \hline 
{\bf Larsen UD} & {\bf 0.000} & 0.052 & 0.216 & 0.059 & 0.592 & 0.328 & {\bf 0.038} & $>0.05$ & $>0.05$ \\ \hline 
{\bf DixOB UD} & {\bf 0.000} & {\bf 0.000} & {\bf 0.000} & {\bf 0.000} & 0.106 & 0.692 & 0.872 & 0.664 & $>0.05$ \\ \hline 
{\bf OBri85 UD} & {\bf 0.000} & {\bf 0.000} & {\bf 0.000} & {\bf 0.000} & 0.523 & 0.998 & 0.617 & {\bf 0.032} & NA \\ \hline 
\end{tabular} 
 \end{center}
 } 
\caption{Crossing tree test results for 2003 JPY-USD exchange rate tick data. Distribution and independence tests clearly reject the null hypothesis at levels 1 to 4, most likely because the data does not support the continuous path assumption at those scales. The null hypothesis cannot be rejected at levels 5 and above.} \label{JPYUSD.table}
\end{table}

\section{Conclusion}
\label{conc.sec}

Using a recently proposed characterisation of Brownian motion using the crossing tree, we have proposed a new test for continuous local martingales, equivalently a Brownian motion time-changed by a continuous chronometer. It is particularly well-suited to tick-by-tick data, which is not at regularly spaced time intervals. An advantage is the ability to test for continuous local martingales at a range of timescales, including an ability to test whether the data supports (in a statistical sense) the hypothesis of continuous paths.

An alternative test uses the sample quadratic variation (realised volatility), but suffers from issues (e.g., market microstructure noise) that give the sample quadratic variation poor statistical properties, especially at high frequencies. Further, formal statistical tests require the choice of an arbitrary parameter. As we show, this choice could have a large impact on the test results. Use of the crossing tree characterisation avoids these problems.

From simulation results for a range of diffusions, use of the crossing tree generally shows equal or higher discriminatory power than the quadratic variation approach for a similar number of observations. Although, it notably lacks power to detect drift. 

From testing high frequency foreign exchange rate data for five rates, we show at small timescales (typically below about 15 minutes) the data doesn't support the hypothesis of a continuous local martingale. At larger timescales, with the exception of the EURGBP exchange rate, the continuous local martingale assumption cannot be rejected. In the case of EURGBP, some larger timescales are rejected while others are not. This extends the earlier conclusions of \citet{Andersen2000a} who normalised daily foreign exchange rate returns by realised volatilities (formed by summing 30 minute return volatilities) and observed standard Normal values.

\appendix
\section{Simulation techniques} \label{simtech.sec}

We consider here the problem of simulating a stochastic process at the times $T^0_k$, $k = 0, 1, \ldots$.
That is, simulating the crossing points $X(T^0_k)$ and crossing lengths $W^0_k$.

\subsection{General Results for Diffusion processes}
\label{simtechgen.sec}

For diffusion processes we can use the {\em scale function} \cite[p. 278--290]{rny} to simulate the sequence $\{ X(T^0_k) \}_{k \geq 0}$. Let $W$ denote Brownian motion. Suppose $X$ is a process described by $dX(t) = A(X(t)) dt + B(X(t)) dW(t)$  satisfying the following conditions.
\begin{enumerate} 
\item The paths of X are continuous.
\item The state space $E$ has form $[l,r]$, $(l,r]$, $[l,r)$ or $(l,r)$ where $-\infty \leq l<r \leq \infty$.
\item The process is \emph{regular}. That is, for any $y$ in the interior of $E$, $l<x$, and $y<r$, $\Pb[T(y) < \infty | X(0)=x)>0$ where $T(y)$ is the hitting time for $y$.
\item Both $A$ and $B$ are locally bounded Borel functions and $B(\cdot)>0$.
\end{enumerate}
Fix $x_0$ in the interior of the range of $X$, then let the scale function $s$ be a solution of
\[
\frac d{dx} s(x) = \exp \left\{ -2 \int_{x_0}^x (A(u)/B^2(u)) du  \right\}.
\]
Take $x \in \de \Z$. Then, provided $[x - \de, x + \de]$ is in the interior of the range of $X$, we have 

\begin{align}
\label{hitprob.eq}
p_\delta(x)=\Pb( X(T^0_{k+1}) = x + \de \,|\, X(T^0_k) = x) = \frac{s(x)-s(x-\de)}{s(x+\de)-s(x-\de)}.
\end{align}

Simulating the crossing lengths $W^0_k$ is harder, though can be done in some cases.
See \citet{BJ08} for the case of Brownian motion.
However, using the {\em speed measure} we can get an expression for the {\em expected} crossing length.
Put
\[
m(dx) = 2 B^{-2}(x) \exp \left\{ 2 \int_{x_0}^x (A(u)/B^2(u)) du \right\} dx
\]
then, provided $[x - \de, x + \de]$ is in the interior of the range of $X$, we have \citep[p. 304]{rny},
\begin{eqnarray*}
\Ex( W^0_{k+1} \,|\, X(T^0_k) = x)
&=& \frac{s(x)-s(x-\de)}{s(x+\de)-s(x-\de)} \int_x^{x+\de} (s(x+\de) - s(y)) m(dy) \\
&&\quad + \frac{s(x+\de) - s(x)}{s(x+\de)-s(x-\de)} \int_{x-\de}^x (s(y)-s(x-\de)) m(dy).
\end{eqnarray*}

\subsection{Brownian motion with drift}
Suppose that $dX(t) = \al dt + dW(t)$ for $\al \neq 0$.
Put $x_0 = 0$ then we get $s(x) = -(2\al)^{-1} e^{-2\al x}$ and $m(dx) = 2 e^{2\al x} dx$, whence
\begin{eqnarray}
\Pb( X(T^0_{k+1}) = x + \de \,|\, X(T^0_k) = x) &=& \frac {e^{2\al\de} - 1} {e^{2\al\de} - e^{-2\al\de}}, \notag \\
\Ex( W^0_{k+1} \,|\, X(T^0_k) = x) &=& \frac {\de ( e^{2\al\de} - 1)} {\al (e^{2 \al\de} + 1)}. \label{bmdriftEx.eqn}
\end{eqnarray}

\subsection{Ornstein-Uhlenbeck process}\label{OUsim.sec}
Suppose that $dX(t) = -\al X(t) dt + \si dW(t)$ with $\al, \si > 0.$  Put $x_0 = 0$ then we get $s'(x) = e^{\al x^2/ \si^2}$ and $m(dx) = (2/\si^2) e^{-\al x^2/ \si^2} dx$, whence
\[
\Pb( X(T^0_{k+1}) = x + \de \,|\, X(T^0_k) = x)
= \frac {\int_{x-\de}^x e^{\al u^2/ \si^2} du} {\int_{x-\de}^{x+\de} e^{\al u^2/ \si^2} du}
= p(x) \mbox{ say}
\]
and
\begin{eqnarray*}
\Ex( W^0_{k+1} \,|\, X(T^0_k) = x)
&=& (2/\si^2) p(x) \int_x^{x+\de} \int_y^{x+\de} e^{\al(z^2-y^2)/\si^2} dz\,dy \\
&&\quad + (2/\si^2) (1 - p(x)) \int_{x-\de}^{x} \int_{x-\de}^y e^{\al(z^2-y^2)/\si^2} dz\,dy \\
&=& w(x) \mbox{ say}.
\end{eqnarray*}

With $X(0)\sim N(0,\si^2/(2\al))$, $\{X(t),\, t\ge0 \}$  is stationary. Thus, for $Y_k=X(T^0_k)$, $\{Y_k,\, k=0,1,\ldots\}$ is a positive recurrent random walk on $\de\Z$, defined by $\Pb( Y(t+1) = y + \de \,|\, Y(t) = y) = p(y)$ and $\Pb( Y(t+1) = y - \de \,|\, Y(t) = y) = 1 - p(y)$. Let $\pi$ be the stationary distribution of $Y$. Then, to simulate the crossings of the Ornstein-Uhlenbeck process, we take $Y_0$ as a value from the stationary distribution $\pi$ and then simulate according to the Markov chain described by $p(x)$.

\subsection{Feller's Square Root Diffusion}\label{Fellersim.sec}
Suppose that $dX(t) = \ka (\mu - X(t)) dt + \si \sqrt{X(t)} dW(t)$ for $\ka, \mu, \si > 0$, $2\ka\mu/\si^2 \geq 1$ and $X(0)$ has the stationary Gamma$(a,b)$ distribution with shape parameter $a=2\ka\mu/\si^2$ and scale parameter $b=\sigma^2/(2\ka)$.
For this range of parameter values the process is stationary with sample space $(0, \infty)$.

Put $x_0 = \de$ then we get $s'(x) = (x/\de)^{-2\ka\mu/\si^2} e^{(2\ka/\si^2)(x-\de)}$ and thus for $x = n\de$, $x \geq 2\de$, we have
\[
\Pb( X(T^0_{k+1}) = x + \de \,|\, X(T^0_k) = x)
= \frac {\int_{x-\de}^x u^{-2\ka\mu/\si^2} e^{(2\ka/\si^2)u} du}
{\int_{x-\de}^{x+\de} u^{-2\ka\mu/\si^2} e^{(2\ka/\si^2)u} du}
= p(x) \mbox{ say}.
\]
For $x = \de$ we have, since the process never visits 0, $\Pb( X(T^0_{k+1}) = 2\de \,|\, X(T^0_k) = \de) = 1$.

In practice, to simulate $\{X(T^0_k),\, k=0,1,\ldots\}$ we first obtain the \textit{first} hitting time $T^0_0$ and crossing point $Y_0=X(T_0^0)$. We do this by simulating the discrete approximate path $\{X_k,\, k=0,1,\ldots \}$, using the order 1.0 Milstein scheme \cite[p. 335]{kloeden92}, at regularly spaced points with interval $\Delta$  using
\begin{equation} \label{fellerpath.eqn}
X_{k+1}=X_k+\ka(\mu-X_k)\Delta+\si\sqrt{X_k}Z_k+\frac{\si^2(Z_n^2-\Delta)}{4}. 
\end{equation}
Here $X_0$ is a random Gamma$(a,b)$ value from the stationary distribution of $\{X(t) \}$, and $\{Z_1,Z_2,\ldots\}$ is an i.i.d. sequence of $N(0,\Delta)$ values. With $Y_0$ and $p(x)$, the needed crossings can be simulated directly using the Markov chain. Although $X(T^0_0)$ is the first hit of the lattice, we do not use it as a crossing.

\subsection{Other Processes}
For other processes, where a method to simulate the crossings directly may be unknown, it is still possible to simulate $X(T^0_k)$ and $T^0_k$, $k = 0, 1, \ldots$ if sufficiently long sample paths can be simulated. In effect, one simulates a sample path at regularly spaced times, sufficiently close to minimize interpolation errors, and then determines where the crossings occur directly. The Matlab program `gethits.m' does this. cite{jonesxx}

\section{Determination of $\de$} \label{detdelta.sec}

In this section we discuss how to determine the scale size $\de$ to simulate a process with, on average, a specified number $N$ of level 0 crossings in a specified time interval, typically $(T_0^0,T_0^0+t_0)$, where $t_0$ is a specified constant. In many cases $T_0^0=0$ and so the interval is just $(0,t_0)$.  We make this assumption here unless otherwise stated.

\subsection{Brownian motion}
A consequence of the strong Markov, self-similarity and stationary increment properties of Brownian motion is that $\Ex[W^0_k]=\de^2 $ for all $k$ \citep{BJ08}. On average we want intervals of length $t_0/N$ and so $\de=\sqrt{t_0/N} $.

\subsection{Brownian motion with drift}
As shown in (\ref{bmdriftEx.eqn}), the crossing lengths are constant. That is, they do not depend on \textit{where} the process is. It is sufficient to solve 
\[ \frac {\de ( e^{2\al\de} - 1)} {\al (e^{2 \al\de} + 1)}= \frac{t_0}{N}\] 
for $\delta$.

\subsection{Ornstein-Uhlenbeck process}
For a given $\de>0$, let $\pi$ be the equilibrium distribution of the Markov chain $\{Y_k,k=1,\ldots\}$ in Section \ref{OUsim.sec} for the Ornstein-Uhlenbeck process with
\[dX(t) = -\al X(t) dt + \si dW(t),\, \al,\, \si > 0\] 
and
\[p(x)=\Pb( X(T^0_{k+1}) = x + \de \,|\, X(T^0_k) = x)
= \frac {\int_{x-\de}^x e^{\al u^2/ \si^2} du} {\int_{x-\de}^{x+\de} e^{\al u^2/ \si^2} du}.\]
Then
\begin{equation}\label{OUEW.eqn}
\Ex_\pi (W^0_{k}) = \sum_{n=-\infty}^\infty \pi(n\de) w(n\de) 
\end{equation}
and we need to find $\de$ such that $\Ex_\pi (W^0_{k})=t_0/N$, essentially by enlightened trial-and-error to obtain sufficient accuracy and precision.
Here, by $\Ex_\pi$ we indicate expectation conditioned on $Y_1$ having initial distribution $\pi$ on $\de\Z$, so that $\{Y_k,\, k=0,1,\ldots\}$ is stationary.  For the particular goal of  $\Ex_\pi (W^0_{k})=0.004$, $\de$ was found for two combinations of parameters. For $\al=8$, $\si=1$, $\de=0.063015$ was found to give an absolute relative error in  $\Ex_\pi (W^0_{k})$ of less than 1.05 in 10,000, or about 1 point in 10,000.  Figure \ref{OUpassage.fig} illustrates the empirical distributions of $Y_1$, $Y_2$, $Y_{20}$, $Y_{100}$, and  $Y_{1000}$ for these parameters, using 100,000 sample paths are virtually indistinguishable from the stationary distribution, justifying use of the stationary distribution.  For $\al=10$, $\si=1$, $\de=0.062945$ was found to give an absolute relative error in  $\Ex_\pi (W^0_{k})$ of less than 0.2 in 10,000.

\begin{figure}[ht]
\resizebox{3in}{!}{\includegraphics{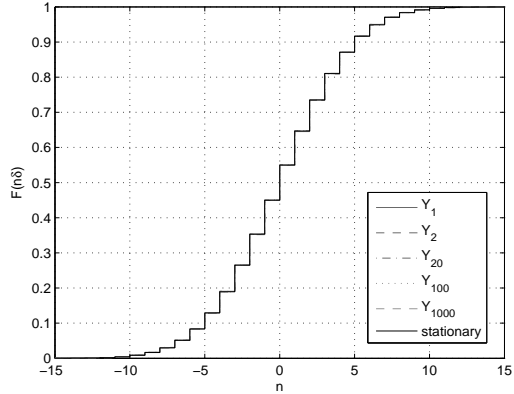}}
\caption{Distribution of empirical passage times for the Ornstein-Uhlenbeck process for $\de=0.063015$, $\al=8$, $\si=1$ using 100,000 Markov chains. They appear indistinguishable from the stationary distribution. \label{OUpassage.fig}}
\end{figure}

\subsection{Feller's  Square Root Diffusion}
Suppose that $dX(t) = \ka (\mu - X(t)) dt + \si \sqrt{X(t)} dW(t)$ for $\ka, \mu, \si > 0$, $2\ka\mu/\si^2 \geq 1$ and $X(0)$ has the stationary Gamma$(a,b)$ distribution with shape parameter $a=2\ka\mu/\si^2$ and scale parameter $b=\sigma^2/(2\ka)$. Although this is also a stationary diffusion, determination of $\de$ must proceed quite differently from the use of (\ref{OUEW.eqn}) above. We can again think in terms of the Markov chain $Y_k=X(T^0_k)$, $\{Y_k,\, k=0,1,\ldots\}$. We can obtain an integral formula for $\Ex( W^0_{k+1} \,|\, X(T^0_k) = x)$ for $x = k\de$, $k \ge 2$ but not for $x = \de$. It is possible to estimate this latter value by simulation to approximate $w$. Since we have all the transition probabilities for the Markov chain, calculation of the stationary distribution $\pi$ is straightforward and we can find
\[ \Ex_\pi (W^0_{\infty}) = \sum_{n=-\infty}^\infty \pi(n\de) w(n\de).  \]
The key problem is the Markov chain is not stationary, even though the underlying process is stationary, and $\Ex_\pi (W^0_{\infty})$ isn't terribly helpful. A reason for this somewhat surprising result is easy to see when parameter values are used. For example, consider the case of $\de=0.0220335$, $\ka=6$, and $\mu=0.2$. Since we know the stationary Gamma distribution we have $\Pr(Y_0=\de) \leq \Pr(X(0)\le \de) = 0.01144$. On the other hand, for the stationary distribution $\Pr(\pi=\de)= 0.0029$ to four digits. Then the following result applies.
\begin{lem}
Assume $Y_0$ does not have the stationary distribution. Then  $Y_k$ does not have the stationary distribution for all $k\ge0$.
\end{lem}
\begin{proof}
We'll show that $Y_1$ cannot have the stationary distribution, but the same argument easily extends to $Y_{k-1}$ and $Y_k$. Assume $Y_1 \equaldist \pi$. Let P be the transition matrix for the Markov chain. Then $\{Y_k/\delta,\,k=0,1,\ldots\}$,with state space $\{1,2,\ldots\}$, is the embedded Markov chain for some continuous-time process $\{V(t),\,t\geq0\}$ with exponential holding times, different from $\{X(t),\,t\geq0\}$, which need not have exponential holding times. Since the transitions of the Markov chain are always $\pm 1$, $\{V(t),\,t\geq0\}$ is a birth and death process. Mean reversion of $\{X(t),\,t\geq0\}$ implies $\{Y_k/\delta,\,k=0,1,\ldots\}$ is positive recurrent, and so ergodic. This means $\{Y_k/\delta,\,k=0,1,\ldots\}$, and so $\{Y_k,\,k=0,1,\ldots\}$, is time reversible. Let $Q$ be the transition matrix for the time reversed Markov chain. Since $\pi$ is the stationary distribution of the time reversed chain too, 
\[Y_0 \equaldist Q Y_1 \equaldist Q \pi \equaldist \pi.\]
But this violates the assumption that $Y_0$ does not have the stationary distribution. Hence, $Y_1$ cannot have the stationary distribution.
\end{proof}

(Note that a possibly simpler argument would say $Y_1 \equaldist \pi$ and so $\pi \equaldist Y_1 \equaldist P Y_0$ while $P \pi \equaldist \pi$. Combining these two equations gives $P Y_0=P \pi$ and so $Y_1=\pi$. However this requires that $P$ be invertible, which is difficult to establish.)

Figure \ref{passagedist.feller.1e5.fig} illustrates the empirical distributions of $Y_1$, $Y_2$, $Y_{20}$, $Y_{100}$, and  $Y_{1000}$ using 100,000 simulated Markov chains. The distributions are clearly different from each other, and from the stationary distribution, even for the 1000th passage time. This suggests calculations based on the stationary distribution of the crossing times are of limited use for Feller's square root process.
 
\begin{figure}[ht]
\resizebox{3in}{!}{\includegraphics{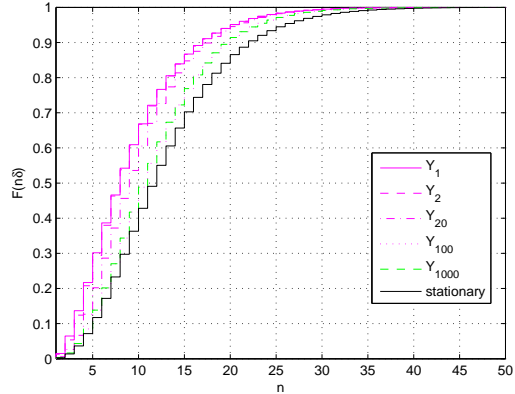}}
\caption{Distribution of empirical passage times for Feller's square root process for $\delta=0.0220335$, $\mu=0.2$, $\ka=6$, and $\si=1$ using 100,000 simulated Markov chains. The distributions appear quite different, indicating the Markov chain is not stationary. \label{passagedist.feller.1e5.fig}}
\end{figure}

Since the Markov chain at passage times is not stationary it was decided to simulate approximate paths of sufficient length using (\ref{fellerpath.eqn}) and trial-and-error to determine $\de$ as needed. For example, with $\ka=6$, $\si=1$, to put 1250 crossings in the interval $(T_0^0,T_0^0+5)$ (dropping the first passage time), $\de=0.027100$ was obtained using a time step $\Delta=10^{-5}$ and 5,000 sample paths. However, other values of $\Delta$ yielded different values for $\de$, which were significant considering the precision obtained. In effect, the goal is to estimate $\delta$ as $\Delta \to 0$, although simulation time quickly becomes impractical as $\Delta$ gets smaller. Regression and extrapolation were used to improve the estimate. Let $m=-\log_{10}(\Delta)$. Then the simulations provides estimates $\hat{\delta}_m$ using $\Delta=10^{-m}$. Since $\log_{10}(\hat{\delta}_{m+1}-\hat{\delta}_m)$ appears quite linear, linear least-squares regression was used to estimate $a$ and $b$ in the recursive form $\hat{\delta}_{m+1}=\hat{\delta}_{m}+10^{am+b}, \, m>3$ where $\hat{\delta}_{3}$ is known. Figure \ref{Feller_regression.fig} shows simulation results with $\hat{\delta}_{m}$  for $\ka=6$, $\si=1$ to obtain either 1250 crossings in an interval of length 5 or 5000 crossings in an interval of length 20. The results are quite similar. This is not surprising.  As the number of crossings increases, the distributions of the crossings $Y_k$ get closer to the stationary distribution and the non-stationary behaviour of $Y_k$ for $k$ small is less dominant. The fitted line provides an improved estimate $\hat{\delta}_\infty=0.028163$ for 1250 crossings.

\begin{figure}[ht]
\resizebox{3in}{!}{\includegraphics{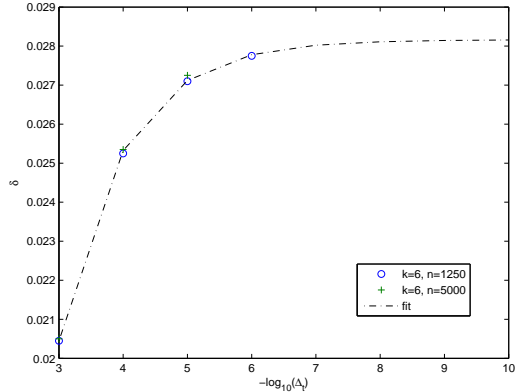}}
\caption{Determination of $\de$ for the Feller square root process with $\ka=6$, $\si=1$. The results for 1250 and 5000 crossings are similar. The fitted curve is for 1250 crossings, and provides an improved estimate for $\de$. \label{Feller_regression.fig}}
\end{figure}

\subsection{Other Processes}
For other processes, determining $\de$ to achieve $N$ crossings on average in the interval $(0,t_0)$ can be accomplished with a trial and error approach. Simulate a large number of independent sample paths (say $M$) with a time resolution sufficiently fine to minimize interpolation errors. Choose a value $\de^*$ as a possible value for $\de$. For sample path $i=1,\ldots,M$ determine the crossings and find the time $T^0_{N,i}$ of the $N$-th crossing time. If the sample mean $\hat{T}^0_N=\sum_{i=1}^M T^0_{N,i}/M$ is sufficiently close to $t_0$ (e.g., $t_0$ is within the confidence interval for $\hat{T}^0_N$) , then $\de^*$ can serve as the desired $\de$. If not, adjust the candidate value $\de^*$ larger or smaller and try again. This can be combined with other root-finding or regression techniques to find $\de$ with fewer iterations.

\bibliographystyle{agsm}
\bibliography{BMstat4}

\end{document}